\documentclass[preprint,12pt]{elsarticle}
\usepackage{amssymb}
\usepackage{amsthm}
\usepackage{amsmath,amssymb,graphicx}
\usepackage[titletoc]{appendix}
\usepackage{booktabs}
\usepackage{cases}
\usepackage{color}
\usepackage{empheq}
\usepackage{emptypage}
\usepackage{eqparbox}
\usepackage{enumerate}
\usepackage{fancyhdr}
\usepackage{graphicx}
\usepackage{listings}
\usepackage{mathrsfs}
\usepackage{multirow}
\usepackage{subfigure}
\usepackage{indentfirst}
\usepackage{verbatim}
\usepackage{tabstackengine}
\usepackage{witharrows}
\usepackage{url}
\usepackage{helvet}
\usepackage{titlesec}
\allowdisplaybreaks

\titleformat{\section}
  {\normalfont\fontsize{16}{19}\sffamily\bf}
  {\thesection}
  {1em}
  {}

\titleformat{\subsection}
  {\normalfont\fontsize{12}{17}\sffamily\bf\slshape}
  {\thesubsection}
  {1em}
  {}

  \titleformat{\subsubsection}
  {\normalfont\fontsize{11}{15}\sffamily\bf\slshape}
  {\thesubsubsection}
  {1em}
  {}

\usepackage{epsfig}
\newtheorem{theorem}{Theorem}[section]

\newtheorem{lemma}{Lemma}[section]
\newtheorem{corollary}{Corollary}[section]
\newtheorem{defn}{Definition}[section]
\newtheorem{remark}{Remark}[section] 

\numberwithin{equation}{section}
\theoremstyle{definition}

\usepackage{appendix}

\makeatletter

\newsavebox{\@brx}
\newcommand{\llangle}[1][]{\savebox{\@brx}{\(\m@th{#1\langle}\)}%
  \mathopen{\copy\@brx\kern-0.5\wd\@brx\usebox{\@brx}}}
\newcommand{\rrangle}[1][]{\savebox{\@brx}{\(\m@th{#1\rangle}\)}%
  \mathclose{\copy\@brx\kern-0.5\wd\@brx\usebox{\@brx}}}
\makeatother

\allowdisplaybreaks

\journal{XXX}

\begin{document}
\begin{frontmatter}

  \title{Port-Hamiltonian formulations of the incompressible Euler equations with a free surface}
  
  \author[add1,add2]{Xiaoyu Cheng}
  \ead{cxy1014@mail.ustc.edu.cn}
  \author[add2]{J. J. W. Van der Vegt}
  \ead{j.j.w.vandervegt@utwente.nl}
  \author[add1]{Yan Xu\corref{cor1}}
  \ead{yxu@ustc.edu.cn}
  \author[add3,add4]{H. J. Zwart}
  \ead{h.j.zwart@utwente.nl}

  \cortext[cor1]{Corresponding author}
  \address[add1]{School of Mathematics, University of Science and Technology of China, Hefei, Anhui, 230026, China}
  \address[add2]{Department of Applied Mathematics, Mathematics of Computational Science Group, University of Twente, P.O. Box 217, 7500 AE, Enschede, The Netherlands}
  \address[add3]{Department of Applied Mathematics, Mathematics of Systems Theory Group, University of Twente, P.O. Box 217, 7500 AE, Enschede, The Netherlands}
  \address[add4]{Department of Mechanical Engineering, Dynamics and Control Group, Eindhoven University of Technology, P.O. Box 513, 5600 MB, Eindhoven, The Netherlands}

  \begin{abstract}
    In this paper, we present port-Hamiltonian formulations of the incompressible Euler equations with a free surface governed by surface tension and gravity forces, modelling e.g. capillary 
    and gravity waves and the evolution of droplets in air. Three sets of variables are considered, namely $(v,\Sigma)$, $(\eta,\phi_{\partial},\Sigma)$ and $(\omega,\phi_{\partial},\Sigma)$, 
    with $v$ the velocity, $\eta$ the solenoidal velocity, $\phi_{\partial}$ a potential, $\omega$ the vorticity, and $\Sigma$ the free surface, resulting in the incompressible Euler equations in primitive variables 
    and the vorticity equation. First, the Hamiltonian formulation for the incompressible Euler equations in a domain with a free surface combined with a fixed boundary surface with a homogeneous boundary condition will be derived in the proper 
    Sobolev spaces of differential forms. Next, these results will be extended to port-Hamiltonian formulations allowing inhomogeneous boundary conditions and a non-zero energy flow through the boundaries.
    Our main results are the construction and proof of Dirac structures in suitable Sobolev spaces of differential forms for each variable set, which provides the core of any port-Hamiltonian formulation. Finally, it is proven that the state dependent 
    Dirac structures are related to Poisson brackets that are linear, skew-symmetric and satisfy the Jacobi identity. 
  \end{abstract}
  
  \begin{keyword}
  Port-Hamiltonian formulation\sep Dirac structure\sep Poisson bracket\sep Incompressible Euler equations\sep Vorticity equation\sep Free surface problems
  \end{keyword}
  
  \end{frontmatter}

\section{Introduction}

The evolution of many infinite dimensional systems can be described using a Hamiltonian formulation consisting of a Hamiltonian functional,
representing the total energy of the system, and a Poisson bracket defining its dynamics. Examples are inviscid fluid mechanics, ideal magnetohydrodynamics,
classical electromagnetism and the Schr$\ddot{\mathrm{o}}$dinger equation, see e.g. \cite{abraham2008foundations,abraham2012manifolds,arnold2008topological,MR2884939,marsden2013introduction,morrison1998hamiltonian}.

In this paper, we will consider Hamiltonian formulations of the incompressible Euler equations with a free surface governed by surface tension and gravity forces. These equations model a large variety of 
fluid mechanical problems, e.g. capillary and gravity waves and the behaviour of droplets in air \cite{wehausen1960surface}.
The incompressible Euler equations with a free surface in a time-dependent domain $\Omega(t)\subset\mathbb{R}^{d}$ are given by \cite{batchelor1989fluid,wehausen1960surface}                        
\begin{subequations}\label{incomeuler}
  \begin{align}
    \frac{\partial\vec{v}}{\partial t}+(\vec{v}\cdot\nabla)\vec{v}&=-\nabla\big(\frac{p}{\rho}+\Phi\big)\quad&&\text{in }\Omega,\\
  \nabla\cdot\vec{v}&=0\quad&&\text{in }\Omega,\\
  \frac{\partial\Sigma}{\partial t}&=\langle\vec{v},\vec{n}\rangle\quad&&\text{at }\Sigma,\\
  p-\bar{p}&=\tau k\quad&&\text{at }\Sigma,\\
  \langle\vec{v},\vec{n}\rangle&=g\quad&&\text{at }\Gamma.\label{intro4}
  \end{align}
\end{subequations}
Here $\vec{v}$ denotes the fluid velocity, $p$ the pressure with $\bar{p}$ a constant atmospheric pressure, $\Phi$ the gravity potential, $\rho$ the fluid density, $\tau$ the surface tension, $t$ time and $\nabla\in\mathbb{R}^{d}$ the nabla operator.
The domain $\Omega(t)$ has a boundary $\partial\Omega$ with external normal vector $\vec{n}$ and mean curvature $k$, and consists of a nonoverlapping free surface $\Sigma$ and a fixed boundary $\Gamma$ where a normal velocity $g$ is specified.

The Hamiltonian formulation of the incompressible Euler equations can be expressed using a Hamiltonian functional and a Lie-Poisson bracket \cite{arnold1966geometrie,beris1990poisson,MR719058}. The Hamiltonian formulation of the incompressible Euler equations 
has a very rich mathematical structure and has been extensively researched, see e.g. \cite{arnold2008topological,camassa2014variational,MR2335765,morrison1998hamiltonian,MR672198}.

The Hamiltonian formulation of the incompressible Euler equations with a free surface governed by surface tension, which describes the evolution of liquid droplets in an inviscid flow, and is obtained by setting $\Sigma(t)=\partial\Omega(t)$, $\Phi=0$ in (\ref{incomeuler}), is presented in \cite{MR838352}, see also \cite{MR1040394}. The evolution of any functional $\mathcal{F}=\mathcal{F}(\vec{v},\Sigma)$ by the flow governed by the
incompressible Euler equations with a free surface with $\Sigma=\partial\Omega$ can be expressed as \cite{MR838352}
\begin{equation}\label{intro1}
  \dot{\mathcal{F}}(\vec{v},\Sigma)=\{\mathcal{F},H\}(\vec{v},\Sigma),
\end{equation}
where the dot represents differentiation with respect to time, $H$ is the Hamiltonian 
\begin{equation}\label{intro2}
  H(\vec{v},\Sigma)=\frac{1}{2}\int_{\Omega(t)}|\vec{v}|^{2}d^{3}x+\tau\int_{\Sigma(t)}ds,
\end{equation}
and the Poisson bracket $\{\cdot,\cdot\}(\vec{v},\Sigma)$ is defined as
\begin{equation}\label{intro3}
  \begin{aligned}
    \{\mathcal{F},\mathcal{G}\}(\vec{v},\Sigma)=\int_{\Omega(t)}\big\langle\nabla\times\vec{v},\frac{\delta\mathcal{F}}{\delta\vec{v}}\times\frac{\delta\mathcal{G}}{\delta\vec{v}}\big\rangle d^{3}x+\int_{\Sigma(t)}\big(\frac{\delta\mathcal{F}}{\delta\Sigma}\frac{\delta\mathcal{G}}{\delta\phi}-\frac{\delta\mathcal{G}}{\delta\Sigma}\frac{\delta\mathcal{F}}{\delta\phi}\big)ds,
  \end{aligned}
\end{equation}
where $\frac{\delta\mathcal{F}}{\delta\vec{v}},\frac{\delta\mathcal{F}}{\delta\Sigma},\frac{\delta\mathcal{F}}{\delta\phi}$ denote, respectively, the functional derivatives with respect to the velocity, free surface and surface potential.

The Hamiltonian formulation (\ref{intro1})--(\ref{intro3}) describes a closed system without energy exchange with the external environment, namely $H$ in (\ref{intro2}) is constant. In the general case, with $\partial\Omega=\Sigma\cup\Gamma$, with $\mathrm{meas}(\Gamma)>0$ and $g\neq 0$ in (\ref{intro4}),
the incompressible Euler equations with a free surface exchange energy with the external environment. In order to describe the Hamiltonian structure of (\ref{incomeuler}) we need to use the more general concept of port-Hamiltonian systems. The main objective of this article is to formulate the Euler equations
with a free surface (\ref{incomeuler}) as a port-Hamiltonian system with special emphasis on formulating the port-Hamiltonian system in the proper Sobolev spaces of differential forms. This will significantly extend the number of applications, modelled by (\ref{incomeuler}), that can be described as a Hamiltonian system.

A port-Hamiltonian system describes a non-conservative open dynamical system that interacts with the environment through energy flow over boundary ports. Currently, port-Hamiltonian systems are a very active field of research.
We refer to \cite{duindam2009modeling,MR2952349,rashad2020twenty,MR3076035,van2014port,MR4180029} for an overview of the theory and applications of port-Hamiltonian systems.

From a mathematical point of view, the distributed-parameter port-Hamiltonian system is 
constructed based on the existing Hamiltonian framework by using the notion of a Dirac structure \citep{van2002hamiltonian,zwart2009distributed},
which is a geometric structure generalizing both symplectic and Poisson structures, and was originally introduced in \citep{courant1990dirac,duindam2009modeling}.
In \citep{van2002hamiltonian,zwart2009distributed}, a Dirac structure $D$ is defined on $\mathcal{F}\times\mathcal{E}$, with  $\mathcal{F}$ and $\mathcal{E}$, respectively, the dual flow and effort spaces.             
The spaces $\mathcal{F}$ and $\mathcal{E}$ are formalized using Sobolev spaces of differential forms. The proper choice of these Sobolev spaces, their duality relations, and the formulation of a bilinear form and Dirac structure for the Euler equations with a free surface will be the main topic of this article.

Port-Hamiltonian formulations of the compressible Euler equations, both in terms of the velocity-pressure and vorticity-streamfunction formulations are proposed in \citep{polner2014hamiltonian}.
Recently, in \citep{rashad2021port1,rashad2021port2}, the authors discuss the port-Hamiltonian formulation of the incompressible Euler equations in terms of the velocity and pressure, but without a free surface.

In this article, we will consider the port-Hamiltonian formulation for the incompressible Euler equations in a domain with a free surface combined with a fixed boundary surface where an inhomogeneous boundary condition is imposed. Three sets of variables will be considered, namely $(v,\Sigma)$, $(\eta,\phi_{\partial},\Sigma)$ and $(\omega,\phi_{\partial},\Sigma)$, with $v$ the velocity, $\eta$ the solenoidal velocity, 
$\phi_{\partial}$ a surface potential, $\omega$ the vorticity and $\Sigma$ the free surface. Using these variables several important formulations of the incompressible Euler equations with a free surface can be obtained such as the primitive variable formulation (\ref{incomeuler}) and the vorticity equation,
with their suitability depending on the type of application. For each of these formulations of (\ref{incomeuler}) we will first derive the Hamiltonian formulation with homogeneous boundary conditions at $\Gamma$ ($g=0$ in (\ref{intro4})) in the proper setting of Sobolev spaces of differential forms. 
Next, we will define a bilinear form $\llangle\cdot,\cdot\rrangle$ and structures $D_{i}$ ($i=1,2,3$) and prove that these structures are Dirac structures, namely $D_{i}=D_{i}^{\bot}$, 
with $\bot$ denoting the orthogonal complement with respect to $\llangle\cdot,\cdot\rrangle$. 
These results constitute the core of this paper. For linear problems, the Dirac structure can be directly related to a Poisson bracket. But for nonlinear problems, it generally only relates to a skew-symmetric pseudo-Poisson bracket \cite{van2002hamiltonian} since the Jacobi identity needs to be proven separately
and the Dirac structure only ensures power conservation.
We will prove that the various state dependent Dirac structures $D_{i}$ ($i=1,2,3$) presented in this article also satisfy the Jacobi identity and that the resulting formulations constitute a port-Hamiltonian system, hence allow a non-zero energy exchange with the external environment.

The outline of this article is as follows.
First, we will give in Section 2 a summary of some important results from exterior algebra, Sobolev spaces, trace operators, and Hodge decompositions that will be extensively used in this paper. 
In Section 3 we will introduce the incompressible Euler equations with a free surface in terms of the velocity and pressure and discuss the decomposition of the velocity $v$ into three parts, namely an exact form $\phi$ (potential), co-exact form $\beta$ and a harmonic form $\alpha$.
Next, in Section 4 we will discuss the generalized Hamiltonian formulation of incompressible Euler equations with a free surface in three sets of variables, namely $(v,\Sigma)$, $(\eta,\phi_{\partial},\Sigma)$ and $(\omega,\phi_{\partial},\Sigma)$ with the homogeneous boundary condition $g=0$ at $\Gamma$.
Finally, in Section 5 we will present the Dirac structures and port-Hamiltonian formulations of the incompressible Euler equations in a domain with both a free boundary surface and a fixed boundary surface with the inhomogeneous boundary condition $g\neq 0$ at $\Gamma$ with respect to the $(v,\Sigma)$, $(\eta,\phi_{\partial},\Sigma)$ and $(\omega,\phi_{\partial},\Sigma)$ variables.

\section{Preliminaries}
Since the language of differential forms is crucial to express the geometrical properties of port-Hamiltonian systems, we will introduce in this section some of the geometrical concepts and tools that will be used later on and set the notation used in this article. For more details we refer to \citep{abraham2008foundations,abraham2012manifolds,MR2884939,marsden2013introduction,schwarz2006hodge}.

\subsection{Differential forms, function spaces, and related results}
Let $\Omega$ be an open, bounded and connected oriented $n$-dimensional manifold with an ($n-$1)-dimensional Lipschitz boundary $\partial\Omega$. 

\subsubsection{Riemannian metric and index operators}
Assume that there is a Riemannian metric $\llangle\cdot,\cdot\rrangle:T\Omega\times T\Omega\to\mathbb{R}$, where $T\Omega=\bigcup_{p\in\Omega} T_{p}\Omega$, with $T_{p}\Omega$ the $n$-dimensional tangent space at the point $p\in\Omega$. At each point $p\in\Omega$, $\llangle\cdot,\cdot\rrangle$ is symmetric, bilinear and positive definite.
For any vector $X\in T_{p}\Omega$, the index lowering operator $\flat:T_{p}\Omega\to T^{\ast}_{p}\Omega$, with $T_{p}^{\ast}\Omega$ the dual space of $T_{p}\Omega$, is defined as \citep{abraham2012manifolds,abraham2008foundations}: $X^{\flat}(\cdot):=\llangle X,\cdot\rrangle$. Dually, the index raising operator $\sharp:T_{p}^{\ast}\Omega\to T_{p}\Omega$ is given by $(X^{\flat})^{\sharp}=X.$

\subsubsection{Exterior product and inner product}

Let $\Lambda^{k}(\Omega)$ be the space that contains all smooth differential $k$-forms on $\Omega$.
Denote with $\wedge:\Lambda^{k}(\Omega)\times\Lambda^{j}(\Omega)\to\Lambda^{k+j}(\Omega)$ the exterior product of differential forms \citep{abraham2008foundations}, which satisfies the relation
\begin{equation}\label{wedge}
  \lambda\wedge\mu=(-1)^{kj}\mu\wedge\lambda,
\end{equation}
where $\lambda\in\Lambda^{k}(\Omega)$, $\mu\in\Lambda^{j}(\Omega)$. The space $\Lambda^{k}(\Omega)$ is endowed with the inner product \citep{arnold2006finite}
\begin{equation}
 \langle\lambda,\mu\rangle_{\Lambda^{k}}:=\sum_{\sigma\in S(k,n)}\lambda(E_{\sigma(1)},\ldots,E_{\sigma(k)})\mu(E_{\sigma(1)},\ldots,E_{\sigma(k)}),
\end{equation}
with $(E_{1},\ldots,E_{n})$ an orthonormal frame on $\Omega$ with respect to the metric 
$\llangle\cdot,\cdot\rrangle$, and $S(k,n)$, with $1\leq k\leq n$, the set of all permutations of the numbers $\{1,2,\ldots,n\}$ such that $\sigma(1)<\ldots<\sigma(k)$.

\subsubsection{Hodge star operator, exterior derivative and coderivative operators, interior product}
Since $\dim\Lambda^{k}(\Omega)=\dim\Lambda^{n-k}(\Omega)$, there exists an
isometry taking $k$-forms into $(n-k)$-forms, which is called the Hodge star operator.
Let $v_{\Omega}\in\Lambda^{n}(\Omega)$ be the volume form of $\Omega$. Then the Hodge operator $\ast:\Lambda^{k}(\Omega)\to\Lambda^{n-k}(\Omega)$ is defined by the relation \citep{schwarz2006hodge}
\begin{equation}
  \lambda\wedge\ast\mu=\langle\lambda,\mu\rangle_{\Lambda^{k}}v_{\Omega},\quad\forall\lambda,\mu\in\Lambda^{k}(\Omega).
\end{equation}
We have the relation
 \begin{equation}\label{eq42}
  \ast\ast\lambda=(-1)^{k(n-k)}\lambda,\quad\lambda\in\Lambda^{k}(\Omega).
 \end{equation}
 
For all $k$-forms, we denote by $d:\Lambda^{k}(\Omega)\to\Lambda^{k+1}(\Omega)$ the (weak) exterior derivative operator, taking differential $k$-forms on the domain $\Omega$ to differential $(k+1)$-forms.
Conversely, the coderivative operator $\delta:\Lambda^{k}(\Omega)\to\Lambda^{k-1}(\Omega)$ is defined as \citep{arnold2006finite}
\begin{equation}\label{rela1}
  \delta:=(-1)^{n(k+1)+1}\ast d\ast.
\end{equation}

Another operator, taking $k$-forms into $(k-1)$-forms, is called interior product, which is defined as:  
given $X\in T_{p}\Omega$ and $\alpha\in\Lambda^{k}(\Omega)$, then $i_{X}\alpha$ is the $(k-1)$-form given by
\begin{equation*}  
  i_{X}\alpha(X_{1},\ldots,X_{k-1}):=\alpha(X,X_{1},\ldots,X_{k-1}),
\end{equation*}
where $X_{1},\dots,X_{k-1}$ are any vectors that belong to $T_{p}\Omega$.
We have the following relation between the Hodge star operator and the interior product \citep{abraham2012manifolds}
\begin{equation}\label{rela4}
  i_{X}\alpha=\ast(X^{\flat}\wedge\ast\alpha).
\end{equation}

\subsubsection{Sobolev spaces of differential forms}
 The spaces $L^{2}\Lambda^{k}(\Omega)$ and $H^{s}\Lambda^{k}(\Omega)$ with $k\in\mathbb{N}\cup\{0\},s\geq 0$ are the spaces of differential $k$-forms whose coefficient functions belong, respectively, to the Sobolev spaces $L^{2}(\Omega)$ and $H^{s}(\Omega)$, with $H^{0}(\Omega)=L^{2}(\Omega)$, \citep{arnold2006finite,arnold2010finite}. We define the corresponding $L^{2}$-inner product and norm as 
\begin{equation}\label{inner product}
  \langle\lambda,\mu\rangle_{L^{2}\Lambda^{k}(\Omega)}:=\int_{\Omega}\lambda\wedge\ast\mu,\quad 
  \Vert\lambda\Vert^{2}_{L^{2}\Lambda^{k}(\Omega)}:=\langle\lambda,\lambda\rangle_{L^{2}\Lambda^{k}(\Omega)},\quad\lambda,\mu\in L^{2}\Lambda^{k}(\Omega).
\end{equation}
There is another important Sobolev space \citep{arnold2006finite,arnold2010finite} given by
\begin{equation}\label{cobou}
H\Lambda^{k}(\Omega):=\{\lambda\in L^{2}\Lambda^{k}(\Omega)\mid d\lambda\in L^{2}\Lambda^{k+1}(\Omega)\},
\end{equation}
with norm defined as
\begin{equation*}
  \Vert\lambda\Vert^{2}_{H\Lambda^{k}(\Omega)}:=\Vert\lambda\Vert^{2}_{L^{2}\Lambda^{k}(\Omega)}+\Vert d\lambda\Vert^{2}_{L^{2}\Lambda^{k+1}(\Omega)}.
\end{equation*}
Analogously, we define \citep{arnold2006finite,retherford1993hilbert}
\begin{equation*}
H^{\ast}\Lambda^{k}(\Omega):=\ast H\Lambda^{n-k}(\Omega)=\{\lambda\in L^{2}\Lambda^{k}(\Omega)\mid\delta\lambda\in L^{2}\Lambda^{k-1}(\Omega)\}.
\end{equation*}

\subsection{Trace operator and Stokes' theorem}

\subsubsection{Trace operator}
Given a differential $k$-form $\alpha$, then $\alpha|_{\partial\Omega}$, the boundary value of $\alpha$, is defined by \citep{schwarz2006hodge}
\begin{equation*}
  \alpha\mid_{\partial\Omega}:\Gamma(T\Omega\mid_{\partial\Omega})\times\ldots\times\Gamma(T\Omega\mid_{\partial\Omega})\to C^{\infty}(\Omega),
\end{equation*}
where $\Gamma(T\Omega)$ is the space of all smooth vector fields. Any $k$-form $\alpha$ at $\partial\Omega$ can be decomposed into a tangential component $\textbf{t}\alpha:\Lambda^{k}(\Omega)\to\Lambda^{k}(\partial\Omega)$ and a normal component $\boldsymbol{n}\alpha:=\alpha|_{\partial\Omega}-\textbf{t}\alpha$, i.e. $\alpha|_{\partial\Omega}=\textbf{t}\alpha+\boldsymbol{n}\alpha$ \citep{schwarz2006hodge}. The trace operator is defined as $\text{tr}:=\textbf{t}$.

Choose any vector fields $X_{1},\ldots,X_{k}\in \Gamma(T\Omega\mid_{\partial\Omega})$, decomposed into their tangential and normal parts $X_{i}=X_{i}^{\parallel}+X_{i}^{\perp},\ i=1,\ldots,k$. We define then
\begin{equation}
  \mathrm{tr}(\alpha)(X_{1},\ldots,X_{k}):=\alpha(X_{1}^{\parallel},\ldots,X_{k}^{\parallel}),
\end{equation}
hence,
\begin{equation}\label{perp}
  \begin{aligned}
    \boldsymbol{n}(\alpha)(X_{1},\ldots,X_{k}):=\alpha(X_{1}^{\perp},\ldots,X_{k}^{\perp}).
  \end{aligned}
\end{equation}

The trace operator tr has the following important properties.
\begin{lemma}{\citep{arnold2006finite,schwarz2006hodge}}\label{trlemma}
  For $\alpha\in\Lambda^{k}(\Omega)$ and $\beta\in\Lambda^{j}(\Omega)$, there hold the following relations
  \begin{equation}\label{proper2}
    \mathrm{tr}(\alpha\wedge\beta)=\mathrm{tr}(\alpha)\wedge\mathrm{tr}(\beta),\quad \mathrm{tr}(d\alpha)=d(\mathrm{tr}(\alpha)),\quad\boldsymbol{n}(\delta\alpha)=\delta\boldsymbol{n}(\alpha),
  \end{equation}                                                 
  \begin{equation}\label{proper3}
    \ast\boldsymbol{n}(\alpha)=\mathrm{tr}(\ast\alpha),\quad \ast\mathrm{tr}(\alpha)=\boldsymbol{n}(\ast\alpha).
  \end{equation}
\end{lemma}

\begin{lemma}\citep[Stokes' Theorem]{marsden2013introduction}
  On an $n$-dimensional oriented manifold $\Omega$ with boundary $\partial\Omega$, we have for any $\alpha\in\Lambda^{n-1}(\Omega)$
\begin{equation}\label{stokes}
  \int_{\Omega}d\alpha=\int_{\partial\Omega}\mathrm{tr}(\alpha).
\end{equation}
\end{lemma}

\subsubsection{Boundary traces of Sobolev spaces $H^{1}\Lambda^{k}(\Omega)$}
We can extend the trace operator $\mathrm{tr}:\Lambda^{k}(\Omega)\to\Lambda^{k}(\partial\Omega)$ to a bounded linear operator from $H^{1}\Lambda^{k}(\Omega)$ onto $H^{\frac{1}{2}}\Lambda^{k}(\partial\Omega)$, where the space $H^{\frac{1}{2}}\Lambda^{k}(\Omega)$ and the norm are defined as
\begin{equation*}
  \begin{aligned}
     &H^{\frac{1}{2}}\Lambda^{k}(\partial\Omega):=\{\mu\in L^{2}\Lambda^{k}(\partial\Omega)\mid\exists\lambda\in H^{1}\Lambda^{k}(\Omega),\ s.t.\ \mathrm{tr}(\lambda)=\mu\},\\
     &\Vert\mu\Vert_{H^{\frac{1}{2}}\Lambda^{k}(\partial\Omega)}:=\inf_{\mathrm{tr}(\lambda)=\mu,\lambda\in H^{1}\Lambda^{k}(\Omega)}\Vert\lambda\Vert_{H^{1}\Lambda^{k}(\Omega)}.
  \end{aligned}
\end{equation*}

\subsubsection{Boundary traces of Sobolev spaces $H\Lambda^{k}(\Omega)$}
If $k=0$, since $H\Lambda^{0}(\Omega)=H^{1}\Lambda^{0}(\Omega)$, we have that $H^{\frac{1}{2}}\Lambda^{0}(\partial\Omega)$ is the trace space of $H\Lambda^{0}(\Omega)$.

If $0\leq k<n$, we denote the dual space of $H^{\frac{1}{2}}\Lambda^{k}(\partial\Omega)$ by $H^{-\frac{1}{2}}\Lambda^{k}(\partial\Omega)$. In \citep[p19]{arnold2006finite},
it is shown that the trace operator $\textrm{tr}:\Lambda^{k}(\Omega)\to\Lambda^{k}(\partial\Omega)$ can be extended to a bounded linear operator that maps $H\Lambda^{k}(\Omega)$ onto $H^{-\frac{1}{2}}\Lambda^{k}(\partial\Omega)$ for $0<k<n$.                                                                                                  
Since $H\Lambda^{n}(\Omega)=L^{2}\Lambda^{n}(\Omega)$, there is no trace for $k=n$.

Since $H^{\ast}\Lambda^{k}(\Omega)=\ast H\Lambda^{n-k}(\Omega)$, then for any $\mu\in H\Lambda^{n-k}(\Omega)$ there exists a $\lambda\in H^{\ast}\Lambda^{k}(\Omega)$, such that $\mu=(-1)^{k(n-k)}\ast\lambda$.
For $0<k<n$, $\mathrm{tr}(\mu)\in H^{-\frac{1}{2}}\Lambda^{n-k}(\partial\Omega)$, it immediately follows that $\mathrm{tr}(\ast\lambda)\in H^{-\frac{1}{2}}\Lambda^{n-k}(\partial\Omega)$. 
Similarly, for $\lambda\in H^{\ast}\Lambda^{n}(\Omega)$, $\mathrm{tr}(\ast\lambda)\in H^{\frac{1}{2}}\Lambda^{0}(\partial\Omega)$.

We can now state the integration by parts formula, see \citep{arnold2006finite,arnold2010finite,schwarz2006hodge}: for either $\lambda\in H\Lambda^{k-1}(\Omega)$, $\mu\in H^{1}\Lambda^{k}(\Omega)$ or $\lambda\in H^{1}\Lambda^{k-1}(\Omega)$, $\mu\in H^{\ast}\Lambda^{k}(\Omega)$,
\begin{equation}\label{part}
\begin{aligned}
  \langle d\lambda,\mu\rangle_{L^{2}\Lambda^{k}(\Omega)}&=\langle\lambda,\delta\mu\rangle_{L^{2}\Lambda^{k-1}(\Omega)}+\int_{\partial\Omega}\text{tr}(\lambda)\wedge\text{tr}(\ast\mu)\\
  &=\langle\lambda,\delta\mu\rangle_{L^{2}\Lambda^{k-1}(\Omega)}+\int_{\partial\Omega}\mathrm{tr}(\lambda)\wedge\ast\boldsymbol{n}(\mu),
\end{aligned}
\end{equation}
where the integral on the boundary is interpreted via the pairing of $H^{-\frac{1}{2}}\Lambda^{k-1}$ $(\partial\Omega)$ and $H^{\frac{1}{2}}\Lambda^{n-k}(\partial\Omega)$.

Furthermore, the boundary term $\mathrm{tr}(\lambda)\wedge\ast\boldsymbol{n}(\mu)$ can be represented as:
\begin{lemma}{\citep[p27]{schwarz2006hodge}}
  Given the differential forms $\lambda\in H\Lambda^{k-1}(\Omega)$ and $\mu\in H^{1}\Lambda^{k}(\Omega)$ or $\lambda\in H^{1}\Lambda^{k-1}(\Omega)$ and $\mu\in H^{\ast}\Lambda^{k}(\Omega)$. Let $\gamma:=\mathrm{tr}(\lambda)\wedge\ast\boldsymbol{n}(\mu)$. Then
  \begin{equation}\label{perp2}
    \gamma=\langle\lambda,i_{\mathcal{N}}\mu\rangle_{\Lambda^{k-1}}v_{\partial\Omega},
  \end{equation}
  where $v_{\partial\Omega}=i_{\mathcal{N}}v_{\Omega}|_{\partial\Omega}\in\Lambda^{n-1}(\partial\Omega)$ is the area form for the boundary $\partial\Omega$, with $\mathcal{N}$ the unit outward normal vector to the boundary $\partial\Omega$, and $v_{\Omega}\in\Lambda^{n}(\Omega)$ the volume form of $\Omega$.
\end{lemma}

\begin{lemma}\label{lemma2}
  Given $f_{1},f_{2}\in H^{1}\Lambda^{0}(\Omega)$. Let $\mu\in H\Lambda^{1}(\Omega)$ or $\mu\in H^{\ast}\Lambda^{1}(\Omega)$, with $\boldsymbol{n}(\mu)=0$ at $\partial\Omega$, then 
  \begin{equation}
  \int_{\Omega}df_{1}\wedge\ast\delta(\mu\wedge df_{2})=\int_{\partial\Omega}\big(\langle df_{1},\mu\rangle_{\Lambda^{1}}i_{\mathcal{N}}df_{2}\big)v_{\partial\Omega}.
  \end{equation}
\end{lemma}

\begin{proof}
  Using the integration by parts formula (\ref{part}), we have
  \begin{equation*}
    \int_{\Omega}df_{1}\wedge\ast\delta(\mu\wedge df_{2})=-\int_{\partial\Omega}\mathrm{tr}(df_{1})\wedge\ast\boldsymbol{n}(\mu\wedge df_{2}).
  \end{equation*}
  From (\ref{perp2}) we have that
  \begin{equation}\label{eta3}
    \begin{aligned}
      \mathrm{tr}(df_{1})\wedge\ast\boldsymbol{n}(\mu\wedge df_{2})&=\langle df_{1},i_{\mathcal{N}}(\mu\wedge df_{2})\rangle_{\Lambda^{1}}v_{\partial\Omega}\\
      &=\langle df_{1},i_{\mathcal{N}}\mu\wedge df_{2}-\mu\wedge i_{\mathcal{N}}df_{2}\rangle_{\Lambda^{1}}v_{\partial\Omega}.
    \end{aligned}
  \end{equation}
  Using the definition of $\boldsymbol{n}(\mu)$ in (\ref{perp}) and applying the boundary condition $\boldsymbol{n}(\mu)=0$, we have
  \begin{equation}\label{eta4}
     i_{\mathcal{N}}\mu=\mu(\mathcal{N})=\boldsymbol{n}(\mu)(\mathcal{N})=0. 
  \end{equation}
  Hence, by using (\ref{eta3}) and (\ref{eta4}), we have that
  \begin{equation*}
    \begin{aligned}
      \int_{\Omega}df_{1}\wedge\ast\delta(\mu\wedge df_{2})&=\int_{\partial\Omega}\langle df_{1},\mu\wedge i_{\mathcal{N}}df_{2}\rangle_{\Lambda^{1}}v_{\partial\Omega}\\
      &=\int_{\partial\Omega}\big(\langle df_{1},\mu\rangle_{\Lambda^{1}}i_{\mathcal{N}}df_{2}\big)v_{\partial\Omega}.
    \end{aligned}
  \end{equation*}
\end{proof}

\subsection{Hodge decompositions}
In order to define the Hodge decompositions, we first define the spaces $H\Lambda^{k}(\Omega)$ and $H^{\ast}\Lambda^{k}(\Omega)$ with a zero boundary trace denoted by 
\begin{equation*}
  \begin{aligned}
    \mathring{H}\Lambda^{k}(\Omega)&:=\{\mu\in H\Lambda^{k}(\Omega)\mid\text{tr}(\mu)=0\},\\
    \mathring{H}^{\ast}\Lambda^{k}(\Omega)&:=\{\mu\in H^{\ast}\Lambda^{k}(\Omega)\mid\text{tr}(\ast\mu)=0\}.
  \end{aligned}
\end{equation*}
\par
For the spaces $H\Lambda^{k}(\Omega)$, $H^{\ast}\Lambda^{k}(\Omega)$, $\mathring{H}\Lambda^{k}(\Omega)$ and $\mathring{H}^{\ast}\Lambda^{k}(\Omega)$, we have the corresponding spaces of cycles, boundaries, cocycles and coboundaries. For details, we refer to \citep{arnold2006finite,arnold2010finite}.
The $k$-cocycles are denoted as
\begin{equation}\label{cocy}
\mathfrak{Z}^{k}:=\{\mu\in H\Lambda^{k}(\Omega)\mid d\mu=0\}, \quad \mathring{\mathfrak{Z}}^{k}:=\{\mu\in\mathring{H}\Lambda^{k}(\Omega)\mid d\mu=0\},
\end{equation} 
and the $k$-coboundaries as
\begin{equation}\label{coboun}
\mathfrak{B}^{k}:=dH\Lambda^{k-1}(\Omega),\quad\mathring{\mathfrak{B}}^{k}:=d\mathring{H}\Lambda^{k-1}(\Omega).
\end{equation}
The $k$-cycles are denoted as
\begin{equation}\label{cy}
\mathfrak{Z}^{\ast k}:=\{\mu\in H^{\ast}\Lambda^{k}(\Omega)\mid\delta\mu=0\}, \quad \mathring{\mathfrak{Z}}^{\ast k}:=\{\mu\in\mathring{H}^{\ast}\Lambda^{k}(\Omega)\mid\delta\mu=0\},
\end{equation} 
and the $k$-boundaries as
\begin{equation}\label{bound}
\mathfrak{B}^{\ast k}:=\delta H^{\ast}\Lambda^{k+1}(\Omega),\quad\mathring{\mathfrak{B}}^{\ast k}:=\delta\mathring{H}^{\ast}\Lambda^{k+1}(\Omega).
\end{equation}

Each of the spaces of cycles is closed in $H\Lambda^{k}(\Omega)$, $H^{\ast}\Lambda^{k}(\Omega)$ or $L^{2}\Lambda^{k}(\Omega)$. The spaces of boundaries are closed in $L^{2}\Lambda^{k}(\Omega)$
due to the Poincar\'e inequality \citep{arnold2006finite}. Finally, the $k$-th harmonic forms are defined as
\begin{equation}\label{har}
\begin{aligned}
\mathfrak{H}^{k}&:=\{\mu\in H\Lambda^{k}(\Omega)\cap\mathring{H}^{\ast}\Lambda^{k}({\Omega})\mid d\mu=0,\ \delta\mu=0\},\\
\mathring{\mathfrak{H}}^{k}&:=\{\mu\in \mathring{H}\Lambda^{k}(\Omega)\cap H^{\ast}\Lambda^{k}({\Omega})\mid d\mu=0,\ \delta\mu=0\}.
\end{aligned}
\end{equation}

\begin{lemma}\citep[p22]{arnold2006finite}
  The spaces of cycles, cocycles, boundaries, and coboundaries, and their orthogonal complement in $L^{2}\Lambda^{k}(\Omega)$, indicated with $\bot$, satisfy the relations
  \begin{equation}
    \begin{aligned}
      &\mathfrak{Z}^{k\bot}\subset\mathfrak{B}^{k\bot}=\mathring{\mathfrak{Z}}^{\ast k},\quad &\mathfrak{Z}^{\ast k\bot}\subset\mathfrak{B}^{\ast k\bot}=\mathring{\mathfrak{Z}}^{k},\\
      &\mathring{\mathfrak{Z}}^{k\bot}\subset\mathring{\mathfrak{B}}^{k\bot}=\mathfrak{Z}^{\ast k},\quad &\mathring{\mathfrak{Z}}^{\ast k\bot}\subset\mathring{\mathfrak{B}}^{\ast k\bot}=\mathfrak{Z}^{k}.
    \end{aligned}
  \end{equation}
\end{lemma}

Based on the spaces that we defined above, there exist two types of Hodge decompositions of $L^{2}\Lambda^{k}(\Omega)$, each with different boundary conditions \citep{arnold2006finite,schwarz2006hodge,arnold2010finite},

\begin{subequations}
\begin{alignat}{2}
&\text{(i): }L^{2}\Lambda^{k}&=\mathfrak{B}^{k}\oplus\mathring{\mathfrak{B}}^{\ast k}\oplus\mathfrak{H}^{k},\label{hode1}\\
&\text{(ii): }L^{2}\Lambda^{k}&=\mathring{\mathfrak{B}}^{k}\oplus\mathfrak{B}^{\ast k}\oplus\mathring{\mathfrak{H}}^{k}.\label{hode2}
\end{alignat}
\end{subequations}

\subsection{Shape derivatives}

Let $V\subseteq\Omega$ with $\dim V=m\leq n=\dim\Omega$ be an oriented compact submanifold. Given a flow $\xi_{t}:
\Omega\to\Omega$ defined in a neighborhood of $V$ for small time $t$, and define the submanifold $V(t)=\xi_{t}V$ \citep{MR2884939}.                                                                                                                                                                                                                                                                                                                                                                                                                                                                                                                                                                                                                                                                                                                                                                                                                                                                                                                                                                                                                                                                                                                                                                                                                                                                                                                                                                                                                                                                                                                                                                                                                                                                                                                                                                                                                                                                                                                                                                                                                                                                                                                                                                                                                                                                                                                                                                                                                                                                                                                                                                                                                                                                                                                                                                                                                                                                                                                                                                                                                                                                                                                                                                                                                                                                                                                                                                                                                                                                                                                                                                                                                                                                                                                                                                                                                                                                                                                                                                                                                                                                                                                                                                                                                                                                                                                                                                                                                                                                                                                                                                                                                                                                                                                                                                                                                                                                                                                                                                                                                                                                                                                                                                                                                                                                                                                                                                                                                                                                                                                                                                                                                                                                                                                                                                                                                                                                                                                                                                                                                                                                                                                                                                                                                                                                                                                                                                                                                                                                                                                                                                                                                                                                                                                                                                                                                                                                                                                                                                                                                                                                                                                                                                                                                                                                                                                                                                                                                                                                                                                                                                                                                                                                                                                                                                                                                                                                                                                                                                                                      
Obviously, $V(0)=\xi_{0}V=V$. 

\begin{defn}{(\citep{MR3525087})}
  Let $X=\frac{d\xi_{t}(x)}{dt}\mid_{t=0}$ be the velocity field. Let the functional $I$ be defined as $I(V(t);\alpha)=\int_{V(t)}\alpha,\ \alpha\in\Lambda^{m}(\Omega)$. Then the shape derivative of the functional $I$ at $V$ in the direction of the vector field $X$ is defined as
  \begin{equation}\label{shape5}
    dI(V;X,\alpha):=\lim_{t\to 0}\frac{1}{t}\Big(I(V(t);\alpha)-I(V(0);\alpha)\Big).
  \end{equation}
\end{defn}
\noindent
From (\ref{shape5}), we have that
\begin{equation}\label{shape4}
  \begin{aligned}
    dI(V;X,\alpha)&=\lim_{t\to 0}\frac{1}{t}\Big(\int_{V(t)}\alpha(\xi_{t}x)-\int_{V}\alpha(x)\Big)\\
    &=\lim_{t\to 0}\frac{1}{t}\Big(\int_{V}\xi_{t}^{\ast}\alpha(\xi_{t}x)-\int_{V}\alpha(x)\Big)\\
    &=\int_{V}\lim_{t\to 0}\frac{1}{t}\Big(\xi_{t}^{\ast}\alpha(\xi_{t}x)-\alpha(x)\Big)\\
    &=\int_{V}\mathfrak{L}_{X}\alpha,
  \end{aligned}
\end{equation}
where $x\in V$ and $\xi^{\ast}_{t}:\Lambda^{m}(V(t))\to\Lambda^{m}(V)$ is a pull-back map and $\mathfrak{L}_{X}$ is the Lie derivative.

\section{Incompressible Euler equations, function spaces, and 
Hodge decompositions}

In this section, we will first define several Sobolev spaces of differential forms that will be extensively used in this article. 
Next, we will state the incompressible Euler equations with a free surface in terms of differential forms, discuss the Hodge decomposition of the velocity, and give some auxiliary results.

\subsection{Definition of Sobolev spaces}
For the analysis of the incompressible Euler equations, we introduce the following notation
\begin{align*}
  &P\Lambda^{k}(\Omega):=\left\{\gamma\in H\Lambda^{k}(\Omega)\cap H^{\ast}\Lambda^{k}(\Omega)\mid d\gamma=0\text{ in }\Omega\right\},\\
  &\mathring{P}\Lambda^{k}(\Omega):=\{\gamma\in H\Lambda^{k}(\Omega)\cap H^{\ast}\Lambda^{k}(\Omega)\mid d\gamma=0\text{ in }\Omega,\ \mathrm{tr}(\gamma)=0\text{ at }\Gamma\};\\
  &\text{   }\\
  &P^{\ast}\Lambda^{k}(\Omega):=\{\gamma\in H\Lambda^{k}(\Omega)\cap H^{\ast}\Lambda^{k}(\Omega)\mid \delta\gamma=0\text{ in }\Omega\},\\
  &\mathring{P}^{\ast}\Lambda^{k}(\Omega):=\{\gamma\in H\Lambda^{k}(\Omega)\cap H^{\ast}\Lambda^{k}(\Omega)\mid\delta\gamma=0\text{ in }\Omega,\ \mathrm{tr}(\ast\gamma)=0\text{ at }\Gamma\};\\
  &\text{   }\\
  &\mathring{V}\Lambda^{k}(\Omega):=\{\mu\in H\Lambda^{k}(\Omega)\cap\mathring{H}^{\ast}\Lambda^{k}(\Omega)\mid d\mu=0\text{ in }\Omega\},\\
  &\mathring{V}^{\ast}\Lambda^{k}(\Omega):=\{\mu\in\mathring{H}\Lambda^{k}(\Omega)\cap H^{\ast}\Lambda^{k}(\Omega)\mid \delta\mu=0\text{ in }\Omega\}.
\end{align*}

Note, $P^{\ast}\Lambda^{k}(\Omega)=\ast P\Lambda^{n-k}(\Omega)$, $\mathring{P}^{\ast}\Lambda^{k}(\Omega)=\ast\mathring{P}\Lambda^{n-k}(\Omega)$ and $\ast\mathring{V}\Lambda^{k}(\Omega)=\mathring{V}^{\ast}\Lambda^{n-k}(\Omega)$. For $k=1$ elements of $P^{\ast}\Lambda^{1}(\Omega)$ are proxies of divergence-free vector fields.

\subsection{Incompressible Euler equations with a free surface}
Let $\Omega(t)\subset\mathbb{R}^{n},n\in\{1,2,3\}$ be a time dependent oriented manifold with $(n-1)$-dimensional Lipschitz continuous boundary $\partial\Omega$. The boundary $\partial\Omega$ is split into a free surface part $\Sigma$ and a possibly empty fixed part $\Gamma=\partial\Omega\backslash\Sigma$. 
The incompressible Euler equations with a free surface, including gravity as an external force and surface tension at the free surface, see e.g. \citep[p447-455]{wehausen1960surface}, stated in differential forms \citep{abraham2012manifolds,arnold2008topological} are
\begin{subequations}\label{dieuler}
  \begin{align}
    &v_{t}=-i_{v^{\sharp}}dv-d\big(\frac{1}{2}\llangle v^{\sharp},v^{\sharp}\rrangle+\frac{\tilde{p}}{\rho}+\Phi\big)=-i_{v^{\sharp}}dv-dh\quad \text{in }\Omega, \label{dieuler1}\\
    &\delta v=0\quad \text{in }\Omega, \label{dieuler2}\\
    &\Sigma_{t}=\ast\boldsymbol{n}v \quad\text{ at}\ \Sigma, \label{dieuler3}\\
    &\mathrm{tr}(\tilde{p})=\tau k\quad\text{ at}\ \Sigma, \label{dieuler4}\\
    &\ast\boldsymbol{n}v=g \quad\text{ at }\Gamma\label{dieuler5}.
  \end{align}
\end{subequations}
The variables in (\ref{dieuler}) are defined as follows:
\begin{itemize}
  \item $v\in P^{\ast}\Lambda^{1}(\Omega)$ denotes the fluid velocity; $p\in H^{1}\Lambda^{0}(\Omega)$ the pressure, $\tilde{p}:=p-\bar{p}\in H^{1}\Lambda^{0}(\Omega)$, with $\bar{p}\in\mathbb{R}^{+}$ a constant atmospheric pressure; $\rho\in\mathbb{R}^{+}$ a constant fluid density; 
  $\Phi\in H^{1}\Lambda^{0}(\Omega)$ the gravity potential, which is a linear harmonic function with $d\Phi=g_{0}$, with $g_{0}$ the gravity acceleration and $\delta g_{0}=0$, e.g. in Cartesian coordinates $\Phi=g_{0}z$, with $z$ the free surface height above the reference $x-y$ plane.
  We also define 
  \begin{equation}\label{h}
    h:=\frac{1}{2}\llangle v^{\sharp},v^{\sharp}\rrangle+\frac{\tilde{p}}{\rho}+\Phi\in H^{1}\Lambda^{0}(\Omega).
  \end{equation}

  \item Given a flow $\xi_{t}:\Omega\to\Omega$, we define the free surface as $\Sigma(t)=\xi_{t}\Sigma$. Obviously, $\Sigma(0)=\Sigma$. We define the time derivative of $\Sigma$ as
  \begin{equation*}
    \Sigma_{t}:=\Big\{\frac{\partial }{\partial t}\xi_{t}(\Sigma(0))\in\mathbb{R}^{n},x\in\Sigma\subseteq\partial\Omega\Big\}.
  \end{equation*} 
  Hence, (\ref{dieuler3}) expresses that the velocity of points at the free surface $\Sigma$ equals the normal fluid velocity.
  
  \item $k$ is the mean curvature of $\Sigma$, defined as $k:=\mathrm{div}(\mathcal{N})$, with $\mathcal{N}$ the unit outward normal vector at $\partial\Omega$. The divergence $\mathrm{div}$ satisfies $\mathrm{div}(X)v_{\Omega}=d(i_{X}v_{\Omega})$ for any vectors $X\in T_{p}\Omega$ and volume form $v_{\Omega}$; $\tau\in\mathbb{R},\tau\geq 0$ is the surface tension, which is assumed to be a constant.
  Hence (\ref{dieuler4}) expresses that at the free surface $\Sigma$ the pressure jump $p-\bar{p}$ at $\Sigma$ is equal to the surface tension.
  
  \item $g\in H^{-\frac{1}{2}}\Lambda^{n-1}(\partial\Omega)$ is a prescribed normal velocity at $\Gamma$ and is equal to zero at $\Sigma$.
\end{itemize}
The Euler equations with a free surface (\ref{dieuler}) model free surface gravity waves, including waves around an object, capillary waves, and also for instance the shape of droplets in the air.

\subsection{Hodge decomposition of the velocity}

In this section, we will discuss the Hodge decomposition of the velocity. This will introduce three new variables, namely an exact form $\phi$ (potential), a co-exact form $\beta$ and a harmonic form $\alpha$.  
This Hodge decomposition will be an essential tool in the analysis of the different Hamiltonian formulations that will be discussed in the subsequent sections.

According to (\ref{hode1}) and (\ref{hode2}), there are two different Hodge decompositions for $v\in P^{\ast}\Lambda^{1}(\Omega)$, namely
\begin{equation}\label{vhode}
  v=d\phi+\delta\beta+\alpha,\ \mathrm{where}
\end{equation}
\begin{equation*}
  \begin{aligned}
    (i)&:d\phi\in\mathfrak{B}^{1},\ \delta\beta\in\mathring{\mathfrak{B}}^{\ast 1}\ \mathrm{and}\ \alpha\in\mathfrak{H}^{1},\\
    (ii)&:d\phi\in\mathring{\mathfrak{B}}^{1},\ \delta\beta\in\mathfrak{B}^{\ast 1}\ \mathrm{and}\ \alpha\in\mathring{\mathfrak{H}}^{1}.
  \end{aligned}
\end{equation*}

Since in general, an inviscid incompressible fluid must satisfy an inhomogeneous boundary condition for the normal component of the velocity at the domain boundary $\partial\Omega=\Sigma\cup\Gamma$, we will use Hodge decomposition $(i)$ for the velocity $v$.

\begin{lemma}\label{lemma10}
The exact form $\phi$, co-exact form $\beta$ and harmonic form $\alpha$ in the Hodge decomposition (\ref{vhode}) $v=d\phi+\delta\beta+\alpha$, where $v\in P^{\ast}\Lambda^{1}(\Omega),\ d\phi\in\mathfrak{B}^{1},\ \delta\beta\in\mathring{\mathfrak{B}}^{\ast 1}\text{ and }\alpha\in\mathfrak{H}^{1}$, solve the following boundary value problems separately 
\begin{align}  
  &
  \begin{cases}\label{potential1}
    &\delta d\phi=0 \quad \mathrm{in}\ \Omega,\\
    &\ast\boldsymbol{n}(d\phi)=\ast\boldsymbol{n}v=\frac{\partial\Sigma}{\partial t} \quad \mathrm{on}\ \Sigma, \\
    &\ast\boldsymbol{n}(d\phi)=\ast\boldsymbol{n}v=g \quad \mathrm{ on }\ \Gamma,
  \end{cases}\\
  &
  \begin{cases}\label{stream}
    &d\delta\beta=dv \quad\mathrm{in}\ \Omega,\\
    &d\beta=0 \quad\mathrm{in}\ \Omega,\\
    &\mathrm{tr}(\ast\beta)=0 \quad\mathrm{on}\ \Sigma\cup\Gamma,
  \end{cases}\\
  &
  \begin{cases}\label{harmonic1}
    &d\alpha=\delta\alpha=0 \quad \mathrm{in}\ \Omega,\\
    &\mathrm{tr}(\ast\alpha)=0 \quad \mathrm{on}\ \partial\Omega. \\
 \end{cases}
\end{align}
  
\end{lemma}

\begin{proof}
  From the Hodge decomposition $v=d\phi+\delta\beta+\alpha$ and the divergence-free condition (\ref{dieuler2}) since $v\in P^{\ast}\Lambda^{1}(\Omega)$, it is obvious that
  \begin{equation*}
    \delta v=\delta d\phi=0,\quad dv=d\delta\beta.
  \end{equation*}
  Since $d\delta\beta=dv\in\mathfrak{B}^{2}$, and from Lemma 2.5 we have $\mathfrak{B}^{2}\subset\mathfrak{Z}^{2}$, with $\mathfrak{Z}^{2}$ given by (\ref{cocy}), which implies that $d\beta=0$. Also, since $\delta\beta\in\mathring{\mathfrak{B}}^{\ast 1},\alpha\in\mathfrak{H}^{1}$, it directly follows that $\mathrm{tr}(\ast\beta)=\mathrm{tr}(\ast\alpha)=0$, and $d\alpha=\delta\alpha=0$. Then using (\ref{proper2}) and (\ref{proper3}), we obtain
  \begin{equation}\label{eta2}
    \begin{aligned}
      &\boldsymbol{n}(\delta\beta)=(-1)^{n-1}\boldsymbol{n}(\ast d\ast\beta)=(-1)^{n-1}\ast\mathrm{tr}(d\ast\beta)=(-1)^{n-1}\ast d\mathrm{tr}(\ast\beta)=0,\\
      &\boldsymbol{n}(\alpha)=(-1)^{n-1}\ast\mathrm{tr}(\ast\alpha)=0.
    \end{aligned}
  \end{equation}
  Thus, 
  \begin{equation*}
    \boldsymbol{n}(d\phi)=\boldsymbol{n}v-\boldsymbol{n}(\delta\beta)-\boldsymbol{n}(\alpha)=\boldsymbol{n}v.
  \end{equation*}
\end{proof}

In case $\Gamma=\partial\Omega$, hence a domain without a free surface, the harmonic form $\alpha$ in the Hodge decomposition (\ref{vhode}) is uncoupled from $d\phi$ and $\delta\beta$. It is then straightforward to
consider a connected domain and include $\alpha$ in the Hamiltonian formulation. At the free surface $\Sigma$, $\alpha$ is, however, coupled to $\Sigma$ through the boundary condition $\boldsymbol{n}(\alpha)=0$ at $\partial\Omega$, 
which significantly complicates the Hamiltonian formulation. For clarity of exposition, we assume therefore in the mathematical formulations that use a Hodge decomposition, in Sections 4 and 5 that the domain $\Omega$ is simply connected, which implies $\alpha=0$.

For the analysis in the subsequent sections we need explicit solutions of the equations for the exact form $\phi$ (\ref{potential1}) and the co-exact form $\beta$ (\ref{stream}).
For the Poisson equation (\ref{potential1}), we define the solution operator $N_{\phi}$, which is stated in the following lemma.

\begin{lemma}\label{lemma6}
  Given the space
  \begin{equation*}
    V:=\{\phi\in H\Lambda^{0}(\Omega)\mid\int_{\Omega}\phi v_{\Omega}=0\}.
  \end{equation*}
  The weak formulation of the Poisson equation (\ref{potential1}): Find $\phi\in V$, such that
  \begin{equation}\label{eq82}
    \langle d\phi,d\psi\rangle_{L^{2}\Lambda^{1}(\Omega)}=\int_{\partial\Omega}\mathrm{tr}(\psi)\wedge g_{\partial},\quad\forall\psi\in H\Lambda^{0}(\Omega),
  \end{equation}
  with $g_{\partial}\in H^{-\frac{1}{2}}\Lambda^{n-1}(\partial\Omega)$ defined as
  \begin{equation*}
    g_{\partial}=\ast\boldsymbol{n}(d\phi)=
    \begin{cases}
      \Sigma_{t}\quad\mathrm{on}\ \Sigma,\\
      g\quad\ \ \mathrm{on}\ \Gamma,
    \end{cases}
  \end{equation*}
  has a unique solution
  \begin{equation}\label{laplace}
    \phi(x)=N_{\phi}(g_{\partial}),
  \end{equation}
  with solution operator $N_{\phi}:H^{-\frac{1}{2}}\Lambda^{n-1}(\partial\Omega)\to H\Lambda^{0}(\Omega)$.
\end{lemma}

The proof of this lemma is standard and therefore omitted.
For (\ref{stream}) we define the solution operator $N_{\beta}$, which is stated in the following lemma.

\begin{lemma}\label{lemma11}
  The weak formulation of (\ref{stream}): Find $\beta\in\mathring{V}\Lambda^{2}(\Omega)$, such that 
  \begin{equation}\label{eq118}
    \langle\delta\beta,\delta\gamma\rangle_{L^{2}\Lambda^{1}(\Omega)}=\langle\omega,\gamma\rangle_{L^{2}\Lambda^{2}(\Omega)},\quad\forall\gamma\in\mathring{V}\Lambda^{2}(\Omega),
  \end{equation}
  with $\omega=dv\in\mathring{V}\Lambda^{2}(\Omega)$, has a unique solution
  \begin{equation}\label{eq81}
    \beta=N_{\beta}(\omega),
  \end{equation}
  with solution operator $N_{\beta}(\cdot):\mathring{V}\Lambda^{2}(\Omega)\to\mathring{V}\Lambda^{2}(\Omega)$.
\end{lemma}

\begin{proof}
  Using the Poincar\'e inequality \citep[Theorem 2.2]{arnold2006finite}, there exists a positive constant $c$ such that
  \begin{equation*}
    \Vert\beta\Vert_{L^{2}\Lambda^{2}(\Omega)}\leq c\big(\Vert d\beta\Vert_{L^{2}\Lambda^{3}(\Omega)}+\Vert\delta\beta\Vert_{L^{2}\Lambda^{1}(\Omega)}\big),\quad\forall\beta\in H\Lambda^{2}(\Omega)\cap\mathring{H}^{\ast}\Lambda^{2}(\Omega).
  \end{equation*}
  Taking $\gamma=\beta$ in (\ref{eq118}) and using $\beta\in\mathring{V}\Lambda^{2}(\Omega)$, which implies $d\beta=0$, gives the result.
\end{proof}

\section{Hamiltonian formulations of the incompressible Euler 
equations with a free surface}

In this section, we will formulate the classical Hamiltonian system for the incompressible Euler equations with a free surface (\ref{dieuler}) by applying the differential geometric tools introduced in Section 2.
Hamiltonian formulations for the $(v,\Sigma)$, $(\eta,\phi_{\partial},\Sigma)$ and $(\omega,\phi_{\partial},\Sigma)$ variables will be given. These Hamiltonian formulations will provide the mathematical framework for the port-Hamiltonian formulations
that will be discussed in Section 5. As an extension of the Hamiltonian system for dynamic free surface problems governed by surface tension considered in \cite{MR838352}, we will take a gravity force and, next to the free surface $\Sigma$, also a fixed boundary surface $\Gamma$ into consideration.
For the Hamiltonian formulations discussed in this section, we will assume the homogeneous boundary condition $g=0$ at $\Gamma$ in (\ref{dieuler}e). In Section 5 we will then consider in the port-Hamiltonian formulations the inhomogeneous boundary condition $g\neq 0$ at $\Gamma$.

\subsection{Velocity-free surface Hamiltonian formulation}

In order to define the Poisson bracket for the incompressible Euler equations with a free surface in terms of the $(v,\Sigma)$ variables, we state the 
corresponding functional derivatives $(\frac{\delta\mathcal{F}}{\delta v},\frac{\delta\mathcal{F}}{\delta\Sigma})$.

\begin{defn}
  Under weak smoothness assumptions on the functional $\mathcal{F}:P^{\ast}\Lambda^{1}(\Omega)\to\mathbb{R}$, the functional derivative $\frac{\delta\mathcal{F}}{\delta v}\in P\Lambda^{n-1}(\Omega)$ is defined by \citep{abraham2008foundations,abraham2012manifolds,morrison1998hamiltonian}
  \begin{equation}\label{func1}
    \begin{aligned}
      \int_{\Omega}\frac{\delta\mathcal{F}}{\delta v}\wedge\partial v=\lim_{\epsilon\to 0}\frac{1}{\epsilon}\big(\mathcal{F}(v+\epsilon\partial v )-\mathcal{F}(v)\big),\ \forall\partial v\in P^{\ast}\Lambda^{1}(\Omega).
    \end{aligned}
  \end{equation}
\end{defn}

Using the shape derivative stated in Section 2.4, we can give the following definition of the functional derivative with respect to the free surface boundary $\Sigma$.
\begin{defn}\label{defn1}
  Let $\partial\Sigma=\partial\Sigma_{v}v_{\Sigma}\in H^{-\frac{1}{2}}\Lambda^{n-1}(\Sigma)$, where $v_{\Sigma}$ denotes the surface form, $v_{\Sigma}=v_{\partial\Omega}|_{\Sigma}\in\Lambda^{n-1}(\Sigma)$, and $\partial\Sigma_{v}\in H^{-\frac{1}{2}}\Lambda^{0}(\Sigma)$ represent infinitesimal volume preserving variations of $\Sigma\subseteq\partial\Omega$ in its normal direction,
  which are extended to zero at $\Gamma\subseteq\partial\Omega$.
 Given arbitrary functionals $\mathcal{F}=\int_{V}f, \text{ where } f\in L^{2}\Lambda^{m}(\Omega)$ is the density function, with $V=\Sigma\subseteq\partial\Omega$ or $V=\Omega$, $m=\dim(V)$. Assume that the velocity $v$ remains constant when $\Sigma$ varies. Then the functional derivative $\frac{\delta\mathcal{F}}{\delta\Sigma}\in H^{\frac{1}{2}}\Lambda^{0}(\Sigma)$ is defined as 
  \begin{equation*}
    \int_{V}\frac{\delta\mathcal{F}}{\delta\Sigma}\wedge\partial\Sigma=d\mathcal{F}(V;\partial\Sigma_{v}\mathcal{N},f),\ \forall\partial\Sigma\in H^{-\frac{1}{2}}\Lambda^{n-1}(\Sigma),
  \end{equation*}
  where the shape derivative $d\mathcal{F}$ is given by (\ref{shape5}).
\end{defn}

\subsubsection{Hamiltonian and corresponding functional derivatives with respect to $(v,\Sigma)$}

The Hamiltonian $H:P^{\ast}\Lambda^{1}(\Omega)\times H^{-\frac{1}{2}}\Lambda^{n-1}(\Sigma)\to\mathbb{R}$ for the incompressible Euler equations with a free surface (\ref{dieuler}) is given by, see \citep{MR838352,morrison1998hamiltonian}
\begin{equation}\label{total2}
  \begin{aligned}
    H(v,\Sigma)&=\int_{\Omega}\big(\frac{1}{2}\llangle v^{\sharp},v^{\sharp}\rrangle+\Phi\big) v_{\Omega}+\frac{\tau}{\rho}\int_{\Sigma}v_{\Sigma}.
  \end{aligned}
\end{equation}

\begin{lemma}\label{lemma1}
  The functional derivatives of the Hamiltonian (\ref{total2}) with respect to $(v,\Sigma)$ are
   \begin{subequations}\label{h3}
    \begin{align}
      &\frac{\delta H}{\delta v}=(-1)^{n-1}\ast v\qquad\qquad\qquad\quad\text{in }\Omega,\label{h1}\\
      &\frac{\delta H}{\delta\Sigma}=\mathrm{tr}\big(\frac{1}{2}\llangle v^{\sharp},v^{\sharp}\rrangle+\Phi\big)+\frac{\tau k}{\rho}\qquad \text{at }\Sigma.\label{h2}
    \end{align}
     \end{subequations}
\end{lemma}

\begin{proof}
For any $v, \partial v\in P^{\ast}\Lambda^{1}(\Omega)$, we have
\begin{equation*}
\begin{aligned}
\lim_{\epsilon\to 0}\frac{1}{\epsilon}\big(H(v+\epsilon\partial v)-H(v)\big)&=\int_{\Omega}\llangle v^{\sharp},\partial v^{\sharp}\rrangle v_{\Omega}=\langle v,\partial v\rangle_{L^{2}\Lambda^{1}(\Omega)}\\
&=\int_{\Omega}(-1)^{n-1}\ast v\wedge\partial v.
\end{aligned}
\end{equation*}
Therefore, using (\ref{func1}) the functional derivative of the Hamiltonian with respect to $v$ is given by (\ref{h1}).
Next, we turn to compute the functional derivative with respect to $\Sigma$. Given
\begin{equation*}
    H_{1}\big(\Omega;(\frac{1}{2}\llangle v^{\sharp},v^{\sharp}\rrangle+\Phi)v_{\Omega}\big)=\int_{\Omega}\big(\frac{1}{2}\llangle v^{\sharp},v^{\sharp}\rrangle+\Phi\big) v_{\Omega}.
\end{equation*}
Using (\ref{shape4}), Cartan's formula $\mathfrak{L}_{X}=i_{X}d+di_{X}$, and Stokes' theorem, we obtain since $\partial\Sigma_{v}=0$ at $\Gamma$ that
\begin{equation*}
  \begin{aligned}
    dH_{1}\big(\Omega;\partial\Sigma_{v}\mathcal{N},(\frac{1}{2}\llangle v^{\sharp},v^{\sharp}\rrangle+\Phi)v_{\Omega}\big)&=\int_{\Omega}d\big((\frac{1}{2}\llangle v^{\sharp},v^{\sharp}\rrangle+\Phi) i_{\partial\Sigma_{v}\mathcal{N}} v_{\Omega}\big)\\
    &=\int_{\Sigma}\mathrm{tr}(\frac{1}{2}\llangle v^{\sharp},v^{\sharp}\rrangle+\Phi)\wedge i_{\partial\Sigma_{v}\mathcal{N}}v_{\Omega}.
  \end{aligned}
\end{equation*}
Using the relation $i_{fX}\alpha=fi_{X}\alpha$, where $f\in\Lambda^{0}(\Omega)$, $X\in T_{p}\Omega$ and $\alpha\in\Lambda^{k}(\Omega)$, we have
\begin{equation*}
  i_{\partial\Sigma_{v}\mathcal{N}}v_{\Omega}|_{\Sigma}=\partial\Sigma_{v}i_{\mathcal{N}}v_{\Omega}|_{\Sigma}=\partial\Sigma_{v}v_{\Sigma}=\partial\Sigma,
\end{equation*}
with $i_{\mathcal{N}}v_{\Omega}|_{\Sigma}=v_{\Sigma}$ and $\mathcal{N}$ the outward pointing normal vector to the boundary $\Sigma$. Thus,
\begin{equation}\label{shape6}
  dH_{1}\big(\Omega;\partial\Sigma_{v}\mathcal{N},(\frac{1}{2}\llangle v^{\sharp},v^{\sharp}\rrangle+\Phi)v_{\Omega}\big)=\int_{\Sigma}\mathrm{tr}\big(\frac{1}{2}\llangle v^{\sharp},v^{\sharp}\rrangle+\Phi \big)\wedge\partial\Sigma.
\end{equation}
Given
\begin{equation*}
  H_{2}(\Sigma;\tau v_{\Sigma})=\frac{\tau}{\rho}\int_{\Sigma}v_{\Sigma}.
\end{equation*}
Using (\ref{shape4}), Cartan's formula and the definition of the mean curvature $k$,
the shape derivative of $H_{2}$ is given by
\begin{equation*}
  dH_{2}(\Sigma;\partial\Sigma_{v}\mathcal{N},\frac{\tau}{\rho}v_{\Sigma})=\frac{\tau}{\rho}\int_{\Sigma}\mathfrak{L}_{\partial\Sigma_{v}\mathcal{N}}v_{\Sigma}=\frac{\tau}{\rho}\int_{\Sigma}\big(di_{\partial\Sigma_{v}\mathcal{N}}v_{\Sigma}+i_{\partial\Sigma_{v}\mathcal{N}}dv_{\Sigma}\big).
\end{equation*}
Since $i_{\partial\Sigma_{v}\mathcal{N}}v_{\Sigma}=\partial\Sigma_{v}i_{\mathcal{N}}i_{\mathcal{N}}v_{\Omega}|_{\Sigma}=0$, the shape derivative of $H_{2}$ becomes
\begin{equation}\label{shape2}
  dH_{2}(\Sigma;\partial\Sigma_{v}\mathcal{N},\frac{\tau}{\rho}v_{\Sigma})=\frac{\tau}{\rho}\int_{\Sigma}i_{\partial\Sigma_{v}\mathcal{N}}dv_{\Sigma}=\frac{\tau}{\rho}\int_{\Sigma}\partial\Sigma_{v}i_{\mathcal{N}}\mathrm{div}(\mathcal{N})v_{\Omega}=\frac{\tau}{\rho}\int_{\Sigma}k\partial\Sigma.
\end{equation}
From (\ref{shape6}) and (\ref{shape2}), we obtain (\ref{h2}).
\end{proof}

\subsubsection{Hamiltonian formulation with respect to $(v,\Sigma)$ variables}

Inspired by the classical form of the Poisson bracket for the incompressible Euler equations, see e.g. \citep{beris1990poisson,MR838352}, the generalized Poisson bracket for the incompressible Euler equations with a free surface can be formulated as
\begin{defn}
Let $\mathcal{F}$, $\mathcal{G}:\mathring{P}^{\ast}\Lambda^{1}(\Omega)\times H^{-\frac{1}{2}}\Lambda^{n-1}(\Sigma)\rightarrow\mathbb{R}$ be arbitrary functionals. The Poisson bracket $\{\cdot,\cdot\}(v,\Sigma):\mathcal{F}\times\mathcal{G}\to\mathbb{R}$
in terms of the velocity $v$ and free surface $\Sigma$ is defined as 
\begin{equation}\label{bracket3}
  \begin{aligned}
    \left\{\mathcal{F},\mathcal{G}\right\}(v,\Sigma)=&(-1)^{n-1}\int_{\Omega}(\ast dv)\wedge(\ast\frac{\delta \mathcal{G}}{\delta v})\wedge(\ast\frac{\delta \mathcal{F}}{\delta v})\\
    +&(-1)^{n-1}\int_{\Sigma}\Big(\frac{\delta\mathcal{F}}{\delta\Sigma}\wedge\mathrm{tr}(\frac{\delta\mathcal{G}}{\delta v})-\frac{\delta\mathcal{G}}{\delta\Sigma}\wedge\mathrm{tr}(\frac{\delta\mathcal{F}}{\delta v})\Big).                   
  \end{aligned}
  \end{equation}
\end{defn}

\begin{remark}
  The bracket $\{\cdot,\cdot\}(v,\Sigma)$ defined in (\ref{bracket3}) is bilinear, skew-symmetric and satisfies the Jacobi identity \citep{MR838352}.
\end{remark}

\begin{theorem}\label{thm3}
  Given a connected, oriented Lipschitz continuous domain $\Omega$ with a non-overlapping free surface boundary $\Sigma\subseteq\partial\Omega$ and fixed boundary $\Gamma\subseteq\partial\Omega$, $\Sigma\cup\Gamma=\partial\Omega$.
  For any functional $\mathcal{F}(v,\Sigma):\mathring{P}^{\ast}\Lambda^{1}(\Omega)\times H^{-\frac{1}{2}}\Lambda^{n-1}(\Sigma)\to\mathbb{R}$, and Hamiltonian functional (\ref{total2}), the incompressible Euler equations with a free surface (\ref{dieuler}), and $g=0$ at $\Gamma$,
   can be expressed as the Hamiltonian system
  \begin{equation}\label{hami3}
    \dot{\mathcal{F}}(v,\Sigma)=\{\mathcal{F},H\}(v,\Sigma).
  \end{equation}
\end{theorem}

\begin{proof}
  First, we will prove that (\ref{dieuler}) implies (\ref{hami3}). Inserting the functional derivatives (\ref{h3}) with respect to $(v,\Sigma)$ into the Poisson bracket (\ref{bracket3}), we have
  \begin{equation*}\label{eq37}
    \begin{aligned}
      &\{\mathcal{F},H\}(v,\Sigma)=(-1)^{n-1}\int_{\Omega}(\ast dv)\wedge v\wedge(\ast\frac{\delta\mathcal{F}}{\delta v})+(-1)^{n-1}\int_{\Sigma}\frac{\delta\mathcal{F}}{\delta\Sigma}\wedge(-1)^{n-1}\mathrm{tr}(\ast v)\\
      &-(-1)^{n-1}\int_{\Sigma}\mathrm{tr}(\frac{1}{2}\llangle v^{\sharp},v^{\sharp}\rrangle+\Phi)\wedge\mathrm{tr}(\frac{\delta\mathcal{F}}{\delta v})-(-1)^{n-1}\int_{\Sigma}\frac{\tau k}{\rho}\wedge\mathrm{tr}(\frac{\delta\mathcal{F}}{\delta v}).
    \end{aligned}
  \end{equation*}
  Using $\frac{\delta\mathcal{F}}{\delta v}\in\mathring{P}\Lambda^{n-1}(\Omega)$, which implies that $d(\frac{\delta\mathcal{F}}{\delta v})=0$ in $\Omega$ and $\mathrm{tr}(\frac{\delta\mathcal{F}}{\delta v})=0$ at $\Gamma$, and Stokes' theorem, gives
  \begin{equation*}
    \begin{aligned}
    \int_{\Sigma}\mathrm{tr}(\frac{1}{2}\llangle v^{\sharp},v^{\sharp}\rrangle+\Phi)\wedge\mathrm{tr}(\frac{\delta\mathcal{F}}{\delta v})=\int_{\Omega}d\big(\frac{1}{2}\llangle v^{\sharp},v^{\sharp}\rrangle+\Phi\big)\wedge\frac{\delta\mathcal{F}}{\delta v}.
    \end{aligned}
  \end{equation*}
Then, $\{\mathcal{F},H\}$ can be rewritten as 
\begin{equation}\label{Ha3}
  \begin{aligned}
    \{\mathcal{F},H\}(v,\Sigma)=&-\int_{\Omega}\frac{\delta\mathcal{F}}{\delta v}\wedge\Big(i_{v^{\sharp}}dv+d(\frac{1}{2}\llangle v^{\sharp},v^{\sharp}\rrangle+\Phi)\Big)+\int_{\Sigma}\frac{\delta\mathcal{F}}{\delta\Sigma}\wedge\ast\boldsymbol{n}v\\
     &+(-1)^{n}\int_{\Sigma}\frac{\tau k}{\rho}\wedge\mathrm{tr}(\frac{\delta\mathcal{F}}{\delta v}),
  \end{aligned}
\end{equation}
where we used that $\ast(v\wedge\ast dv)=i_{v^{\sharp}}dv$, which follows from (\ref{rela4}).
From $\frac{\delta\mathcal{F}}{\delta v}\in\mathring{P}\Lambda^{n-1}(\Omega)$ and Stokes' theorem, together with (\ref{dieuler4}), we obtain
\begin{equation}\label{eq101}
  \begin{aligned}
    \int_{\Sigma}\frac{\tau k}{\rho}\wedge\mathrm{tr}(\frac{\delta\mathcal{F}}{\delta v})=\int_{\partial\Omega}\mathrm{tr}(\frac{\tilde{p}}{\rho})\wedge\mathrm{tr}(\frac{\delta\mathcal{F}}{\delta v})=(-1)^{n-1}\int_{\Omega}\frac{\delta\mathcal{F}}{\delta v}\wedge d\big(\frac{\tilde{p}}{\rho}\big).
  \end{aligned}
\end{equation} 
Using (\ref{eq101}), and subsequently (\ref{dieuler1}), (\ref{dieuler3}) and the functional chain rule, we obtain from (\ref{Ha3})
\begin{equation*}
  \begin{aligned}
    \{\mathcal{F},H\}(v,\Sigma)=&-\int_{\Omega}\frac{\delta\mathcal{F}}{\delta v}\wedge\Big(i_{v^{\sharp}}dv+d(\frac{1}{2}\llangle v^{\sharp},v^{\sharp}\rrangle+\frac{\tilde{p}}{\rho}+\Phi)\Big)+\int_{\Sigma}\frac{\delta\mathcal{F}}{\delta\Sigma}\wedge\ast\boldsymbol{n}v\\
    =&\int_{\Omega}\frac{\delta\mathcal{F}}{\delta v}\wedge v_{t}+\int_{\Sigma}\frac{\delta\mathcal{F}}{\delta\Sigma}\wedge\Sigma_{t}\\
    =&\dot{\mathcal{F}}(v,\Sigma).
  \end{aligned}
\end{equation*}

Next, we will prove that (\ref{hami3}) implies (\ref{dieuler}). 
If we consider $(\ref{hami3})$ with the right hand side written as (\ref{Ha3}) 
and add and subtract $\int_{\Omega}\frac{\delta\mathcal{F}}{\delta v}\wedge d\big(\frac{\tilde{p}}{\rho}\big)$ to the right hand side of (\ref{hami3}),
then we obtain using $\frac{\delta\mathcal{F}}{\delta v}\in\mathring{P}\Lambda^{n-1}(\Omega)$ and Stokes' theorem
\begin{equation*}
  \begin{aligned}
    &\int_{\Omega}\frac{\delta\mathcal{F}}{\delta v}\wedge\big(v_{t}+i_{v^{\sharp}}dv+d(\frac{1}{2}\llangle v^{\sharp},v^{\sharp}\rrangle+\frac{\tilde{p}}{\rho}+\Phi)\big)+\int_{\Sigma}\frac{\delta\mathcal{F}}{\delta\Sigma}\wedge\big(\Sigma_{t}-\ast\boldsymbol{n}v\big)\\
    &+(-1)^{n}\int_{\Sigma}\mathrm{tr}(\frac{\delta\mathcal{F}}{\delta v})\wedge\big(\mathrm{tr}(\frac{\tilde{p}}{\rho})-\frac{\tau k}{\rho}\big)=0.
  \end{aligned}
\end{equation*}
Taking $\frac{\delta\mathcal{F}}{\delta v}\in\mathring{P}\Lambda^{n-1}(\Omega)$, $\frac{\delta\mathcal{F}}{\delta\Sigma}\in H^{\frac{1}{2}}\Lambda^{0}(\Omega)$ arbitrary gives (\ref{dieuler1}), (\ref{dieuler3}) and (\ref{dieuler4}).
Equations (\ref{dieuler2}) and (\ref{dieuler5}), with $g=0$ at $\Gamma$, are automatically satisfies if $v\in\mathring{P}^{\ast}\Lambda^{1}(\Omega)$.
\end{proof}

\begin{corollary}
  For the Hamiltonian $H:P^{\ast}\Lambda^{1}(\Omega)\times H^{-\frac{1}{2}}\Lambda^{n-1}(\Sigma)\to\mathbb{R}$ stated in (\ref{total2}), we have
  \begin{equation*}
    \dot{H}=\{H,H\}(v,\Sigma)=0.
  \end{equation*}
\end{corollary}

\begin{proof}
  This result follows directly from Theorem \ref{thm3} using the skew-symmetry of the Poisson bracket (\ref{bracket3}).
\end{proof}

This means that the total energy of an inviscid fluid with free surface boundary $\Sigma$ and homogeneous boundary condition $g=0$ at the fixed boundary $\Gamma$ is conserved when there is no fluid flow through the fixed boundary $\Gamma$.

In Section 5.2, we will extend the Hamiltonian system stated in Theorem \ref{thm3} to a port-Hamiltonian system, which also incorporates the inhomogeneous boundary condition (\ref{dieuler5}) with $g\neq 0$, and accounts for a net energy flow through the open surface $\Gamma$.

\subsection{Solenoidal velocity-potential-free surface Hamiltonian formulation}

Based on the Poisson bracket discussed in Section 4.1 and the Hodge decompositions given in Section 3, we will discuss in this section the Hamiltonian formulation of the incompressible Euler equations with a free surface in terms of a potential function $\phi$, a solenoidal velocity field $\eta$ and the free surface $\Sigma$.

We will denote the restriction of the potential function $\phi\in H\Lambda^{0}(\Omega)$ to the boundary $\partial\Omega$ as $\phi_{\partial}=\mathrm{tr}(\phi)\in H^{\frac{1}{2}}\Lambda^{0}(\partial\Omega)$. The solenoidal velocity field is defined as $\eta=\delta\beta\in\mathring{\mathfrak{B}}^{\ast 1}$, which is tangential to $\partial\Omega$. This follows directly from (\ref{eta2}), which states that $\boldsymbol{n}(\eta)=0$ at $\partial\Omega$.

In the computation of the functional derivatives of $\mathcal{F}$ in Section 4.1, we can use the fact that the variations of $v$ and $\Sigma$ are not restricted at the free surface. This is not the case for $\eta$ and $\Sigma$ since they are coupled through the boundary condition $\boldsymbol{n}(\eta)=0$. 

Given the volume-preserving diffeomorphism $g_{t}=\varphi_{t}\circ\xi_{t}:\Omega\to\Omega$, where $\xi_{t},\varphi_{t}:\Omega\to\Omega$ are both also volume-preserving 
diffeomorphisms, with $\xi_{t}$ divergence-free on $\Omega$ and parallel to $\partial\Omega$.
Differentiating the diffeomorphism $g_{t}$ and using the relation $\partial\Sigma=\ast\boldsymbol{n}(\partial\varphi)$,
we obtain that variations $\partial\eta$ have the form \citep{MR838352}
\begin{equation}\label{eta1}
  \partial\eta=\eta'+[\eta,u]_{1},
\end{equation} 
where $\eta'\in P^{\ast}\Lambda^{1}(\Omega)$ has a trace parallel to $\partial\Omega$, which implies $\boldsymbol{n}(\eta')=0$ at $\partial\Omega$, $u\in P^{\ast}\Lambda^{1}(\Omega)$ and is coupled to the surface perturbations $\partial\Sigma$ through $\ast\boldsymbol{n}(u)=\partial\Sigma$. At the fixed surface $\Gamma$, $\ast\boldsymbol{n}(u)=0$. 
The bracket $[\cdot,\cdot]_{1}:P^{\ast}\Lambda^{1}(\Omega)\times P^{\ast}\Lambda^{1}(\Omega)\to P^{\ast}\Lambda^{1}(\Omega)$ is a Lie-bracket, which satisfies for $\alpha,\beta\in P^{\ast}\Lambda^{1}(\Omega)$, see e.g. \citep{MR2884939,rashad2021port1}, 
\begin{equation}\label{lie}
  [\alpha,\beta]_{1}:=[\alpha^{\sharp},\beta^{\sharp}]^{\flat}=(-1)^{n-1}\delta(\alpha\wedge\beta),
\end{equation}
where $[\cdot,\cdot]:T\Omega\times T\Omega\to T\Omega$ is the classical Lie bracket for vector fields.

We now consider the dynamic equations for $\eta$ and $\phi$.
Following the same approach as used for (\ref{eta1}), we have that $\eta\in\mathring{\mathfrak{B}}^{\ast 1}$ satisfies the equation
\begin{equation}\label{varia2}
  \eta_{t}=[\eta,u']_{1},
\end{equation}
where $\ast\boldsymbol{n}(u')=\Sigma_{t}$ at $\Sigma$ and $\ast\boldsymbol{n}(u')=0$ at $\Gamma$. We therefore choose $u'=dN_{\phi}(\Sigma_{t})$.
Using (\ref{rela4}), (\ref{dieuler1}) and (\ref{varia2}), it follows that
\begin{equation}
  d\phi_{t}=v_{t}-\eta_{t}=v_{t}-[\eta,u']_{1}=-\ast(v\wedge\ast d\eta)-d\big(\frac{1}{2}\llangle v^{\sharp},v^{\sharp}\rrangle+\frac{\tilde{p}}{\rho}+\Phi\big)-[\eta,u']_{1}.
\end{equation}
Hence, the incompressible Euler equations with a free surface (\ref{dieuler}) can be expressed in terms of the $(\eta,\phi_{\partial},\Sigma)$ variables as
\begin{subequations}\label{vor4}
  \begin{empheq}[left=\empheqlbrace]{align}
    &\eta_{t}=[\eta,u']_{1}\quad\text{ in }\Omega\\
    &d\phi_{t}=-\ast(v\wedge\ast d\eta)-d\big(\frac{1}{2}\llangle v^{\sharp},v^{\sharp}\rrangle+\frac{\tilde{p}}{\rho}+\Phi\big)-[\eta,u']_{1}\text{ in }\Omega,\\
    &\delta d\phi=0\text{ in }\Omega,\\
    &\ast\boldsymbol{n}(d\phi)=\Sigma_{t}\text{ at }\Sigma,\\
    &\ast\boldsymbol{n}(d\phi)=g\text{ at }\Gamma,\\
    &\mathrm{tr}(\frac{\tilde{p}}{\rho})=\frac{\tau k}{\rho}\text{ at }\Sigma,
  \end{empheq}
\end{subequations}
with
\begin{equation}\label{vor6}
  v=d\phi+\eta\quad\text{in }\Omega.
\end{equation}
Note, taking the divergence of (\ref{vor4}b) we obtain
\begin{equation}\label{vor5}
  \delta d(\frac{\tilde{p}}{\rho})=\delta i_{v^{\sharp}}d\eta+\delta d\big(\frac{1}{2}\llangle v^{\sharp},v^{\sharp}\rrangle\big),
\end{equation}
hence (\ref{vor4}) can be reduced to two Poisson equations, namely (\ref{vor5}), (\ref{vor4}f) for the pressure $\frac{\tilde{p}}{\rho}$
and (\ref{vor4}b)--(\ref{vor4}e) for the potential $\phi$, and a transport equation for $\eta$ (\ref{vor4}a).

\subsubsection{Hodge decomposition of the functional derivative with respect to velocity}

For the analysis in the following sections, we will also need a Hodge decomposition of the functional derivative $\frac{\delta\mathcal{F}}{\delta v}\in P\Lambda^{n-1}(\Omega)$. We will use the Hodge decomposition (\ref{hode2}) for a simply connected domain, which results in 
\begin{equation}\label{hode3}
  \frac{\delta\mathcal{F}}{\delta v}=\frac{\delta_{\beta}\mathcal{F}}{\delta v}+\frac{\delta_{\phi}\mathcal{F}}{\delta v},
\end{equation}
where $\frac{\delta_{\beta}\mathcal{F}}{\delta v}\in\mathring{\mathfrak{B}}^{(n-1)}$ and $\frac{\delta_{\phi}\mathcal{F}}{\delta v}\in\mathfrak{B}^{\ast(n-1)}$.\\

\begin{lemma}\label{lemma5}
  Given a simply-connected, oriented Lipschitz continuous domain $\Omega$.
  For $\frac{\delta\mathcal{F}}{\delta v}\in P\Lambda^{n-1}(\Omega)$, the functional derivative $\frac{\delta_{\phi}\mathcal{F}}{\delta v}\in\mathfrak{B}^{\ast(n-1)}$ can be expressed as 
  \begin{equation}\label{eta5}
    \frac{\delta_{\phi}\mathcal{F}}{\delta v}=-\ast dw,
  \end{equation}
  where 
  \begin{equation}\label{extraterm}
    w=-N_{\phi}\big(\mathrm{tr}(\frac{\delta\mathcal{F}}{\delta v})\big)\in H\Lambda^{0}(\Omega),
  \end{equation}
  with $N_{\phi}:H^{-\frac{1}{2}}\Lambda^{n-1}(\partial\Omega)\to H\Lambda^{0}(\Omega)$ the solution operator of the Poisson equation (\ref{potential1}).
\end{lemma}
  
\begin{proof}
  Since $\frac{\delta_{\beta}\mathcal{F}}{\delta v}\in\mathring{\mathfrak{B}}^{(n-1)}$, from definition (\ref{coboun}), there exists an $w'\in\mathring{H}\Lambda^{n-2}(\Omega)$ such that 
  \begin{equation}\label{eq59}
    \frac{\delta_{\beta}\mathcal{F}}{\delta v}=dw',
  \end{equation}
  with boundary condition
  \begin{equation*}
    \mathrm{tr}(\frac{\delta_{\beta}\mathcal{F}}{\delta v})=\mathrm{tr}(dw')=0.
  \end{equation*}
  Similarly, from $\frac{\delta_{\phi}\mathcal{F}}{\delta v}\in\mathfrak{B}^{\ast(n-1)}$ and definition (\ref{bound}), there exists an $w\in H\Lambda^{0}(\Omega)$ such that 
  \begin{equation}\label{eq55}
    \frac{\delta_{\phi}\mathcal{F}}{\delta v}=\delta(\ast w)=-\ast dw.
  \end{equation}
  From $\frac{\delta\mathcal{F}}{\delta v}\in P\Lambda^{n-1}(\Omega)$ we have $d(\frac{\delta\mathcal{F}}{\delta v})=0$ and using subsequently (\ref{hode3}), (\ref{eq59}) and (\ref{eq55}), we obtain that $w$ satisfies the Laplace equation (\ref{extraterm}), since
  \begin{equation*}
    0=\ast d(\frac{\delta\mathcal{F}}{\delta v})=(-1)^{n}\delta(\ast\frac{\delta_{\beta}\mathcal{F}}{\delta v}+\ast\frac{\delta_{\phi}\mathcal{F}}{\delta v})=(-1)^{n}\big(\delta(\ast dw')-\delta(\ast\ast dw)\big)=\delta dw,
  \end{equation*}
  with boundary condition
  \begin{equation*}
    \ast\boldsymbol{n}(dw)=\mathrm{tr}(\ast dw)=-\mathrm{tr}(\frac{\delta\mathcal{F}}{\delta v})+\mathrm{tr}(\frac{\delta_{\beta}\mathcal{F}}{\delta v})=-\mathrm{tr}(\frac{\delta\mathcal{F}}{\delta v}).
  \end{equation*}
  Using the solution operator for the Laplace equation as stated in Lemma \ref{lemma6}, we obtain (\ref{extraterm}).
\end{proof}

\subsubsection{Functional derivatives with respect to $(\eta,\phi_{\partial},\Sigma)$ variables}

The first step in the derivation of the solenoidal velocity-potential-free surface $(\eta,\phi_{\partial},\Sigma)$ Hamiltonian formulation of the incompressible Euler equations with a free surface is to obtain the functional derivatives with respect to $\eta$, $\phi_{\partial}$ and $\Sigma$.

Given arbitrary functionals $\mathcal{F}(v,\Sigma):P^{\ast}\Lambda^{1}(\Omega)\times H^{-\frac{1}{2}}\Lambda^{n-1}(\Sigma)\to\mathbb{R}$ and $\tilde{\mathcal{F}}(\eta,\phi_{\partial},\Sigma):\mathring{\mathfrak{B}}^{\ast 1}\times H^{\frac{1}{2}}\Lambda^{0}(\partial\Omega)\times H^{-\frac{1}{2}}\Lambda^{n-1}(\Sigma)\to\mathbb{R}$, which satisfy the relation
\begin{equation}\label{hami1}
  \tilde{\mathcal{F}}(\eta,\phi_{\partial},\Sigma)=\mathcal{F}(v,\Sigma),
\end{equation}
we can define the following functional derivatives.

\begin{defn}
The functional derivative $\frac{\delta\tilde{\mathcal{F}}}{\delta\eta}\in\mathring{\mathfrak{B}}^{(n-1)}$ is defined as
\begin{equation*}
  \int_{\Omega}\frac{\delta\tilde{\mathcal{F}}}{\delta\eta}\wedge\partial\eta=\lim_{\epsilon\to 0}\frac{1}{\epsilon}\big(\tilde{\mathcal{F}}(\eta+\epsilon\partial\eta)-\tilde{\mathcal{F}}(\eta)\big),\ \forall\partial\eta\in\mathring{\mathfrak{B}}^{\ast 1}.
\end{equation*}
The functional derivative $\frac{\delta\tilde{\mathcal{F}}}{\delta\phi_{\partial}}\in H^{-\frac{1}{2}}\Lambda^{n-1}(\partial\Omega)$ is defined as
\begin{equation*}
  \int_{\partial\Omega}\frac{\delta\tilde{\mathcal{F}}}{\delta\phi_{\partial}}\wedge\partial\phi_{\partial}=\lim_{\epsilon\to 0}\frac{1}{\epsilon}\big(\tilde{\mathcal{F}}(\phi_{\partial}+\epsilon\partial\phi_{\partial})-\tilde{\mathcal{F}}(\phi_{\partial})\big),\ \forall\partial\phi_{\partial}\in H^{\frac{1}{2}}\Lambda^{0}(\partial\Omega).
\end{equation*}
\end{defn} 
\noindent
The functional derivative $\frac{\delta\tilde{\mathcal{F}}}{\delta\Sigma}$ is stated in Definition \ref{defn1} with $\mathcal{F}$ replaced by $\tilde{\mathcal{F}}$. 

The next lemma provides the relation between the functional derivatives $(\frac{\delta\mathcal{F}}{\delta v},\frac{\delta\mathcal{F}}{\delta\Sigma})$ and $(\frac{\delta\tilde{\mathcal{F}}}{\delta\eta},\frac{\delta\tilde{\mathcal{F}}}{\delta\phi_{\partial}},\frac{\delta\tilde{\mathcal{F}}}{\delta\Sigma})$.
Here, we choose $u=dN_{\phi}(\partial\Sigma)\in H\Lambda^{1}(\Omega)$, with $\partial\Sigma=0$ at $\Gamma$ and $N_{\phi}$ the solution operator of the Laplace equation stated in Lemma \ref{lemma6}.

\begin{lemma}\label{lemma7}
  For any functionals $\tilde{\mathcal{F}}(\eta,\phi_{\partial},\Sigma):$ $\mathring{\mathfrak{B}}^{\ast 1}\times$  $H^{\frac{1}{2}}\Lambda^{0}(\partial\Omega)\times$ $H^{-\frac{1}{2}}\Lambda^{n-1}(\Sigma)$ $\to\mathbb{R}$ and $\mathcal{F}(v,\Sigma):P^{\ast}\Lambda^{1}(\Omega)\times H^{-\frac{1}{2}}\Lambda^{n-1}(\Sigma)\to\mathbb{R}$, 
  the functional derivatives have the following relations
   \begin{subequations}\label{func3}
    \begin{empheq}[left=\empheqlbrace]{align}
      \frac{\delta\tilde{\mathcal{F}}}{\delta\eta}&=\frac{\delta_{\beta}\mathcal{F}}{\delta v},\label{eq103}\\
      \frac{\delta\tilde{\mathcal{F}}}{\delta\phi_{\partial}}&=(-1)^{n-1}\mathrm{tr}(\frac{\delta\mathcal{F}}{\delta v}),\label{eq104}\\
      \frac{\delta\tilde{\mathcal{F}}}{\delta\Sigma}&=\frac{\delta\mathcal{F}}{\delta\Sigma}+(-1)^{n-1}\langle dN_{\phi}\big(\frac{\delta\tilde{\mathcal{F}}}{\delta\phi_{\partial}}\big),\eta\rangle_{\Lambda^{1}}.\label{eq105}
    \end{empheq}
  \end{subequations}
  Furthermore, we have
  \begin{equation}\label{func4}
    \frac{\delta\mathcal{F}}{\delta v}=\frac{\delta\tilde{\mathcal{F}}}{\delta\eta}+(-1)^{n-1}\ast dN_{\phi}\big(\frac{\delta\tilde{\mathcal{F}}}{\delta\phi_{\partial}}\big).
  \end{equation}
\end{lemma}

\begin{proof}
The variation $\partial v$ has the form
\begin{equation*}
  \partial v=d(\partial\phi)+\partial\eta.
\end{equation*}
Hence,
\begin{equation*}
  \begin{aligned}
  \int_{\Omega}\frac{\delta\mathcal{F}}{\delta v}\wedge\partial v=\int_{\Omega}\frac{\delta\mathcal{F}}{\delta v}\wedge \big(d(\partial\phi)+\partial\eta\big).
  \end{aligned}
\end{equation*}
Applying the decomposition (\ref{hode3}), with (\ref{eta5}) and (\ref{eta1}), we obtain
\begin{equation}\label{com1}
  \begin{aligned}
    \int_{\Omega}\frac{\delta\mathcal{F}}{\delta v}\wedge\partial v=\int_{\Omega}\Big(\frac{\delta\mathcal{F}}{\delta v}\wedge d(\partial\phi)+\frac{\delta_{\beta}\mathcal{F}}{\delta v}\wedge\partial\eta-\ast dw\wedge\eta'-\ast dw\wedge [\eta,u]_{1}\Big).\\
  \end{aligned}
\end{equation}
Using the integration by parts formula (\ref{part}) and $\frac{\delta\mathcal{F}}{\delta v}\in P\Lambda^{n-1}(\Omega)$, which follows from (\ref{func1}) and implies $d(\frac{\delta\mathcal{F}}{\delta v})=0$, hence $\delta\ast(\frac{\delta\mathcal{F}}{\delta v})=0$, we have that
\begin{equation}\label{com2}
  \begin{aligned}
    \int_{\Omega}\frac{\delta\mathcal{F}}{\delta v}\wedge d(\partial\phi)=&\big\langle\ast\frac{\delta\mathcal{F}}{\delta v},d(\partial\phi)\big\rangle_{L^{2}\Lambda^{1}(\Omega)}=(-1)^{n-1}\int_{\partial\Omega}\mathrm{tr}(\frac{\delta\mathcal{F}}{\delta v})\wedge\partial\phi_{\partial}.\\
  \end{aligned}
\end{equation}
For $w\in H\Lambda^{0}(\Omega)$, as stated in Lemma \ref{lemma5}, and $\eta'\in P^{\ast}\Lambda^{1}(\Omega)$, we have using (\ref{part})
\begin{equation}\label{com3}
  \begin{aligned}
    \int_{\Omega}\ast dw\wedge\eta'=(-1)^{n-1}\langle w,\delta\eta'\rangle_{L^{2}\Lambda^{0}(\Omega)}+(-1)^{n-1}\int_{\partial\Omega}\mathrm{tr}(w)\wedge\ast\boldsymbol{n}(\eta')
    =0,
  \end{aligned}
\end{equation}
since $\ast\boldsymbol{n}(\eta')=0$ at $\partial\Omega$. Using (\ref{lie}), Lemma \ref{lemma2} and the fact that $\eta\in P^{\ast}\Lambda^{1}(\Omega)$, $u=dN_{\phi}(\partial\Sigma)$, we obtain
\begin{equation}\label{com4}
  \begin{aligned}
   \int_{\Omega}\ast dw\wedge [\eta,u]_{1}&=\int_{\Omega}dw\wedge\ast\delta(\eta\wedge u)=\int_{\partial\Omega}(\langle dw,\eta\rangle_{\Lambda^{1}}i_{\mathcal{N}}u)v_{\partial\Omega}.
  \end{aligned}
\end{equation}
Since $u=dN_{\phi}(\partial\Sigma)$ with $\partial\Sigma=0\text{ at }\Gamma$, we have using Lemma \ref{lemma6} that $\ast\boldsymbol{n}(u)=\partial\Sigma\text{ at }\Sigma$ and $\ast\boldsymbol{n}(u)=0\text{ at }\Gamma$.
From (\ref{rela4}) and the relation $v_{\Sigma}=i_{\mathcal{N}}v_{\Omega}|_{\Sigma}$, it follows that 
\begin{equation}\label{eq50}
  \ast v_{\Sigma}=\ast i_{\mathcal{N}}v_{\Omega}|_{\Sigma}=\ast\ast(\mathcal{N}^{\flat}\wedge\ast v_{\Omega})|_{\Sigma}=(-1)^{n-1}\mathcal{N}^{\flat},
\end{equation}
which, using $\partial\Sigma=\partial\Sigma_{v}v_{\Sigma}$, results in
\begin{equation}\label{eq60}
  \boldsymbol{n}(u)=(-1)^{n-1}\ast\partial\Sigma=(-1)^{n-1}\partial\Sigma_{v}(\ast v_{\Sigma})=(\partial\Sigma_{v})^{\flat}\quad\text{at }\Sigma.
\end{equation}
Finally, using (\ref{perp}), it follows that
\begin{equation}\label{eq61}
  (i_{\mathcal{N}}u)v_{\partial\Omega}=(i_{\mathcal{N}}u)v_{\Sigma}=u(\mathcal{N})v_{\Sigma}=\boldsymbol{n}(u)(\mathcal{N})v_{\Sigma}=(\partial\Sigma_{v})^{\flat}(\mathcal{N})v_{\Sigma}=\partial\Sigma,
\end{equation}
and (\ref{com4}) becomes
\begin{equation}\label{eq56}
  \int_{\Omega}\ast dw\wedge [\eta,u]_{1}=\int_{\Sigma}\langle dw,\eta\rangle_{\Lambda^{1}}\wedge\partial\Sigma.
\end{equation}
From (\ref{hami1}), using the chain rule, we now obtain that
\begin{equation}\label{com5}
  \int_{\Omega}\frac{\delta\mathcal{F}}{\delta v}\wedge\partial v+\int_{\Sigma}\frac{\delta\mathcal{F}}{\delta\Sigma}\wedge\partial\Sigma=\int_{\Omega}\frac{\delta\tilde{\mathcal{F}}}{\delta\eta}\wedge\partial\eta+\int_{\partial\Omega}\frac{\delta\tilde{\mathcal{F}}}{\delta\phi_{\partial}}\wedge\partial\phi_{\partial}+\int_{\Sigma}\frac{\delta\tilde{\mathcal{F}}}{\delta\Sigma}\wedge\partial\Sigma.
\end{equation}
Substituting (\ref{com1})-(\ref{com3}) and (\ref{eq56}) into (\ref{com5}), we have
\begin{equation*}
  \begin{aligned}
    &\int_{\Omega}\frac{\delta_{\beta}\mathcal{F}}{\delta v}\wedge\partial\eta+(-1)^{n-1}\int_{\partial\Omega}\mathrm{tr}(\frac{\delta\mathcal{F}}{\delta v})\wedge\partial\phi_{\partial}+\int_{\Sigma}(\frac{\delta\mathcal{F}}{\delta\Sigma}-\langle dw,\eta\rangle_{\Lambda^{1}})\wedge\partial\Sigma\\
    =&\int_{\Omega}\frac{\delta\tilde{\mathcal{F}}}{\delta\eta}\wedge\partial\eta+\int_{\partial\Omega}\frac{\delta\tilde{\mathcal{F}}}{\delta\phi_{\partial}}\wedge\partial\phi_{\partial}+\int_{\Sigma}\frac{\delta\tilde{\mathcal{F}}}{\delta\Sigma}\wedge\partial\Sigma,
  \end{aligned}
\end{equation*}
which implies using (\ref{extraterm}) that (\ref{eq103})--(\ref{eq105}) hold. Finally, (\ref{func4}) follows using (\ref{hode3})--(\ref{extraterm}).
\end{proof}

Using the fact that $\phi$ is a harmonic function, the Hamiltonian $H(v,\Sigma)$, stated in (\ref{total2}), can be expressed as $\tilde{H}:\mathring{\mathfrak{B}}^{\ast 1}\times H^{\frac{1}{2}}\Lambda^{0}(\partial\Omega)\times H^{-\frac{1}{2}}\Lambda^{n-1}(\Sigma)\to\mathbb{R}$,
\begin{equation}\label{hami4}
  \tilde{H}(\eta,\phi_{\partial},\Sigma)=\int_{\Omega}\big(\frac{1}{2}\llangle\eta^{\sharp},\eta^{\sharp}\rrangle+\Phi\big)v_{\Omega}+\frac{1}{2}\int_{\partial\Omega}\mathrm{tr}(\phi)\wedge\ast\boldsymbol{n}(d\phi)+\frac{\tau}{\rho}\int_{\Sigma}v_{\Sigma},
\end{equation}
with $\ast\boldsymbol{n}(d\phi)$ given by (\ref{potential1}).
The functional derivatives of $\tilde{H}$ with respect to the $(\eta,\phi_{\partial},\Sigma)$ variables can be obtained using (\ref{eq103})--(\ref{eq105}), together with (\ref{h1})--(\ref{h2}),
\begin{subequations}\label{func7}
  \begin{empheq}[left=\empheqlbrace]{align}
    &\frac{\delta\tilde{H}}{\delta\eta}=(-1)^{n-1}\ast\eta,\label{eq106}\\
    &\frac{\delta\tilde{H}}{\delta\phi_{\partial}}=\ast\boldsymbol{n}(d\phi),\label{eq107}\\
    &\frac{\delta\tilde{H}}{\delta\Sigma}=\mathrm{tr}\big(\frac{1}{2}\llangle v^{\sharp},v^{\sharp}\label{eq108}\rrangle+\Phi\big)+\frac{\tau k}{\rho}+(-1)^{n-1}\langle d\phi,\eta\rangle_{\Lambda^{1}}.
  \end{empheq}
\end{subequations}

\subsubsection{Hamiltonian formulation with respect to $(\eta,\phi_{\partial},\Sigma)$ variables}

Using Lemma \ref{lemma7}, we can express the Poisson bracket (\ref{bracket3}) in terms of the $(\eta,\phi_{\partial},\Sigma)$ variables, which is stated in the following lemma.

\begin{lemma}\label{lemma3}
Let $\tilde{\mathcal{F}},\tilde{\mathcal{G}}:\mathring{\mathfrak{B}}^{\ast 1}\times H^{\frac{1}{2}}\Lambda^{0}(\partial\Omega)\times H^{-\frac{1}{2}}\Lambda^{n-1}(\Sigma)\to\mathbb{R}$
be arbitrary functionals. The Poisson bracket (\ref{bracket3}) in the $(\eta,\phi_{\partial},\Sigma)$ variables is equal to
\begin{equation}\label{bracket5}
  \begin{aligned}
    \{\tilde{\mathcal{F}},\tilde{\mathcal{G}}\}(\eta,\phi_{\partial},\Sigma)=&(-1)^{n-1}\int_{\Omega}(\ast d\eta)\wedge\big(\ast\frac{\delta\tilde{\mathcal{G}}}{\delta\eta}+dN_{\phi}(\frac{\delta\tilde{\mathcal{G}}}{\delta\phi_{\partial}})\big)\wedge\big(\ast\frac{\delta\tilde{\mathcal{F}}}{\delta\eta}+dN_{\phi}(\frac{\delta\tilde{\mathcal{F}}}{\delta\phi_{\partial}})\big)\\
    &+\int_{\Sigma}\Big(\frac{\delta\tilde{\mathcal{F}}}{\delta\Sigma}+(-1)^{n}\langle dN_{\phi}\big(\frac{\delta\tilde{\mathcal{F}}}{\delta\phi_{\partial}}\big),\eta\rangle_{\Lambda^{1}}\Big)\wedge\frac{\delta\tilde{\mathcal{G}}}{\delta\phi_{\partial}}\\
    &-\int_{\Sigma}\Big(\frac{\delta\tilde{\mathcal{G}}}{\delta\Sigma}+(-1)^{n}\langle dN_{\phi}\big(\frac{\delta\tilde{\mathcal{G}}}{\delta\phi_{\partial}}\big),\eta\rangle_{\Lambda^{1}}\Big)\wedge\frac{\delta\tilde{\mathcal{F}}}{\delta\phi_{\partial}}.                 
  \end{aligned}
  \end{equation}
\end{lemma}

\begin{proof}
  After substituting (\ref{func4}) into the Poisson bracket (\ref{bracket3}) the result is immediate.
\end{proof}
\noindent
The Hamiltonian system for the Euler equations with a free surface (\ref{dieuler}) in terms of the $(\eta,\phi_{\partial},\Sigma)$ variables is now stated in the following theorem.

\begin{theorem}\label{thm4}
  Given a simply-connected, oriented Lipschitz continuous domain $\Omega$ with a nonoverlapping free surface $\Sigma\subseteq\partial\Omega$ and fixed boundary $\Gamma\subseteq\partial\Omega$, $\Sigma\cup\Gamma=\partial\Omega$.
  For any functional $\tilde{\mathcal{F}}(\eta,\phi_{\partial},\Sigma):\mathring{\mathfrak{B}}^{\ast 1}\times H^{\frac{1}{2}}\Lambda^{0}(\partial\Omega)\times H^{-\frac{1}{2}}\Lambda^{n-1}(\Sigma)\to\mathbb{R}$, and Hamiltonian functional $\tilde{H}$ (\ref{hami4}), the incompressible Euler equations with a free surface in the solenoidal velocity-potential-free surface variables $(\eta,\phi_{\partial},\Sigma)$ (\ref{vor4})--(\ref{vor6}), with $g=0$ at $\Gamma$, can be expressed as the Hamiltonian system
  \begin{equation}\label{eq33}
    \dot{\tilde{\mathcal{F}}}(\eta,\phi_{\partial},\Sigma)=\{\tilde{\mathcal{F}},\tilde{H}\}(\eta,\phi_{\partial},\Sigma).
  \end{equation}
\end{theorem}

\begin{proof}
  Inserting the Hamiltonian $\tilde{H}$ (\ref{hami4}) into the bracket (\ref{bracket5}), we have
  \begin{equation*}
    \begin{aligned}
      \{\tilde{\mathcal{F}},\tilde{H}\}(\eta,\phi_{\partial},\Sigma)=&(-1)^{n-1}\int_{\Omega}(\ast d\eta)\wedge\big(\ast\frac{\delta\tilde{H}}{\delta\eta}+dN_{\phi}(\frac{\delta\tilde{H}}{\delta\phi_{\partial}})\big)\wedge\big(\ast\frac{\delta\tilde{\mathcal{F}}}{\delta\eta}+dN_{\phi}(\frac{\delta\tilde{\mathcal{F}}}{\delta\phi_{\partial}})\big)\\
      &+\int_{\Sigma}\Big(\frac{\delta\tilde{\mathcal{F}}}{\delta\Sigma}+(-1)^{n}\langle dN_{\phi}\big(\frac{\delta\tilde{\mathcal{F}}}{\delta\phi_{\partial}}\big),\eta\rangle_{\Lambda^{1}}\Big)\wedge\frac{\delta\tilde{H}}{\delta\phi_{\partial}}\\
      &-\int_{\Sigma}\Big(\frac{\delta\tilde{H}}{\delta\Sigma}+(-1)^{n}\langle dN_{\phi}\big(\frac{\delta\tilde{H}}{\delta\phi_{\partial}}\big),\eta\rangle_{\Lambda^{1}}\Big)\wedge\frac{\delta\tilde{\mathcal{F}}}{\delta\phi_{\partial}}\\
      :=&T_{1}+T_{2}+T_{3}. 
    \end{aligned}
    \end{equation*}

    If (\ref{dieuler}) holds with $g=0$ at $\Gamma$, we can compute from (\ref{func3}), (\ref{func4}), and (\ref{func7}), the terms $T_{1}$, $T_{2}$ and $T_{3}$ as follows
    \begin{align*}
      T_{1}=&(-1)^{n-1}\int_{\Omega}(\ast d\eta)\wedge v\wedge\ast\big(\frac{\delta\tilde{\mathcal{F}}}{\delta\eta}+(-1)^{n-1}\ast dN_{\phi}(\frac{\delta\tilde{\mathcal{F}}}{\delta\phi_{\partial}})\big)\\
      =&-\int_{\Omega}\big(\frac{\delta\tilde{\mathcal{F}}}{\delta\eta}+(-1)^{n-1}\ast dN_{\phi}(\frac{\delta\tilde{\mathcal{F}}}{\delta\phi_{\partial}})\big)\wedge i_{v^{\sharp}}dv,\\
      T_{2}=&\int_{\Sigma}\Big(\frac{\delta\tilde{\mathcal{F}}}{\delta\Sigma}+(-1)^{n}\langle dN_{\phi}\big(\frac{\delta\tilde{\mathcal{F}}}{\delta\phi_{\partial}}\big),\eta\rangle_{\Lambda^{1}}\Big)\wedge\ast\boldsymbol{n}(d\phi),\\
      T_{3}=&-\int_{\Sigma}\big(\mathrm{tr}(\frac{1}{2}\llangle v^{\sharp},v^{\sharp}\rrangle +\Phi)+\frac{\tau k}{\rho}\big)\wedge(-1)^{n-1}\mathrm{tr}(\frac{\delta\mathcal{F}}{\delta v})\\
      =&-\int_{\Omega}\big(\frac{\delta\tilde{\mathcal{F}}}{\delta\eta}+(-1)^{n-1}\ast dN_{\phi}(\frac{\delta\tilde{\mathcal{F}}}{\delta\phi_{\partial}})\big)\wedge d\big(\frac{1}{2}\llangle v^{\sharp},v^{\sharp}\rrangle+\Phi+\frac{\tilde{p}}{\rho}\big),
    \end{align*}
    where for $T_{3}$ we used $\frac{\delta\mathcal{F}}{\delta v}\in\mathring{P}\Lambda^{n-1}(\Omega)$ since $v\in\mathring{P}^{\ast}\Lambda^{1}(\Omega)$, (\ref{func4}) and (\ref{eq101}), which follows from (\ref{dieuler4}).
    Collecting all terms gives
    \begin{equation}\label{eq31}
      \begin{aligned}
        \{\tilde{\mathcal{F}},\tilde{H}\}(\eta,\phi_{\partial},\Sigma)=&\int_{\Omega}\big(\frac{\delta\tilde{\mathcal{F}}}{\delta\eta}+(-1)^{n-1}\ast dN_{\phi}(\frac{\delta\tilde{\mathcal{F}}}{\delta\phi_{\partial}})\big)\wedge\big(-i_{v^{\sharp}}dv-d(\frac{1}{2}\llangle v^{\sharp},v^{\sharp}\rrangle+\frac{\tilde{p}}{\rho}+\Phi)\big)\\
        &+\int_{\Sigma}\big(\frac{\delta\tilde{\mathcal{F}}}{\delta\Sigma}+(-1)^{n}\langle dN_{\phi}(\frac{\delta\tilde{\mathcal{F}}}{\delta\phi_{\partial}}),\eta\rangle_{\Lambda^{1}}\big)\wedge\ast\boldsymbol{n}(d\phi)\\
        =&\int_{\Omega}\big(\frac{\delta\tilde{\mathcal{F}}}{\delta\eta}+(-1)^{n}\ast dN_{\phi}(\frac{\delta\tilde{\mathcal{F}}}{\delta\phi_{\partial}})\big)\wedge v_{t}\\
        &+\int_{\Sigma}\big(\frac{\delta\tilde{\mathcal{F}}}{\delta\Sigma}+(-1)^{n}\langle dN_{\phi}(\frac{\delta\tilde{\mathcal{F}}}{\delta\phi_{\partial}}),\eta\rangle_{\Lambda^{1}}\big)\wedge\Sigma_{t}\\
        =&\dot{\tilde{\mathcal{F}}}(\eta,\phi_{\partial},\Sigma),
      \end{aligned}
    \end{equation}
    which gives (\ref{eq33}). Note in the last step we used the functional chain rule with $\frac{\delta\mathcal{F}}{\delta v}$ and $\frac{\delta\mathcal{F}}{\delta\Sigma}$ replaced by, respectively, (\ref{func4}) and (\ref{eq105}), and (\ref{dieuler1}), (\ref{dieuler3}).

    Next, we assume that (\ref{eq33}) holds. If we use (\ref{eq105}) and (\ref{func4}) for the functional derivatives of $\tilde{\mathcal{F}}$ in (\ref{eq31}), then the proof is identical to the second part of Theorem \ref{thm3}.    
  \end{proof}

  In Section 5.3 we will extend the Hamiltonian system stated in Theorem \ref{thm4} to a port-Hamiltonian system which also incorporates the inhomogeneous boundary condition (\ref{dieuler5}) with $g\neq 0$.

\subsection{Vorticity-potential-free surface Hamiltonian formulation}

The objective of this section is to derive the Hamiltonian formulation of the Euler equations with a free surface in terms of the vorticity $\omega$, the potential function $\phi_{\partial}$ and the free surface $\Sigma$.
  
We start with defining the vorticity $\omega=dv\in\mathring{V}\Lambda^{2}(\Omega)$. After applying the exterior derivative $d$ to both sides of (\ref{dieuler1}) and using (\ref{rela4}) we obtain the vorticity equation 
\begin{equation}\label{incom2}
\omega_{t}+ d\ast(v\wedge\ast\omega)=0.
\end{equation}
Therefore, the dynamic equations for the incompressible Euler equations with a free surface in the $(\omega,\phi_{\partial},\Sigma)$ variables are
\begin{subequations}\label{vor1}
  \begin{empheq}[left=\empheqlbrace]{align}
    &\omega_{t}+d\ast\big(v\wedge\ast\omega\big)=0\text{ in }\Omega\qquad\qquad\qquad\qquad\qquad\qquad\qquad\qquad \ ,\\
    &d\omega=0\text{ in }\Omega,\\
    &\ast\boldsymbol{n}(\omega)=0\text{ at }\Gamma\cup\Sigma,
  \end{empheq}
\end{subequations}
\begin{subequations}\label{vor2}
  \begin{empheq}[left=\empheqlbrace]{align}
    &d\phi_{t}=-\ast(v\wedge\ast\omega)-d\big(\frac{1}{2}\llangle v^{\sharp},v^{\sharp}\rrangle+\frac{\tilde{p}}{\rho}+\Phi\big)-[\delta N_{\beta}(\omega),u']_{1}\text{ in }\Omega,\\
    &\delta d\phi=0\text{ in }\Omega,\\
    &\ast\boldsymbol{n}(d\phi)=\Sigma_{t}\text{ at }\Sigma,\\
    &\ast\boldsymbol{n}(d\phi)=g\text{ at }\Gamma,\\
    &\mathrm{tr}(\frac{\tilde{p}}{\rho})=\frac{\tau k}{\rho}\text{ at }\Sigma,
  \end{empheq}
\end{subequations}
with
\begin{equation}\label{vor3}
  v=d\phi+\delta N_{\beta}(\omega)\quad\text{in }\Omega,
\end{equation}
and $u'=dN_{\phi}(\Sigma_{t})$, with $\Sigma_{t}=0$ at $\Gamma$.

\subsubsection{Functional derivatives with respect to $(\omega,\phi_{\partial},\Sigma)$ variables}
Our aim is now to formulate the Hamiltonian formulation in terms of the $(\omega,\phi_{\partial},\Sigma)$ variables. We start by defining the functional $\bar{\mathcal{F}}(\omega,\phi_{\partial},\Sigma):\mathring{V}\Lambda^{2}(\Omega)\times H^{\frac{1}{2}}\Lambda^{0}(\partial\Omega)\times H^{-\frac{1}{2}}\Lambda^{n-1}(\Sigma)\to\mathbb{R}$ as
\begin{equation*}
  \bar{\mathcal{F}}(\omega,\phi_{\partial},\Sigma)=\tilde{\mathcal{F}}(\eta,\phi_{\partial},\Sigma)=\mathcal{F}(v,\Sigma).
\end{equation*}

\begin{defn}
  The functional derivative $\frac{\delta\bar{\mathcal{F}}}{\delta\omega}\in\mathring{V}^{\ast}\Lambda^{n-2}(\Omega)$ is defined as
  \begin{equation*}
    \int_{\Omega}\frac{\delta\bar{\mathcal{F}}}{\delta\omega}\wedge\partial\omega=\lim_{\epsilon\to 0}\frac{1}{\epsilon}\big(\bar{\mathcal{F}}(\omega+\epsilon\partial\omega)-\bar{\mathcal{F}}(\omega)\big),\ \forall\partial\omega\in\mathring{V}\Lambda^{2}(\Omega).
  \end{equation*}
\end{defn}

Since $\omega$ is independent of $\phi_{\partial}$ and $\Sigma$, it's obvious that
\begin{equation}\label{func5}
  \frac{\delta\bar{\mathcal{F}}}{\delta\phi_{\partial}}=\frac{\delta\tilde{\mathcal{F}}}{\delta\phi_{\partial}},\quad \frac{\delta\bar{\mathcal{F}}}{\delta\Sigma}=\frac{\delta\tilde{\mathcal{F}}}{\delta\Sigma}=\frac{\delta\mathcal{F}}{\delta\Sigma}+(-1)^{n-1}\langle dN_{\phi}(\frac{\delta\tilde{\mathcal{F}}}{\delta\phi_{\partial}}),\eta\rangle_{\Lambda^{1}}.
\end{equation}

We now present the relations between the functional derivatives in terms of $\omega$ and $\eta$.

\begin{lemma}\label{lemma12}
  For an arbitrary an functional $\bar{\mathcal{F}}(\omega,\phi_{\partial},\Sigma):\mathring{V}\Lambda^{2}(\Omega)\times H^{\frac{1}{2}}\Lambda^{0}(\partial\Omega)\times$$H^{-\frac{1}{2}}\Lambda^{n-1}(\Sigma)$ $\to\mathbb{R}$, we have
  \begin{equation}\label{func6}
    \frac{\delta\tilde{\mathcal{F}}}{\delta\eta}=(-1)^{n-1}d\big(\frac{\delta\bar{\mathcal{F}}}{\delta\omega}\big).
  \end{equation}
  Furthermore, we have
  \begin{equation}\label{fun8}
    \frac{\delta\mathcal{F}}{\delta v}=(-1)^{n-1}d(\frac{\delta\bar{\mathcal{F}}}{\delta\omega})+(-1)^{n-1}\ast dN_{\phi}(\frac{\delta\bar{\mathcal{F}}}{\delta\phi_{\partial}}).
  \end{equation}
\end{lemma}

\begin{proof}
  Using the relation $\omega=dv=d\eta$ and the chain rule, we have used the integration by parts formula (\ref{part})
  \begin{equation*}
    \begin{aligned}
     \int_{\Omega}\frac{\delta\tilde{\mathcal{F}}}{\delta\eta}\wedge\partial\eta&=\int_{\Omega}\frac{\delta\bar{\mathcal{F}}}{\partial\omega}\wedge\partial\omega=\int_{\Omega}\frac{\delta\bar{\mathcal{F}}}{\partial\omega}\wedge d(\partial\eta)=\langle d(\partial\eta),\ast\frac{\delta\bar{\mathcal{F}}}{\delta\omega}\rangle_{L^{2}\Lambda^{2}(\Omega)}\\
      &=\langle\delta\ast\frac{\delta\bar{\mathcal{F}}}{\delta\omega},\partial\eta\rangle_{L^{2}\Lambda^{1}(\Omega)}+\int_{\partial\Omega}\mathrm{tr}(\partial\eta)\wedge\mathrm{tr}(\frac{\delta\bar{\mathcal{F}}}{\delta\omega}).
    \end{aligned}
  \end{equation*}
  Since $\frac{\delta\bar{\mathcal{F}}}{\delta\omega}\in\mathring{V}^{\ast}\Lambda^{n-2}(\Omega)$, it follows that $\mathrm{tr}(\frac{\delta\bar{\mathcal{F}}}{\delta\omega})=0$. Thus,
  \begin{equation*}
    \begin{aligned}
      \int_{\Omega}\frac{\delta\tilde{\mathcal{F}}}{\delta\eta}\wedge\partial\eta=\int_{\Omega}\partial\eta\wedge\ast\delta\ast\big(\frac{\delta\bar{\mathcal{F}}}{\delta\omega}\big)=(-1)^{n-1}\int_{\Omega}d\big(\frac{\delta\bar{\mathcal{F}}}{\delta\omega}\big)\wedge\partial\eta.
    \end{aligned}
  \end{equation*}
  Hence,
  \begin{equation}\label{hami5}
    \frac{\delta\tilde{\mathcal{F}}}{\delta\eta}=(-1)^{n-1}d\big(\frac{\delta\bar{\mathcal{F}}}{\delta\omega}\big).
  \end{equation}
  Substituting (\ref{hami5}) into (\ref{func4}) and using (\ref{func5}), we immediately obtain (\ref{fun8}).
\end{proof}
  
\begin{lemma}\label{lemma13}
  The Hamiltonian functional $\bar{H}(\omega,\phi_{\partial},\Sigma):\mathring{V}\Lambda^{2}(\Omega)\times H^{\frac{1}{2}}\Lambda^{0}(\partial\Omega)$ $\times H^{-\frac{1}{2}}\Lambda^{n-1}(\Sigma)$ $\to\mathbb{R}$ is given by
  \begin{equation}\label{hami7}
    \bar{H}(\omega,\phi_{\partial},\Sigma)=\frac{1}{2}\int_{\Omega}\beta\wedge\ast\omega+\int_{\Omega}\Phi v_{\Omega}+\frac{1}{2}\int_{\partial\Omega}\mathrm{tr}(\phi)\wedge\ast\boldsymbol{n}(d\phi)+\frac{\tau}{\rho}\int_{\Sigma}v_{\Sigma}.
  \end{equation}
  The functional derivatives in terms of $(\omega,\phi_{\partial},\Sigma)$ are
  \begin{subequations}\label{eq120}
    \begin{empheq}[left=\empheqlbrace]{align}
      &\frac{\delta\bar{H}}{\delta\omega}=\ast\beta,\label{eq114}\\
      &\frac{\delta\bar{H}}{\delta\phi_{\partial}}=\ast\boldsymbol{n}(d\phi),\label{eq115}\\
      &\frac{\delta\bar{H}}{\delta\Sigma}=\mathrm{tr}\big(\frac{1}{2}\llangle v^{\sharp},v^{\sharp}\rrangle+\Phi\big)+\frac{\tau k}{\rho}+(-1)^{n-1}\langle d\phi,\eta\rangle_{\Lambda^{1}}.\label{eq116}
    \end{empheq}
  \end{subequations}
\end{lemma}

  \begin{proof}
    With (\ref{hami4}), and after integration by parts of the first term in (\ref{hami4}) using (\ref{stream}), we obtain (\ref{hami7}).
    Introducing $\bar{H}$ into (\ref{func6}) and using (\ref{eq106}), we obtain that $\frac{\delta\bar{H}}{\delta\omega}$ satisfies
  \begin{equation*}
    d\frac{\delta\bar{H}}{\delta\omega}=(-1)^{n-1}\frac{\delta\tilde{H}}{\delta\eta}=\ast\eta=\ast\delta\beta=d(\ast\beta),
  \end{equation*}
  with the boundary condition
  \begin{equation*}
    \mathrm{tr}(\frac{\delta\bar{H}}{\delta\omega})=0.
  \end{equation*} 
  Thus, we obtain (\ref{eq114}). Finally, (\ref{eq115}) and (\ref{eq116}) hold directly from (\ref{eq107}), (\ref{eq108}) and (\ref{func5}).
\end{proof}

\subsubsection{Poisson bracket with respect to $(\omega,\phi_{\partial},\Sigma)$ variables}
The Poisson bracket for the Hamiltonian formulation in terms of the $(\omega,\phi_{\partial},\Sigma)$ variables are stated in the following lemma.

\begin{lemma}\label{lemma9}
  Let $\bar{\mathcal{F}},\bar{\mathcal{G}}:\mathring{V}\Lambda^{2}(\Omega)\times H^{\frac{1}{2}}\Lambda^{0}(\partial\Omega)\times H^{-\frac{1}{2}}\Lambda^{n-1}(\Sigma)\to\mathbb{R}$.
  The Poisson bracket with respect to $(\omega,\phi_{\partial},\Sigma)$ can be expressed as
  \begin{equation}\label{bracket6}
    \begin{aligned}
      &\{\bar{\mathcal{F}},\bar{\mathcal{G}}\}(\omega,\phi_{\partial},\Sigma)\\
      =&(-1)^{n-1}\int_{\Omega}(\ast\omega)\wedge\Big((-1)^{n-1}\ast d\frac{\delta\bar{\mathcal{G}}}{\delta\omega}+dN_{\phi}(\frac{\delta\bar{\mathcal{G}}}{\delta\phi_{\partial}})\big)
      \wedge\big((-1)^{n-1}\ast d\frac{\delta\bar{\mathcal{F}}}{\delta\omega}+dN_{\phi}(\frac{\delta\bar{\mathcal{F}}}{\delta\phi_{\partial}})\Big)\\
      &+\int_{\Sigma}\Big(\frac{\delta\bar{\mathcal{F}}}{\delta\Sigma}+(-1)^{n}\langle dN_{\phi}(\frac{\delta\bar{\mathcal{F}}}{\delta\phi_{\partial}}),\delta N_{\beta}(\omega)\rangle_{\Lambda^{1}} \Big)\wedge\frac{\delta\bar{\mathcal{G}}}{\delta\phi_{\partial}}\\
      &-\int_{\Sigma}\Big(\frac{\delta\bar{\mathcal{G}}}{\delta\Sigma}+(-1)^{n}\langle dN_{\phi}(\frac{\delta\bar{\mathcal{G}}}{\delta\phi_{\partial}}),\delta N_{\beta}(\omega)\rangle_{\Lambda^{1}} \Big)\wedge\frac{\delta\bar{\mathcal{F}}}{\delta\phi_{\partial}}.
    \end{aligned}
  \end{equation}
\end{lemma}

\begin{proof}
After substituting (\ref{func5}) and (\ref{func6}) into the Poisson bracket (\ref{bracket5}) and using (\ref{eq81}), which gives $\eta=\delta N_{\beta}(\omega)$, the result is immediate.
\end{proof}

\begin{theorem}\label{thm1}
  Given a simply-connected, oriented Lipschitz continuous domain $\Omega$ with a nonoverlapping free surface $\Sigma\subseteq\partial\Omega$ and fixed boundary $\Gamma\subseteq\partial\Omega$, $\Sigma\cup\Gamma=\partial\Omega$.
  For any functionals $\bar{\mathcal{F}}(\omega,\phi_{\partial},\Sigma):\mathring{V}\Lambda^{2}(\Omega)\times H^{\frac{1}{2}}\Lambda^{0}(\partial\Omega)\times H^{-\frac{1}{2}}\Lambda^{n-1}(\Sigma)\to\mathbb{R}$ and Hamiltonian functional $\bar{H}$ (\ref{hami7}), the incompressible Euler equations with a free surface in the vorticity-potential-free surface variables $(\omega,\phi_{\partial},\Sigma)$ (\ref{vor1})--(\ref{vor3}), with $g=0$ at $\Gamma$, can be expressed as the Hamiltonian system
  \begin{equation}\label{eq65}
    \dot{\bar{\mathcal{F}}}(\omega,\phi_{\partial},\Sigma)=\{\bar{\mathcal{F}},\bar{H}\}(\omega,\phi_{\partial},\Sigma).
  \end{equation}
\end{theorem}

\begin{proof}
  First, we will prove that (\ref{vor1})--(\ref{vor3}) imply (\ref{eq65}).
  Inserting the functional derivatives of the Hamiltonian $\bar{H}$ (\ref{eq120}) into the Poisson bracket (\ref{bracket6}) gives
  \begin{equation*}
   \begin{aligned}
      &\{\bar{\mathcal{F}},\bar{H}\}(\omega,\phi_{\partial},\Sigma)\\
      =&(-1)^{n-1}\int_{\Omega}(\ast\omega)\wedge\big((-1)^{n-1}\ast d\frac{\delta\bar{H}}{\delta\omega}+dN_{\phi}(\frac{\delta\bar{H}}{\delta\phi_{\partial}})\big)
      \wedge\big((-1)^{n-1}\ast d\frac{\delta\bar{\mathcal{F}}}{\delta\omega}+dN_{\phi}(\frac{\delta\bar{\mathcal{F}}}{\delta\phi_{\partial}})\big)\\
      &+\int_{\Sigma}\Big(\frac{\delta\bar{\mathcal{F}}}{\delta\Sigma}+(-1)^{n}\langle dN_{\phi}\big(\frac{\delta\bar{\mathcal{F}}}{\delta\phi_{\partial}}\big),\delta N_{\beta}(\omega)\rangle_{\Lambda^{1}}\Big)\wedge\frac{\delta\bar{H}}{\delta\phi_{\partial}}\\
      &-\int_{\Sigma}\Big(\frac{\delta\bar{H}}{\delta\Sigma}+(-1)^{n}\langle dN_{\phi}\big(\frac{\delta\bar{H}}{\delta\phi_{\partial}}\big),\delta N_{\beta}(\omega)\rangle_{\Lambda^{1}}\Big)\wedge\frac{\delta\bar{\mathcal{F}}}{\delta\phi_{\partial}}\\
      :=&T_{1}+T_{2}+T_{3}. 
    \end{aligned}
  \end{equation*}
  The term $T_{1}$ can be computed using $\ast d\frac{\delta\bar{H}}{\delta\omega}=\ast d\ast\beta=(-1)^{n-1}\delta\beta$, $dN_{\phi}\big(\frac{\delta\bar{H}}{\delta\phi_{\partial}}\big)=dN_{\phi}\big(\ast\boldsymbol{n}(d\phi)\big)=d\phi$, 
  the Hodge decomposition (\ref{vhode}) and (\ref{part}), yielding
  \begin{equation*}
    \begin{aligned}
      T_{1}=&(-1)^{n-1}\int_{\Omega}(\ast\omega)\wedge v\wedge\big((-1)^{n-1}\ast d\frac{\delta\bar{\mathcal{F}}}{\delta\omega}+dN_{\phi}(\frac{\delta\bar{\mathcal{F}}}{\delta\phi_{\partial}})\big)\\
        =&\int_{\Omega}\frac{\delta\bar{\mathcal{F}}}{\delta\omega}\wedge\ast\delta\big((\ast\omega)\wedge v\big)+(-1)^{n-1}\int_{\Omega}(\ast\omega)\wedge v\wedge dN_{\phi}(\frac{\delta\bar{\mathcal{F}}}{\delta\phi_{\partial}})\\
        =&-\int_{\Omega}\frac{\delta\bar{\mathcal{F}}}{\delta\omega}\wedge d\ast(v\wedge\ast\omega)+(-1)^{n}\int_{\Omega}\ast dN_{\phi}(\frac{\delta\bar{\mathcal{F}}}{\delta\phi_{\partial}})\wedge\ast(v\wedge\ast\omega).
    \end{aligned}
  \end{equation*}
  The term $T_{2}$ can be computed using (\ref{eq115}), resulting in
  \begin{equation*}
    T_{2}=\int_{\Sigma}\big(\frac{\delta\bar{\mathcal{F}}}{\delta\Sigma}+(-1)^{n}\langle dN_{\phi}(\frac{\delta\bar{\mathcal{F}}}{\delta\phi_{\partial}}),\delta N_{\beta}(\omega)\rangle_{\Lambda^{1}}\big)\wedge\ast\boldsymbol{n}(d\phi).
  \end{equation*}
  From (\ref{eq60}) and (\ref{eq61}), we have that
  \begin{equation}\label{eq69}
    \ast\boldsymbol{n}(d\phi)=\Sigma_{t}=(i_{\mathcal{N}}u')v_{\Sigma}\text{ at }\Sigma.
  \end{equation}
  Using (\ref{eq69}), Lemma 2.4, which is possible since $\delta N_{\beta}(\omega)$, $u'\in P^{\ast}\Lambda^{1}(\Omega)$ and, using (\ref{proper2}), (\ref{stream}), $\boldsymbol{n}(\delta N_{\beta}(\omega))=\delta\boldsymbol{n}(\beta)=\delta\ast\mathrm{tr}(\ast\beta)=0$ at $\partial\Omega$, $\ast\boldsymbol{n}(u')=0$
  at $\Gamma$ and (\ref{lie}) we obtain
  \begin{equation*} 
    \begin{aligned}
      T_{2}=&\int_{\Sigma}\frac{\delta\bar{\mathcal{F}}}{\delta\Sigma}\wedge\ast\boldsymbol{n}(d\phi)+(-1)^{n}\int_{\partial\Omega}\langle dN_{\phi}(\frac{\delta\bar{\mathcal{F}}}{\delta\phi_{\partial}}),\delta N_{\beta}(\omega)\rangle_{\Lambda^{1}}i_{\mathcal{N}}(u')v_{\Sigma}\\
      =&\int_{\Sigma}\frac{\delta\bar{\mathcal{F}}}{\delta\Sigma}\wedge\ast\boldsymbol{n}(d\phi)+(-1)^{n}\int_{\Omega}dN_{\phi}(\frac{\delta\bar{\mathcal{F}}}{\delta\phi_{\partial}})\wedge\ast\delta\big(\delta N_{\beta}(\omega)\wedge u'\big)\\
      =&\int_{\Sigma}\frac{\delta\bar{\mathcal{F}}}{\delta\Sigma}\wedge\ast\boldsymbol{n}(d\phi)+(-1)^{n}\int_{\Omega}\ast dN_{\phi}(\frac{\delta\bar{\mathcal{F}}}{\delta\phi_{\partial}})\wedge[\delta N_{\beta}(\omega),u']_{1}.
    \end{aligned}
  \end{equation*}
  The term $T_{3}$ can be computed using (\ref{eq115}) and (\ref{eq116}), together with (\ref{vor2}e)
  \begin{equation*}
    \begin{aligned}
      T_{3}=&-\int_{\Sigma}\big(\mathrm{tr}(\frac{1}{2}\llangle v^{\sharp},v^{\sharp}\rrangle+\Phi)+\frac{\tau k}{\rho}\big)\wedge\frac{\delta\bar{\mathcal{F}}}{\delta\phi_{\partial}}\\
      =&(-1)^{n}\int_{\Sigma}\mathrm{tr}\big(\frac{1}{2}\llangle v^{\sharp},v^{\sharp}\rrangle+\Phi+\frac{\tilde{p}}{\rho}\big)\wedge\mathrm{tr}(\frac{\delta\mathcal{F}}{\delta v}).
    \end{aligned}
  \end{equation*}
  Since $\frac{\delta\mathcal{F}}{\delta v}\in\mathring{P}\Lambda^{n-1}(\Omega)$, which implies $d\big(\frac{\delta\mathcal{F}}{\delta v}\big)=0$ in $\Omega$ and $\mathrm{tr}\big(\frac{\delta\mathcal{F}}{\delta v}\big)=0$
  at $\Gamma$, we obtain using Stokes' theorem and (\ref{fun8})
  \begin{equation*}
    \begin{aligned}
      T_{3}=&-\int_{\Omega}\frac{\delta\mathcal{F}}{\delta v}\wedge d\big(\frac{1}{2}\llangle v^{\sharp},v^{\sharp}\rrangle+\Phi+\frac{\tilde{p}}{\rho}\big)\\
      =&(-1)^{n}\int_{\Omega}\Big(d\big(\frac{\delta\bar{\mathcal{F}}}{\delta\omega}\big)+\ast dN_{\phi}\big(\frac{\delta\bar{\mathcal{F}}}{\delta\phi_{\partial}}\big)\Big)\wedge d\big(\frac{1}{2}\llangle v^{\sharp},v^{\sharp}\rrangle+\Phi+\frac{\tilde{p}}{\rho}\big).
    \end{aligned}
  \end{equation*}
  Consider now
  \begin{equation*}
    \begin{aligned}
      \int_{\Omega}d\big(\frac{\delta\bar{\mathcal{F}}}{\delta\omega}\big)\wedge d\big(\frac{1}{2}\llangle v^{\sharp},v^{\sharp}\rrangle+\Phi+\frac{\tilde{p}}{\rho}\big)
      =&(-1)^{n-1}\langle d\big(\frac{\delta\bar{\mathcal{F}}}{\delta\omega}\big),\ast d\big(\frac{1}{2}\llangle v^{\sharp},v^{\sharp}\rrangle+\Phi+\frac{\tilde{p}}{\rho}\big)\rangle_{L^{2}\Lambda^{n-1}(\Omega)}\\
      =&(-1)^{n-1}\langle \frac{\delta\bar{\mathcal{F}}}{\delta\omega},\delta\ast d\big(\frac{1}{2}\llangle v^{\sharp},v^{\sharp}\rrangle+\Phi+\frac{\tilde{p}}{\rho}\big)\rangle_{L^{2}\Lambda^{n-2}(\Omega)}\\
      =&0,
    \end{aligned}
  \end{equation*}
  where we used $\frac{\delta\bar{\mathcal{F}}}{\delta\omega}\in\mathring{V}^{\ast}\Lambda^{n-2}(\Omega)$ and 
  $\delta\ast d\big(\frac{1}{2}\llangle v^{\sharp},v^{\sharp}\rrangle+\Phi+\frac{\tilde{p}}{\rho}\big)=0$.
  Hence $T_{3}$ is equal to
  \begin{equation*}
    T_{3}=(-1)^{n}\int_{\Omega}\ast dN_{\phi}\big(\frac{\delta\bar{\mathcal{F}}}{\delta\phi_{\partial}}\big)\wedge d\big(\frac{1}{2}\llangle v^{\sharp},v^{\sharp}\rrangle+\Phi+\frac{\tilde{p}}{\rho}\big).
  \end{equation*}
  Collecting all terms and using (\ref{vor1}a), (\ref{vor2}a), (\ref{vor2}c) and the functional chain rule we obtain
  \begin{equation}\label{eq119}
    \begin{aligned}
      \{\bar{\mathcal{F}},\bar{H}\}(\omega,\phi_{\partial},\Sigma)=&\int_{\Omega}\frac{\delta\bar{\mathcal{F}}}{\delta\omega}\wedge\big(-d\ast(v\wedge\ast\omega)\big)+\int_{\Sigma}\frac{\delta\bar{\mathcal{F}}}{\delta\Sigma}\wedge\ast\boldsymbol{n}(d\phi)\\
      &+(-1)^{n-1}\int_{\Omega}\ast dN_{\phi}\big(\frac{\delta\bar{\mathcal{F}}}{\delta\phi_{\partial}}\big)\wedge\Big(-\ast(v\wedge\ast\omega)\\
      &-d(\frac{1}{2}\llangle v^{\sharp},v^{\sharp}\rrangle+\Phi+\frac{\tilde{p}}{\rho})-[\delta N_{\beta}(\omega),u']_{1}\Big)\\
      =&\int_{\Omega}\frac{\delta\bar{\mathcal{F}}}{\delta\omega}\wedge\omega_{t}+\int_{\partial\Omega}\frac{\delta\bar{\mathcal{F}}}{\delta\phi_{\partial}}\wedge(\phi_{\partial})_{t}+\int_{\Sigma}\frac{\delta\bar{\mathcal{F}}}{\delta\Sigma}\wedge\Sigma_{t}\\
      =&\dot{\bar{\mathcal{F}}}(\omega,\phi_{\partial},\Sigma),
    \end{aligned}
  \end{equation}
  which gives (\ref{eq65}). Here we used in the second step of (\ref{eq119}) the relation
  \begin{equation}\label{eq117}
    \begin{aligned}
      &(-1)^{n-1}\int_{\Omega}\ast dN_{\phi}(\frac{\delta\bar{\mathcal{F}}}{\delta\phi_{\partial}})\wedge d\phi_{t}=\int_{\Omega}d\phi_{t}\wedge\big(d(\frac{\delta\bar{\mathcal{F}}}{\delta\omega})+\ast dN_{\phi}(\frac{\delta\bar{\mathcal{F}}}{\delta\phi_{\partial}})\big)\\
      =&(-1)^{n-1}\int_{\Omega}d\phi_{t}\wedge\frac{\delta\mathcal{F}}{\delta v}=(-1)^{n-1}\int_{\partial\Omega}\mathrm{tr}(\phi_{t})\wedge\mathrm{tr}(\frac{\delta\mathcal{F}}{\delta v})\\
      =&\int_{\partial\Omega}\frac{\delta\bar{\mathcal{F}}}{\delta\phi_{\partial}}\wedge(\phi_{\partial})_{t},
    \end{aligned}
  \end{equation}
  since $\int_{\Omega}d\phi_{t}\wedge d\big(\frac{\delta\bar{\mathcal{F}}}{\delta\omega}\big)=\langle\frac{\delta\bar{\mathcal{F}}}{\delta\omega},\delta\ast d\phi_{t}\rangle_{L^{2}\Lambda^{n-2}(\Omega)}+(-1)^{n-1}\int_{\partial\Omega}\mathrm{tr}(\frac{\delta\bar{\mathcal{F}}}{\delta\omega})\wedge\mathrm{tr}(d\phi_{t})=0$ 
  for $\frac{\delta\bar{\mathcal{F}}}{\delta\omega}\in\mathring{V}^{\ast}\Lambda^{n-2}(\Omega)$ and $\delta\ast d\phi_{t}=\ast d^{2}\phi_{t}=0$.

  Next, we assume that (\ref{eq65}) holds. Using the same steps as in the evaluation of $T_{1}$, $T_{2}$ and $T_{3}$, but without using (\ref{vor2}c), we obtain that (\ref{eq65}) can be expressed as
  \begin{equation*}
    \begin{aligned}
      &\int_{\Omega}\frac{\delta\bar{\mathcal{F}}}{\delta\omega}\wedge\big(\omega_{t}+d\ast(v\wedge\ast\omega)\big)+\int_{\Sigma}\frac{\delta\bar{\mathcal{F}}}{\delta\Sigma}\wedge\big(\Sigma_{t}-\ast\boldsymbol{n}(d\phi)\big)\\
      +&(-1)^{n-1}\int_{\Omega}\ast dN_{\phi}(\frac{\delta\bar{\mathcal{F}}}{\delta\phi_{\partial}})\wedge\big(d\phi_{t}+\ast(v\wedge\ast\omega)+d(\frac{1}{2}\llangle v^{\sharp},v^{\sharp}\rrangle+\Phi)\big)\\
      +&[\delta N_{\beta}(\omega),u']_{1}\big)+\int_{\Sigma}\frac{\tau k}{\rho}\wedge\frac{\delta\bar{\mathcal{F}}}{\delta\phi_{\partial}}=0,
    \end{aligned}
  \end{equation*}
  where we used (\ref{eq117}) for the third term. Then by adding and subtracting $(-1)^{n-1}$ $\int_{\Omega}\ast dN_{\phi}(\frac{\delta\bar{\mathcal{F}}}{\delta\phi_{\partial}})\wedge d(\frac{\tilde{p}}{\rho})$, using the relation
  \begin{equation*}
    (-1)^{n-1}\int_{\Omega}\ast dN_{\phi}(\frac{\delta\bar{\mathcal{F}}}{\delta\phi})\wedge d(\frac{\tilde{p}}{\rho})=\int_{\Sigma}\mathrm{tr}(\frac{\tilde{p}}{\rho})\wedge\frac{\delta\bar{\mathcal{F}}}{\delta\phi_{\partial}},
  \end{equation*}
  which can be derived in the same way as (\ref{eq117}), we obtain
  \begin{equation*}
    \begin{aligned}
      &\int_{\Omega}\frac{\delta\bar{\mathcal{F}}}{\delta\omega}\wedge\big(\omega_{t}+d\ast(v\wedge\ast\omega)\big)+\int_{\Sigma}\frac{\delta\bar{\mathcal{F}}}{\delta\Sigma}\wedge\big(\Sigma_{t}-\ast\boldsymbol{n}(d\phi)\big)\\
      +&(-1)^{n-1}\int_{\Omega}\ast dN_{\phi}(\frac{\delta\bar{\mathcal{F}}}{\delta\phi_{\partial}})\wedge\Big(d\phi_{t}+\ast(v\wedge\ast\omega)+d(\frac{1}{2}\llangle v^{\sharp},v^{\sharp}\rrangle+\Phi+\frac{\tilde{p}}{\rho})+[\delta N_{\beta}(\omega),\eta]_{1}\Big)\\
      +&\int_{\Sigma}\big(\mathrm{tr}(\frac{\tilde{p}}{\rho})-\frac{\tau k}{\rho}\big)\wedge\frac{\delta\bar{\mathcal{F}}}{\delta\phi_{\partial}}=0.
    \end{aligned}
  \end{equation*}
  Taking $\frac{\delta\bar{\mathcal{F}}}{\delta\omega}\in\mathring{V}^{\ast}\Lambda^{n-2}(\Omega)$, $\frac{\delta\bar{\mathcal{F}}}{\delta\phi_{\partial}}\in H^{-\frac{1}{2}}\Lambda^{n-1}(\partial\Omega)$  
  and $\frac{\delta\bar{\mathcal{F}}}{\delta\Sigma}\in H^{\frac{1}{2}}\Lambda^{0}(\Sigma)$ arbitrary gives (\ref{vor1}a), (\ref{vor2}a), (\ref{vor2}c) and (\ref{vor2}e).
  The other conditions in (\ref{vor1})--(\ref{vor2}) are automatically satisfied due to the choice of the function spaces.
\end{proof}
  
\begin{corollary}\label{corollary1}
  Given a simply-connected, oriented Lipschitz domain $\Omega$ with a nonoverlapping free surface $\Sigma\subseteq\partial\Omega$ and fixed boundary $\Gamma\subseteq\partial\Omega$, $\Sigma\cup\Gamma=\partial\Omega$. For any functionals $\bar{\mathcal{F}}_{p}(\phi_{\partial},\Sigma):H^{\frac{1}{2}}\Lambda^{0}(\partial\Omega)\times H^{-\frac{1}{2}}\Lambda^{n-1}(\Sigma)\to\mathbb{R}$,
  with $\bar{\mathcal{F}}_{p}(\phi_{\partial},\Sigma)=\bar{\mathcal{F}}(0,\phi_{\partial},\Sigma)$ the potential flow free surface water waves (\ref{vor2}) with $g\neq 0$ and $\tau=0$ can be expressed as the Hamiltonian system
  \begin{equation*}
    \dot{\bar{\mathcal{F}}}_{p}(\phi_{\partial},\Sigma)=\{\bar{\mathcal{F}}_{p},\bar{H}_{p}\},
  \end{equation*}
with the canonical Poisson bracket given by
\begin{equation*}
  \{\bar{\mathcal{F}}_{p},\bar{\mathcal{G}}_{p}\}(\phi_{\partial},\Sigma)=\int_{\Sigma}\Big(\frac{\delta\bar{\mathcal{F}}_{p}}{\delta\Sigma}\wedge\frac{\delta\bar{\mathcal{G}}_{p}}{\delta\phi_{\partial}}-\frac{\delta\bar{\mathcal{G}}_{p}}{\delta\Sigma}\wedge\frac{\delta\bar{\mathcal{F}}_{p}}{\delta\phi_{\partial}}\Big),
\end{equation*}
and the Hamiltonian functional
\begin{equation*}
  \bar{H}_{p}=\frac{1}{2}\int_{\partial\Omega}\phi_{\partial}\wedge\ast\boldsymbol{n}(d\phi).
\end{equation*}
\end{corollary}

\begin{proof}
  The result is immediate after introducing $\omega=0$ and $\tau=0$ in (\ref{eq65}). Note at the free surface $\tilde{p}=0$ since $p=\bar{p}$ the atmospheric pressure. 
\end{proof}

Corollary \ref{corollary1} states a coordinate-free Hamiltonian formulation of the Hamilton principle for surface waves stated in \citep{MR489348} for $\Omega\subseteq\mathbb{R}^{2}$.

In Section 5.4, we will extend the Hamiltonian system stated in Theorem \ref{thm1} to a port-Hamiltonian system that also incorporates the inhomogeneous boundary condition (\ref{vor2}d) with $g\neq 0$ at $\Gamma$.

\section{Distributed-parameter port-Hamiltonian system}

Based on the Hamiltonian framework for the incompressible Euler equations with a free surface discussed in Section 4, we will now present the port-Hamiltonian formulation for these equations.
The port-Hamiltonian framework allows for inhomogeneous boundary conditions and the exchange of the energy of the system through its boundaries. The port-Hamiltonian framework, therefore, provides a direct extension of Hamiltonian systems to open systems.
We will consider three formulations:
\begin{enumerate}[i)]
  \item a $(v,\Sigma)$ formulation with velocity $v$ and free surface $\Sigma$;
  \item an $(\eta,\phi_{\partial},\Sigma)$ formulation with the solenoidal velocity $\eta$, boundary potential $\phi_{\partial}$ and free surface $\Sigma$;
  \item an $(\omega,\phi_{\partial},\Sigma)$ formulation with vorticity $\omega$, boundary potential $\phi_{\partial}$ and free surface $\Sigma$.
\end{enumerate}
For all three formulations, we will prove that they can be described by a Dirac structure and we will give the related port-Hamiltonian system.
In the next section, we will first introduce some basic notations and the concept of a Dirac structure. The port-Hamiltonian formulations will be discussed in Sections 5.2--5.4.

\subsection{Introduction to Dirac structures}

The central part in defining a port-Hamiltonian system is the definition of a Dirac structure, see \citep{van2014port}. We will first give a brief introduction to general Dirac structures.

Let the linear spaces $\mathcal{F}$ and $\mathcal{E}$ denote, respectively, the space of flows and the space of efforts. The total space $\mathcal{F}\times\mathcal{E}$ is called the space of power variables, and is equipped with a linear and non-degenerate pairing 
\begin{equation*}
\langle\cdot\mid\cdot\rangle:\mathcal{F} \times \mathcal{E} \to \mathbb{R}.
\end{equation*}
Using the pairing $\langle\cdot|\cdot\rangle$, we can define a symmetric bilinear form $\llangle\cdot,\cdot\rrangle$ on $\mathcal{F}\times\mathcal{E}$ as
\begin{equation*}
\llangle(f_{1},e_{1}),(f_{2},e_{2})\rrangle:=\langle e_{1}|f_{2}\rangle+\langle e_{2}|f_{1}\rangle,\ (f_{i},e_{i})\in\mathcal{F}\times\mathcal{E},\ i=1,2.
\end{equation*}
\begin{defn}{\citep{van2014port}}\label{defn2}
Let $\mathcal{F}$ and $\mathcal{E}$ be linear spaces with the pairing $\langle\cdot|\cdot\rangle$. A Dirac structure is a linear subspace $D\subset\mathcal{F}\times\mathcal{E}$ such that $D=D^{\bot}$, with $\bot$ denoting the orthogonal complement with respect to the bilinear form $\llangle\cdot,\cdot\rrangle$.\\
\end{defn}
\begin{remark}
  From Definition \ref{defn2} we have that for any $(f,e)\in D$
  \begin{equation}\label{power1}
    0=\llangle(f,e),(f,e)\rrangle=2\langle e|f\rangle\text{ implies } \langle e|f\rangle=0,
  \end{equation}
  which means that if $(f,e)$ is a pair of power variables, the Dirac structure is power-conserving.
\end{remark}

\subsection{Dirac structure in terms of velocity and free surface variables}

In this section, we will present the Dirac structure with respect to the state variables $(v,\Sigma)$ for the $n$-dimensional incompressible Euler equations with a free surface (\ref{dieuler}).
We start with defining the following pairings, extension and trace-lifting operators.

\begin{defn}
  Given the differential forms $\alpha\in L^{2}\Lambda^{k}(\Omega)$ and $\beta\in L^{2}\Lambda^{n-k}(\Omega)$, we define at $\Omega$ the pairing between $\alpha$ and $\beta$ as
\begin{equation*}
\langle\alpha|\beta\rangle:=\int_{\Omega}\alpha\wedge\beta.
\end{equation*}
Similarly, we define at $\partial\Omega$ the pairing between $\alpha\in L^{2}\Lambda^{k}(\partial\Omega)$ and $\beta\in L^{2}\Lambda^{n-k-1}(\partial\Omega)$ as
\begin{equation*}
\langle\alpha|\beta\rangle:=\int_{\partial\Omega}\alpha\wedge\beta.
\end{equation*}
\end{defn}

\begin{defn}{\citep[Theorem 11.7]{MR0350179}}
  The function space $H_{00}^{\frac{1}{2}}\Lambda^{0}(\Gamma_{1})$, with $\Gamma_{1}\subset\partial\Omega$, is defined as
  \begin{equation}\label{eq123}
    \begin{aligned}
      H_{00}^{\frac{1}{2}}\Lambda^{0}(\Gamma_{1}):=&\{\mu\in H^{\frac{1}{2}}\Lambda^{0}(\Gamma_{1})\mid\rho^{-\frac{1}{2}}\mu\in L^{2}\Lambda^{0}(\Gamma_{1}),\rho=\mathrm{dist}(x,\partial\Omega)\}\\
      =&\{\mu\in H^{\frac{1}{2}}\Lambda^{0}(\Gamma_{1})\mid\mu_{0}\in H^{\frac{1}{2}}\Lambda^{0}(\partial\Omega)\},
    \end{aligned}
  \end{equation}
  with $\mu_{0}$ the extension of $\mu$ by zero outside $\Gamma_{1}$.
\end{defn}

\begin{remark}
Note, $H^{\frac{1}{2}}_{00}\Lambda^{0}(\Gamma_{1})\subset H^{\frac{1}{2}}\Lambda^{0}(\Gamma_{1})$ with strict inclusion.
\end{remark}

\begin{lemma}{\citep[p45]{MR1409140}}
  Given the extension operator $E:H^{\frac{1}{2}}\Lambda^{0}(\Sigma)\to H^{\frac{1}{2}}\Lambda^{0}(\partial\Omega)$, defined as
  \begin{equation}
    E(\mu_{\Sigma})|_{\Sigma}=\mu_{\Sigma},\quad\forall\mu_{\Sigma}\in H^{\frac{1}{2}}\Lambda^{0}(\Sigma),
  \end{equation}
  then $H^{\frac{1}{2}}\Lambda^{0}(\partial\Omega)$ can be decomposed into the spaces
  \begin{equation}
    H^{\frac{1}{2}}\Lambda^{0}(\partial\Omega)=H^{\frac{1}{2}}\Lambda^{0}(\Sigma)\oplus H_{00}^{\frac{1}{2}}\Lambda^{0}(\Gamma).
  \end{equation}
\end{lemma}

\begin{defn}
  The trace-lifting operator $\mathrm{li}:H^{\frac{1}{2}}\Lambda^{0}(\partial\Omega)\to H^{1}\Lambda^{0}(\Omega)$ is defined as
\begin{equation}\label{tracelift}
  \mathrm{tr}\big(\mathrm{li}(\mu)\big)=\mu,\quad\forall\mu\in H^{\frac{1}{2}}\Lambda^{0}(\partial\Omega).
\end{equation}
\end{defn}
Next, we define the function spaces used in the definition of the Dirac structure of the incompressible Euler equations with a free surface in the $(v,\Sigma)$ variables
\begin{equation}\label{space1}
\begin{aligned}
\mathcal{F}_{1}&:=P^{\ast}\Lambda^{1}(\Omega)\times H^{-\frac{1}{2}}\Lambda^{n-1}(\Sigma)\times H^{-\frac{1}{2}}\Lambda^{n-1}(\Gamma),\\
\mathcal{E}_{1}&:=P\Lambda^{n-1}(\Omega)\times H^{\frac{1}{2}}\Lambda^{0}(\Sigma)\times H^{\frac{1}{2}}\Lambda^{0}(\Gamma),
\end{aligned}
\end{equation}
where the spaces $\mathcal{F}_{1}$ and $\mathcal{E}_{1}$ are dual to each other with a non-degenerate pairing using $L^{2}$ as pivot space.

The Dirac structure for the incompressible Euler equations with a free surface in terms of the state variables $(v,\Sigma)$ is stated in the following theorem.

\begin{theorem}\label{thm14}
Let $\Omega\subset\mathbb{R}^{n}$, $n\in\{1,2,3\}$ be a $n$-dimensional oriented connected manifold with Lipschitz continuous boundary $\partial\Omega=\Sigma\cup\Gamma$ with $\Sigma\cap\Gamma=\emptyset$. Assume $v\in P^{\ast}\Lambda^{1}(\Omega)$ is the solution of the incompressible Euler equations with a free surface (\ref{dieuler}). Given the function spaces $\mathcal{F}_{1}$ and $\mathcal{E}_{1}$, defined in (\ref{space1}), together with the bilinear form:
\begin{equation}\label{bilinear5}
\begin{aligned}
\llangle(f^{1},e^{1}),(f^{2},e^{2})\rrangle=&\int_{\Omega}\Big(e_{v}^{1}\wedge f_{v}^{2}+e_{v}^{2}\wedge f_{v}^{1}\Big)+\int_{\Sigma}\Big(e_{\Sigma}^{1}\wedge f_{\Sigma}^{2}+e_{\Sigma}^{2}\wedge f_{\Sigma}^{1}\Big)\\
&+\int_{\Gamma}\Big(e_{b}^{1}\wedge f_{b}^{2}+e_{b}^{2}\wedge f_{b}^{1}\Big),
\end{aligned}
\end{equation}
where 
\begin{equation*}
f^{i}=(f_{v}^{i},f_{\Sigma}^{i},f_{b}^{i})\in\mathcal{F}_{1},\quad e^{i}=(e_{v}^{i},e_{\Sigma}^{i},e_{b}^{i})\in\mathcal{E}_{1},\quad i=1,2.
\end{equation*}
Then $D_{1}\subset\mathcal{F}_{1}\times\mathcal{E}_{1}$, defined as
\begin{equation}\label{dirac1}
  \begin{aligned}
  D_{1}:=\Bigg\{ & (f_{v},f_{\Sigma},f_{b},e_{v},e_{\Sigma},e_{b})\in \mathcal{F}_{1}\times\mathcal{E}_{1}\mid\\
  &
    \begin{pmatrix}
     f_{v}  \\
     f_{\Sigma}
    \end{pmatrix} 
    =
    \begin{pmatrix}
      d\big(\mathrm{li}(\tilde{e}_{\Sigma})\big)+i_{(\ast e_{v})^{\sharp}}dv \\
     (-1)^{n}\mathrm{tr}_{\Sigma}(e_{v})
    \end{pmatrix},
     \\
   &
    \begin{pmatrix}
     f_{b}  \\
     e_{b}
    \end{pmatrix} 
    =
    \begin{pmatrix}
      (-1)^{n}\mathrm{tr}_{\Gamma}(e_{v})\\
      \tilde{e}_{\Sigma}|_{\Gamma} 
    \end{pmatrix}\Bigg\}
  \end{aligned}
  \end{equation}
is a Dirac structure, where $\mathrm{tr}_{\Sigma}(\cdot)$ and $\mathrm{tr}_{\Gamma}(\cdot)$ are, respectively, the boundary traces at $\Sigma$ and $\Gamma$, and $\tilde{e}_{\Sigma}=E(e_{\Sigma})$. 
\end{theorem}

Before we prove that (\ref{dirac1}) is a Dirac structure, we first prove that $f_{v}$ is well-defined.

\begin{lemma}\label{lemma14}
  Given $e_{\Sigma}\in H^{\frac{1}{2}}\Lambda^{0}(\Sigma)$, $e_{v}\in P\Lambda^{n-1}(\Omega)$, then $f_{v}\in P^{\ast}\Lambda^{1}(\Omega)$ in the Dirac structure (\ref{dirac1}) is well-defined for the lifting operator $\mathrm{li}(\cdot)$.
\end{lemma}

\begin{proof}
  Assume that there are two trace-lifting operators $\mathrm{li},\mathrm{li}':H^{\frac{1}{2}}\Lambda^{0}(\partial\Omega)\to H^{1}\Lambda^{0}(\Omega)$ such that
  \begin{equation*}
    \mathrm{tr}\big(\mathrm{li}(\tilde{\mu})\big)=\mathrm{tr}\big(\mathrm{li}'(\tilde{\mu})\big)=\tilde{\mu}=E(\mu_{\Sigma}),\quad\forall\mu_{\Sigma}\in H^{\frac{1}{2}}\Lambda^{0}(\Sigma).
  \end{equation*}
  For any $e_{v}\in P\Lambda^{n-1}(\Omega)$ and $e_{\Sigma}\in H^{\frac{1}{2}}\Lambda^{0}(\Sigma)$, the variables $f_{v}\in P^{\ast}\Lambda^{1}(\Omega)$ and $f_{v}'\in P^{\ast}\Lambda^{1}(\Omega)$ are given as
  \begin{equation*}
    \begin{aligned}
    f_{v}&=d\big(\mathrm{li}(\tilde{e}_{\Sigma})\big)+i_{(\ast e_{v})^{\sharp}}dv,\\
    f_{v}'&=d\big(\mathrm{li}(\tilde{e}_{\Sigma})\big)+i_{(\ast e_{v})^{\sharp}}dv,
    \end{aligned}
  \end{equation*}
  with $\tilde{e}_{\Sigma}=E(e_{\Sigma})$. Then,
  \begin{equation*}
    f_{v}-f_{v}'=d\big(\mathrm{li}(\tilde{e}_{\Sigma})-\mathrm{li}'(\tilde{e}_{\Sigma})\big).
  \end{equation*}
  Since $\delta f_{v}=\delta f_{v}'=0$, we have
  \begin{equation*}
    \delta d\big(\mathrm{li}(\tilde{e}_{\Sigma})-\mathrm{li}'(\tilde{e}_{\Sigma})\big)=0,
  \end{equation*}
  with boundary condition
  \begin{equation*}
    \mathrm{tr}\big(\mathrm{li}(\tilde{e}_{\Sigma})-\mathrm{li}'(\tilde{e}_{\Sigma})\big)=E(e_{\Sigma})-E(e_{\Sigma})=0\quad\text{at}\ \partial\Omega.
  \end{equation*}
  So, $\mathrm{li}(\tilde{e}_{\Sigma})-\mathrm{li}'(\tilde{e}_{\Sigma})$ satisfies the Laplace equation in $\Omega$ with homogeneous Dirichlet boundary condition at $\partial\Omega$, hence
  \begin{equation*}
    \mathrm{li}(\tilde{e}_{\Sigma})-\mathrm{li}'(\tilde{e}_{\Sigma})=0\quad\text{in }\Omega,
  \end{equation*}
  which implies that $f_{v}=f_{v}'$.
\end{proof}

Next, we give a proof of Theorem \ref{thm14}.

\begin{proof} 
  \begin{enumerate}[(1)]
    \item First we check that $i_{(\ast e_{v}^{1})^{\sharp}}dv\wedge e_{v}^{2}$ is skew-symmetric. For any $e_{v}^{1,2}\in P\Lambda^{n-1}(\Omega)$, we have using (\ref{rela4})
    \begin{equation}\label{antisy}
      \begin{aligned}
        i_{(\ast e_{v}^{1})^{\sharp}}dv\wedge e_{v}^{2}
        =\ast\big((\ast e_{v}^{1})\wedge(\ast dv)\big)\wedge e_{v}^{2}=-\ast\big((\ast e_{v}^{2})\wedge(\ast dv)\big)\wedge e_{v}^{1}
        =-i_{(\ast e_{v}^{2})^{\sharp}}dv\wedge e_{v}^{1}.
      \end{aligned}
    \end{equation}
    
    \item Next, we show that $D_{1}\subset D_{1}^{\bot}$. Let $(f^{1},e^{1})\in D_{1}$ be fixed, and consider any $(f^{2},e^{2})\in D_{1}$. From (\ref{antisy}), $e_{v}^{1},\ e_{v}^{2}\in P\Lambda^{n-1}(\Omega)$, hence $de_{v}^{1}=de_{v}^{2}=0$, and Stokes' Theorem, we obtain that
    \begin{align*}
      \llangle (f^{1},e^{1}),(f^{2},e^{2})\rrangle
      =&\int_{\Omega}\Big( e_{v}^{1}\wedge\big(d(\mathrm{li}(\tilde{e}_{\Sigma}^{2}))+i_{(\ast e_{v}^{2})^{\sharp}}dv\big)+e_{v}^{2}\wedge\big(d(\mathrm{li}(\tilde{e}_{\Sigma}^{1}))+i_{(\ast e_{v}^{1})^{\sharp}}dv\big)\Big)\\
      &+(-1)^{n}\int_{\Sigma}\Big(e_{\Sigma}^{1}\wedge\mathrm{tr}_{\Sigma}(e_{v}^{2})+e_{\Sigma}^{2}\wedge\mathrm{tr}_{\Sigma}(e_{v}^{1}) \Big)\\
      &+(-1)^{n}\int_{\Gamma}\Big(\tilde{e}_{\Sigma}^{1}\wedge\mathrm{tr}_{\Gamma}(e_{v}^{2})+\tilde{e}_{\Sigma}^{2}\wedge\mathrm{tr}_{\Gamma}(e_{v}^{1})\Big)\\     
      =&\ 0.      
    \end{align*}
    \noindent
  Hence, we have $(f^{1},e^{1})\in D_{1}^{\bot}$, which implies $D_{1}\subset D_{1}^{\bot}$.
     
    \item Finally, we show that $D_{1}^{\bot}\subset D_{1}$. Let $(f^{1},e^{1})\in D_{1}^{\bot}\subset\mathcal{F}_{1}\times\mathcal{E}_{1}$. Then 
    \begin{equation*}
      \llangle(f^{1},e^{1}),(f^{2},e^{2})\rrangle=0,\quad\forall(f^{2},e^{2})\in D_{1},
    \end{equation*}
   or equivalently for all $(f^{2},e^{2})\in D_{1}$ we have
    \begin{equation}\label{eq11}
      \begin{aligned}
        &\int_{\Omega}\Big(e_{v}^{1}\wedge\big(d(\mathrm{li}(\tilde{e}_{\Sigma}^{2}))+i_{(\ast e_{v}^{2})^{\sharp}}dv\big)+e_{v}^{2}\wedge f_{v}^{1}\Big)+\int_{\Sigma}\Big((-1)^{n}e_{\Sigma}^{1}\wedge\mathrm{tr}_{\Sigma}(e_{v}^{2})+e_{\Sigma}^{2}\wedge f_{\Sigma}^{1}\Big)\\
        &+\int_{\Gamma}\Big(e_{b}^{1}\wedge(-1)^{n}\mathrm{tr}_{\Gamma}(e_{v}^{2})+\tilde{e}_{\Sigma}^{2}\wedge f_{b}^{1}\Big)=0.
      \end{aligned}
    \end{equation}
    Take $e_{v}^{2}\in P\Lambda^{n-1}(\Omega)$ such that $\mathrm{tr}_{\Gamma}(e_{v}^{2})=0$ and choose $\tilde{e}_{\Sigma}^{2}\in H^{\frac{1}{2}}_{00}\Lambda^{0}(\Sigma)$, which from (\ref{eq123}) implies that $\tilde{e}_{\Sigma}^{2}|_{\Gamma}=0$, then we obtain using the skew-symmetry of $i_{(\ast e_{v}^{2})^{\sharp}}dv\wedge e_{v}^{1}$ and $e_{v}^{1}\wedge d\big(\mathrm{li}(\tilde{e}_{\Sigma}^{2})\big)=(-1)^{n-1}d\big(e_{v}^{1}\wedge\mathrm{li}(\tilde{e}_{\Sigma}^{2})\big)$ since $e_{v}^{1}\in P\Lambda^{n-1}(\Omega)$
    \begin{equation*}
      \begin{aligned}
        &\int_{\Omega}\Big(-e_{v}^{2}\wedge i_{(\ast e_{v}^{1})^{\sharp}}dv+(-1)^{n-1}d\big(e_{v}^{1}\wedge\mathrm{li}(\tilde{e}_{\Sigma}^{2})\big)+e_{v}^{2}\wedge f_{v}^{1}\Big)\\
        &+\int_{\partial\Omega}(-1)^{n}\tilde{e}_{\Sigma}^{1}\wedge\mathrm{tr}(e_{v}^{2})+\int_{\Sigma}e_{\Sigma}^{2}\wedge f_{\Sigma}^{1}=0,
      \end{aligned}
    \end{equation*}
    where we extended $\int_{\Sigma}(-1)^{n}e_{\Sigma}^{1}\wedge\mathrm{tr}_{\Sigma}(e_{v}^{2})$ to $\partial\Omega$ since $\mathrm{tr}_{\Gamma}(e_{v}^{2})=0$. Using Stokes' theorem and $\tilde{e}_{\Sigma}^{2}|_{\Gamma}=0$, we then obtain
    \begin{equation*}
      \begin{aligned}
        &\int_{\Omega}\Big(-e_{v}^{2}\wedge i_{(\ast e_{v}^{1})^{\sharp}}dv+(-1)^{n}d\big(e_{v}^{2}\wedge\mathrm{li}(\tilde{e}_{\Sigma}^{1})\big)+e_{v}^{2}\wedge f_{v}^{1}\Big)\\
        &+\int_{\Sigma}\Big((-1)^{n-1}e_{\Sigma}^{2}\wedge \mathrm{tr}_{\Sigma}(e_{v}^{1})+e_{\Sigma}^{2}\wedge f_{\Sigma}^{1}\Big)=0.
      \end{aligned}
    \end{equation*}
    Using $d\big(e_{v}^{2}\wedge\mathrm{li}(\tilde{e}_{\Sigma}^{1})\big)=(-1)^{n-1}e_{v}^{2}\wedge d\big(\mathrm{li}(\tilde{e}_{\Sigma}^{1})\big)$ for $e_{v}^{2}\in P\Lambda^{n-1}(\Omega)$, we have
    \begin{equation*}
      \begin{aligned}
        &\int_{\Omega}e_{v}^{2}\wedge\Big(-i_{(\ast e_{v}^{1})^{\sharp}}dv-d\big(\mathrm{li}(\tilde{e}_{\Sigma}^{1})\big)+f_{v}^{1}\Big)+\int_{\Sigma}e_{\Sigma}^{2}\wedge\big((-1)^{n-1}\mathrm{tr}_{\Sigma}(e_{v}^{1})+f_{\Sigma}^{1}\big)=0.
      \end{aligned}
    \end{equation*}
    Since $e_{v}^{2}$ and $e_{\Sigma}^{2}$ are arbitrary, we have
    \begin{subequations}\label{eq122}
      \begin{empheq}[left=\empheqlbrace]{align}
        f_{v}^{1}&=d(\mathrm{li}(\tilde{e}_{\Sigma}^{1}))+i_{(\ast e_{v}^{1})^{\sharp}}dv\quad\text{in }\Omega,\label{eq62}\\
        f_{\Sigma}^{1}&=(-1)^{n}\mathrm{tr}_{\Sigma}(e_{v}^{1})\quad\quad\quad\quad \ \text{at }\Sigma.\label{eq63}
      \end{empheq}
    \end{subequations}

    Next, we consider the relations for the port variables $f_{b}$, $e_{b}$. After introducing (\ref{eq122}) into (\ref{eq11}), we obtain
      \begin{align*}
        &\int_{\Omega}\Big( e_{v}^{1}\wedge\big(d(\mathrm{li}(\tilde{e}_{\Sigma}^{2}))+i_{(\ast e_{v}^{2})^{\sharp}}dv\big)+e_{v}^{2}\wedge\big(d(\mathrm{li}(\tilde{e}_{\Sigma}^{1}))+i_{(\ast e_{v}^{1})^{\sharp}}dv\big)\Big)\\
        &+(-1)^{n}\int_{\Sigma}\Big(e_{\Sigma}^{1}\wedge\mathrm{tr}_{\Sigma}(e_{v}^{2})+e_{\Sigma}^{2}\wedge\mathrm{tr}_{\Sigma}(e_{v}^{1}) \Big)+\int_{\Gamma}\Big(e_{b}^{1}\wedge(-1)^{n}\mathrm{tr}_{\Gamma}(e_{v}^{2})+\tilde{e}_{\Sigma}^{2}\wedge f_{b}^{1}\Big)\\
        =&\int_{\Gamma}\Big(\tilde{e}_{\Sigma}^{2}\wedge\big((-1)^{n-1}\mathrm{tr}_{\Gamma}(e_{v}^{1})+f_{b}^{1}\big)+(-1)^{n}\mathrm{tr}_{\Gamma}(e_{v}^{2})\wedge\big(-\tilde{e}_{\Sigma}^{1}+e_{b}^{1}\big)\Big)=0.
      \end{align*}
    Since $e_{\Sigma}^{2}$ and $e_{v}^{2}$ are arbitrary, we have
    \begin{subequations}
      \begin{empheq}[left=\empheqlbrace]{align}
        f_{b}^{1}&=(-1)^{n}\mathrm{tr}_{\Gamma}(e_{v}^{1}),\\
        e_{b}^{1}&=\tilde{e}_{\Sigma}^{1}|_{\Gamma}.
      \end{empheq}
    \end{subequations}
    Hence $D_{1}^{\bot}\subset D_{1}$, which together with $D_{1}\subset D_{1}^{\bot}$ gives that $D_{1}=D_{1}^{\bot}$, thus $D_{1}$ is a Dirac structure.
  \end{enumerate}
\end{proof}

\begin{remark}
  Notice that by the power-conserving property (\ref{power1}), together with the bilinear form (\ref{bilinear5}), it immediately follows that for any 
$(f_{v},e_{\Sigma},f_{b},e_{v},e_{\Sigma},e_{b})\in D_{1}\subset\mathcal{F}_{1}\times\mathcal{E}_{1}$ we have
\begin{equation*}
  \int_{\Omega}e_{v}\wedge f_{v}+\int_{\Sigma}e_{\Sigma}\wedge f_{\Sigma}+\int_{\Gamma}e_{b}\wedge f_{b}=0.
\end{equation*}
With the Hamiltonian (\ref{total2}), flow variables and energy variables, together with the chain rule, we have
\begin{equation}\label{eq136}
  \dot{H}=\int_{\Gamma}e_{b}\wedge f_{b}.
\end{equation}
\end{remark}

Next, we aim at deriving the $(v,\Sigma)$ distributed-parameter port-Hamiltonian formulation of the incompressible Euler equations with a free surface. Define the flow variables in the Dirac structure (\ref{dirac1}) as
\begin{alignat*}{2}
f_{v}&=-v_{t}\in P^{\ast}\Lambda^{1}(\Omega),\quad &&f_{\Sigma}=-\Sigma_{t}\in H^{-\frac{1}{2}}\Lambda^{n-1}(\Sigma),
\end{alignat*}
and the energy variables as
\begin{alignat*}{2}
&e_{v}=\frac{\delta H}{\delta v}\in P\Lambda^{n-1}(\Omega),\quad && e_{\Sigma}=\frac{\delta H}{\delta\Sigma}\in H^{\frac{1}{2}}\Lambda^{0}(\Sigma),
\end{alignat*}
with the Hamiltonian $H$ given in (\ref{total2}). Using Theorem \ref{thm14}, the port-Hamiltonian formulation can be stated in the following corollary.

\begin{corollary}
  Given an oriented $n$-dimensional connected manifold $\Omega$ with Lipschitz continuous boundary $\partial\Omega$.
  Given the essential boundary condition (\ref{dieuler4}) and the port variable $f_{b}=-g$.
  The distributed-parameter port-Hamiltonian system for the incompressible Euler equations with a free surface (\ref{dieuler}) with respect to the state space variables $(v,\Sigma)\in P^{\ast}\Lambda^{1}(\Omega)\times H^{-\frac{1}{2}}\Lambda^{n-1}(\Sigma)$, 
  Dirac structure $D_{1}$ (\ref{dirac1}), and Hamiltonian (\ref{total2}), is given as
  \begin{equation}\label{coro}
    \begin{aligned}
      &
    \begin{pmatrix}
      -v_{t}  \\
      -\Sigma_{t} 
     \end{pmatrix} 
     =
     \begin{pmatrix}
      i_{(\ast\cdot)^{\sharp}}dv \quad & d\big(\mathrm{li}(E(\cdot))\big)  \\
      (-1)^{n}\mathrm{tr}_{\Sigma}(\cdot) \quad & 0 
      \end{pmatrix}
      \begin{pmatrix}
        \frac{\delta H}{\delta v} \\
        \frac{\delta H}{\delta\Sigma}
       \end{pmatrix},
      \end{aligned}
    \end{equation}
  \text{with port variables $f_{b},e_{b}$ defined as},
    \begin{equation}\label{coro3}
    \begin{aligned}
     \begin{pmatrix}
      f_{b}  \\
      e_{b}
     \end{pmatrix} 
     =
     \begin{pmatrix}
      (-1)^{n}\mathrm{tr}_{\Gamma}(\cdot)\quad &0 \\
      0\quad & E(\cdot)|_{\Gamma}
     \end{pmatrix}
     \begin{pmatrix}
      \frac{\delta H}{\delta v} \\
      \frac{\delta H}{\delta\Sigma}
     \end{pmatrix}.
   \end{aligned}
  \end{equation}
\end{corollary}

\begin{proof}
  Using (\ref{h3}) and the essential boundary condition (\ref{dieuler4}), we have 
  \begin{equation}\label{eq134}
    \begin{aligned}
      &\frac{\delta H}{\delta v}=(-1)^{n-1}\ast v\in P\Lambda^{n-1}(\Omega),\\
      &\frac{\delta H}{\delta\Sigma}=\mathrm{tr}\big(\frac{1}{2}\llangle v^{\sharp},v^{\sharp}\rrangle+\Phi+\frac{\tilde{p}}{\rho}\big)\in H^{\frac{1}{2}}\Lambda^{0}(\Sigma).
    \end{aligned}
  \end{equation}
  Thus, from Lemma \ref{thm14} we obtain
  \begin{equation}\label{eq135}
    d\big(\mathrm{li}(E(\frac{\delta H}{\delta\Sigma}))\big)=d\big(\frac{1}{2}\llangle v^{\sharp},v^{\sharp}\rrangle+\Phi+\frac{\tilde{p}}{\rho}\big).
  \end{equation}
  Substituting (\ref{eq134}) and (\ref{eq135}) into the port-Hamiltonian formulation (\ref{coro}), we obtain the incompressible Euler equations (\ref{dieuler1}) and (\ref{dieuler3}).
  Next, the port variable $f_{b}=-g=(-1)^{n}\mathrm{tr}_{\Gamma}(\frac{\delta H}{\delta v})=-\ast\boldsymbol{n}(v)$ gives the inhomogeneous boundary condition (\ref{dieuler5}) at $\Gamma$, and $v\in P^{\ast}\Lambda^{1}(\Omega)$ implies (\ref{dieuler2}).
\end{proof}

We will show that the Dirac structure (\ref{bilinear5}), (\ref{dirac1}) is associated to a Poisson bracket $\{\cdot,\cdot\}_{D}(v,\Sigma)$.
Define the bilinear form (see \cite{van2002hamiltonian})
\begin{equation*}
  [e_{1},e_{2}]_{D}:=(e_{1},f_{2})=\int_{\Omega}e_{v}^{1}\wedge f_{v}^{2}+\int_{\Sigma}e_{\Sigma}^{1}\wedge f_{\Sigma}^{2}+\int_{\Gamma}e_{b}^{1}\wedge f_{b}^{2}.
\end{equation*}
For $\frac{\delta\mathcal{F}}{\delta v},\frac{\delta\mathcal{G}}{\delta v}\in P\Lambda^{n-1}(\Omega)$ and $\frac{\delta\mathcal{F}}{\delta\Sigma},\frac{\delta\mathcal{G}}{\delta\Sigma}\in H^{\frac{1}{2}}\Lambda^{0}(\Sigma)$, the bilinear form $\{\cdot,\cdot\}_{D}(v,\Sigma):\mathcal{F}\times\mathcal{G}\to\mathbb{R}$ associated to the Dirac structure (\ref{dirac1}) is
\begin{equation}\label{bracket1}
  \begin{aligned}
    \{\mathcal{F},\mathcal{G}\}_{D}(v,\Sigma)&:=[\delta\mathcal{F},\delta\mathcal{G}]_{D}\\
    &=\int_{\Omega}\frac{\delta\mathcal{G}}{\delta v}\wedge\Big(d\big(\mathrm{li}(E(\frac{\delta\mathcal{F}}{\delta\Sigma}))\big)+i_{(\ast\frac{\delta\mathcal{F}}{\delta v})^{\sharp}}dv \Big)+\int_{\Sigma}\frac{\delta\mathcal{G}}{\delta\Sigma}\wedge(-1)^{n}\mathrm{tr}_{\Sigma}(\frac{\delta\mathcal{F}}{\delta v})\\
    &+\int_{\Gamma}E(\frac{\delta\mathcal{G}}{\delta\Sigma})\wedge(-1)^{n}\mathrm{tr}_{\Gamma}(\frac{\delta\mathcal{F}}{\delta v}).
  \end{aligned}
\end{equation}

\begin{remark}
  With (\ref{rela4}), Stokes' theorem and the condition $d(\frac{\delta\mathcal{G}}{\delta v})=0$, which follows from $\frac{\delta\mathcal{G}}{\delta v}\in P\Lambda^{n-1}(\Omega)$, we have that (\ref{bracket1}) can be rewritten as
  \begin{equation}\label{bracket8}
    \begin{aligned}
      \{\mathcal{F},\mathcal{G}\}_{D}(v,\Sigma)=&(-1)^{n-1}\int_{\Omega}(\ast dv)\wedge(\ast\frac{\delta\mathcal{G}}{\delta v})\wedge(\ast\frac{\delta\mathcal{F}}{\delta v})\\
      &+(-1)^{n-1}\int_{\partial\Omega}\Big(E(\frac{\delta\mathcal{F}}{\delta\Sigma})\wedge\mathrm{tr}(\frac{\delta\mathcal{G}}{\delta v})-E(\frac{\delta\mathcal{G}}{\delta\Sigma})\wedge\mathrm{tr}(\frac{\delta\mathcal{F}}{\delta v}) \Big).
    \end{aligned}
  \end{equation}
\end{remark}

Since the bracket $\{\cdot,\cdot\}_{D}(v,\Sigma)$ has the same structure as the Poisson bracket (\ref{bracket3}),
it is straightforward to check that $\{\cdot,\cdot\}_{D}(v,\Sigma)$ is linear, anti-symmetric and satisfies the Jacobi identity.
The bracket $\{\cdot,\cdot\}_{D}(v,\Sigma)$ is thus a Poisson bracket and the system generated by the Dirac structure $D_{1}$ is a Poisson structure.

\begin{theorem}\label{thm5}
  Let $\Omega\subset\mathbb{R}^{n},n\in\{1,2,3\}$ be an oriented $n$-dimensional connected manifold with Lipschitz continuous boundary $\partial\Omega=\Sigma\cup\Gamma$ with free surface $\Sigma$ and fixed boundary $\Gamma$, $\Sigma\cap\Gamma=\emptyset$. 
  For any functionals $\mathcal{F}(v,\Sigma):P^{\ast}\Lambda^{1}(\Omega)\times H^{-\frac{1}{2}}\Lambda^{n-1}(\Sigma)\to\mathbb{R}$, the port-Hamiltonian formulation of the incompressible Euler equations with a free surface (\ref{coro}) and (\ref{coro3}) is equivalent to
  \begin{equation}\label{eq87}
    \dot{\mathcal{F}}(v,\Sigma)=\{\mathcal{F},H\}_{D}(v,\Sigma)-\int_{\Gamma}E(\frac{\delta\mathcal{F}}{\delta\Sigma})\wedge\ast\boldsymbol{n}(v),
  \end{equation}
  with the Hamiltonian $H$ given by (\ref{total2}).
\end{theorem}

\begin{proof}
  Using the Poisson bracket (\ref{bracket1}) in terms of the $(v,\Sigma)$ variables, together with (\ref{rela4}), we have 
  \begin{equation*}
    \begin{aligned}
      &\{\mathcal{F},H\}_{D}(v,\Sigma)=-\{H,\mathcal{F}\}_{D}(v,\Sigma)\\
      =&-\int_{\Omega}\frac{\delta\mathcal{F}}{\delta v}\wedge\Big(i_{(\ast\frac{\delta H}{\delta v})^{\sharp}}dv+d\big(\mathrm{li}(E(\frac{\delta H}{\delta\Sigma}))\big)\Big)+(-1)^{n-1}\int_{\Sigma}\frac{\delta\mathcal{F}}{\delta\Sigma}\wedge\mathrm{tr}(\frac{\delta H}{\delta v})\\
      &+(-1)^{n-1}\int_{\Gamma}E(\frac{\delta\mathcal{F}}{\delta\Sigma})\wedge\mathrm{tr}(\frac{\delta H}{\delta v}).
    \end{aligned}
  \end{equation*}
  Hence, the right-hand side term of (\ref{eq87}) is 
  \begin{equation}\label{eq90}
    \begin{aligned}
      &\{\mathcal{F},H\}_{D}(v,\Sigma)-\int_{\Gamma}E(\frac{\delta\mathcal{F}}{\delta\Sigma})\wedge\ast\boldsymbol{n}(v)\\
      =&-\int_{\Omega}\frac{\delta\mathcal{F}}{\delta v}\wedge\Big(i_{(\ast\frac{\delta H}{\delta v})^{\sharp}}dv+d\big(\mathrm{li}(E(\frac{\delta H}{\delta\Sigma}))\big)\Big)+(-1)^{n-1}\int_{\Sigma}\frac{\delta\mathcal{F}}{\delta\Sigma}\wedge\mathrm{tr}(\frac{\delta H}{\delta v}).
    \end{aligned}
  \end{equation}
  Next, using the functional chain rule
  \begin{equation}\label{eq91}
    \dot{\mathcal{F}}(v,\Sigma)=\int_{\Omega}\frac{\delta\mathcal{F}}{\delta v}\wedge v_{t}+\int_{\Sigma}\frac{\delta\mathcal{F}}{\delta\Sigma}\wedge\Sigma_{t},
  \end{equation}
  and combining (\ref{eq90}) with (\ref{eq91}), we obtain (\ref{eq87}) if (\ref{coro}) holds.
  
  Conversely, if (\ref{eq87}) holds, together with (\ref{eq90}) and (\ref{eq91}), (\ref{coro}) directly follows since $\mathcal{F}$ is arbitrary. Since the Poisson bracket $\{\cdot,\cdot\}_{D}$ is skew-symmetric, 
  we have from (\ref{eq87})
  \begin{equation}\label{eq137}
    \dot{H}(v,\Sigma)=-\int_{\Gamma}E(\frac{\delta H}{\delta\Sigma})\wedge\ast\boldsymbol{n}(v)=\int_{\Gamma}E(\frac{\delta H}{\delta\Sigma})\wedge(-1)^{n}\mathrm{tr}(\frac{\delta H}{\delta v}).
  \end{equation}
  Using (\ref{eq136}), we obtain the relation for the port variables (\ref{coro3}).
\end{proof}

\begin{corollary}
  Given the Hamiltonian $H: P^{\ast}\Lambda^{1}(\Omega)\times H^{-\frac{1}{2}}\Lambda^{n-1}(\Sigma)\to\mathbb{R}$ stated in (\ref{total2}) and the essential boundary conditions (\ref{dieuler4}),(\ref{dieuler5}).
  The rate of change of the Hamiltonian for the incompressible Euler equations with a free surface (\ref{dieuler}) is 
  \begin{equation}\label{eq94}
    \dot{H}(v,\Sigma)=-\int_{\Gamma}\mathrm{tr}(h)\wedge g.
  \end{equation} 
\end{corollary}

\begin{proof}
  From (\ref{eq137}), with the functional derivative $\frac{\delta H}{\delta\Sigma}=\mathrm{tr}(h)$, which follows from (\ref{h2}), and (\ref{dieuler4}), and the inhomogeneous boundary condition (\ref{dieuler5}) for the velocity, $\ast\boldsymbol{n}(v)=g$ at $\Gamma$,
  we obtain (\ref{eq94}). 
\end{proof}

\begin{remark}
  Consider the homogeneous boundary condition, which means that for any functional $\mathcal{F}$, $\frac{\delta\mathcal{F}}{\delta v}\in\mathring{P}\Lambda^{k}(\Omega)$, we have that $\mathrm{tr}(\frac{\delta\mathcal{F}}{\delta v})=0$.
  Substituting the boundary condition into (\ref{bracket8}) and (\ref{eq87}), we obtain the same results as (\ref{bracket3}) and (\ref{hami3}) shown.
\end{remark}

\subsection{Dirac structure in terms of the $(\eta,\phi_{\partial},\Sigma)$ variables}

Based on the Hamiltonian formulation presented in Section 4.2, we will present in this section the port-Hamiltonian formulation of the incompressible Euler equations with a free surface with respect to $(\eta,\phi_{\partial},\Sigma)$ variables 
with, $\eta$ a solenoidal velocity field, $\phi_{\partial}$ a potential function and free surface $\Sigma$. 
We will first define the trace-lifting operators and function spaces that will be used in the definition of the Dirac structure.

\begin{defn}
The harmonic trace-lifting operator $\mathrm{li}_{\phi}:H^{\frac{1}{2}}\Lambda^{0}(\partial\Omega)\to H^{1}\Lambda^{0}(\Omega)$ is defined as 
\begin{equation}\label{tracelift2}
  \mathrm{tr}\big(\mathrm{li}_{\phi}(\mu)\big)=\mu,\ \forall\mu\in H^{\frac{1}{2}}\Lambda^{0}(\partial\Omega),
\end{equation}
with
\begin{equation*}
  \langle d(\mathrm{li}_{\phi}(\mu)),d\psi\rangle_{L^{2}\Lambda^{1}(\Omega)}=0,\quad\forall\psi\in\mathring{H}^{1}\Lambda^{0}(\Omega).
\end{equation*}
\end{defn}
Next, we define the following function spaces for the Dirac structure in the $(\eta,\phi_{\partial},\Sigma)$ variables as
\begin{equation}\label{space2}
  \begin{aligned}
    \mathcal{F}_{2}:=&\mathring{\mathfrak{B}}^{\ast 1}\times H^{\frac{1}{2}}\Lambda^{0}(\partial\Omega)\times H^{-\frac{1}{2}}\Lambda^{n-1}(\Sigma)\times H^{-\frac{1}{2}}\Lambda^{n-1}(\Gamma),\\
    \mathcal{E}_{2}:=&\mathring{\mathfrak{B}}^{(n-1)}\times H^{-\frac{1}{2}}\Lambda^{n-1}(\partial\Omega)\times H^{\frac{1}{2}}\Lambda^{0}(\Sigma)\times H^{\frac{1}{2}}\Lambda^{0}(\Gamma),
  \end{aligned}
\end{equation}
where the spaces $\mathcal{F}_{2}$ and $\mathcal{E}_{2}$ are dual to each other with a non-degenerate pairing using $L^{2}$ as pivot space.

The following theorem states the Dirac structure for the $(\eta,\phi_{\partial},\Sigma)$ formulation.

\begin{theorem}\label{thm6}
  Let $\Omega\subset\mathbb{R}^{n}$, $n\in\{1,2,3\}$ be a $n$-dimensional oriented and simply connected manifold with Lipschitz continuous boundary $\partial\Omega=\Sigma\cup\Gamma$ with $\Sigma\cap\Gamma=\emptyset$. Given the function spaces $\mathcal{F}_{2}$ and $\mathcal{E}_{2}$, defined in (\ref{space2}), together with the bilinear form:
  \begin{equation}\label{bilinear3}
  \begin{aligned}
  \llangle(f^{1},e^{1}),(f^{2},e^{2})\rrangle=&\int_{\Omega}\big(e_{\eta}^{1}\wedge f_{\eta}^{2}+e_{\eta}^{2}\wedge f_{\eta}^{1}\big)+\int_{\partial\Omega}\big(e_{\phi}^{1}\wedge f_{\phi}^{2}+e_{\phi}^{2}\wedge f_{\phi}^{1}\big)\\
  &+\int_{\Sigma}\big(e_{\Sigma}^{1}\wedge f_{\Sigma}^{2}+e_{\Sigma}^{2}\wedge f_{\Sigma}^{1}\big)+\int_{\Gamma}\big(e_{b}^{1}\wedge f_{b}^{2}+e_{b}^{2}\wedge f_{b}^{1}\big),
  \end{aligned}
  \end{equation}
  where 
  \begin{equation*}
  f^{i}=(f_{\eta}^{i},f_{\phi}^{i},f_{\Sigma}^{i},f_{b}^{i})\in\mathcal{F}_{2},\ e^{i}=(e_{\eta}^{i},e_{\phi}^{i},e_{\Sigma}^{i},e_{b}^{i})\in\mathcal{E}_{2},\ i=1,2.
  \end{equation*}
  Then $D_{2}\subset\mathcal{F}_{2}\times\mathcal{E}_{2}$, defined as
  \begin{equation}\label{dirac2}
  \begin{aligned}
  D_{2}:=
  \Bigg\{ & (f_{\eta},f_{\phi},f_{\Sigma},f_{b},e_{\eta},e_{\phi},e_{\Sigma},e_{b})\in \mathcal{F}_{2}\times\mathcal{E}_{2}\mid\\
  &
    \begin{pmatrix}
     f_{\eta}  \\ 
     d\big(\mathrm{li}_{\phi}(f_{\phi})\big)\\ \\
     f_{\Sigma}|_{\Sigma}
    \end{pmatrix} 
    =
    \begin{pmatrix}
      \begin{aligned}
      &-[\eta,dN_{\phi}(e_{\phi})]_{1},\\ 
      &\ast\Big(\big(\ast e_{\eta}+dN_{\phi}(e_{\phi})\big)\wedge(\ast d\eta)\Big)+[\eta,dN_{\phi}(e_{\phi})]_{1}\\
      &+d\Big(\mathrm{li}\big(\tilde{e}_{\Sigma})+(-1)^{n}\mathrm{li}\big(\langle dN_{\phi}(e_{\phi}),\eta\rangle_{\Lambda^{1}}\big)\Big)\\      &-e_{\phi},
      \end{aligned}
    \end{pmatrix}
     \\
   &
    \begin{pmatrix}
     f_{b}  \\
     e_{b}
    \end{pmatrix} 
    =
    \begin{pmatrix}
     -e_{\phi}|_{\Gamma} \\
     \tilde{e}_{\Sigma}|_{\Gamma} \end{pmatrix}
   \Bigg\},
  \end{aligned}
  \end{equation}
  with $\eta\in\mathring{\mathfrak{B}}^{\ast 1}$ and $\tilde{e}_{\Sigma}=E(e_{\Sigma})$, is a Dirac structure.
  \end{theorem}

  \begin{lemma}\label{lemma15}
    Given $e_{\eta}\in\mathring{\mathfrak{B}}^{(n-1)}$, $e_{\phi}\in H^{-\frac{1}{2}}\Lambda^{n-1}(\partial\Omega)$ and $e_{\Sigma}\in H^{\frac{1}{2}}\Lambda^{0}(\Sigma)$, then $d\big(\mathrm{li}_{\phi}(f_{\phi})\big)$ in (\ref{dirac2}) is well-defined.
  \end{lemma}

  \begin{proof}
    Assume that there exist $\mathrm{li}_{\phi}(f_{\phi}),\ \mathrm{li}'_{\phi}(f_{\phi})\in H^{1}\Lambda^{0}(\Omega)$ such that for two operators $\mathrm{li},\ \mathrm{li}':H^{\frac{1}{2}}\Lambda^{0}(\partial\Omega)\to H^{1}\Lambda^{0}(\Omega)$, we have
    \begin{equation*}
      \begin{aligned}
        d\big(\mathrm{li}_{\phi}(f_{\phi})\big)&=\ast\Big(\big(\ast e_{\eta}+dN_{\phi}(e_{\phi})\big)\wedge(\ast d\eta)\Big)+d\Big(\mathrm{li}\big(\tilde{e}_{\Sigma})+(-1)^{n}\mathrm{li}\big(\langle dN_{\phi}(e_{\phi}),\eta\rangle_{\Lambda^{1}}\big)\Big),\\
        d\big(\mathrm{li}'_{\phi}(f_{\phi})\big)&=\ast\Big(\big(\ast e_{\eta}+dN_{\phi}(e_{\phi})\big)\wedge(\ast d\eta)\Big)+d\Big(\mathrm{li}'\big(\tilde{e}_{\Sigma})+(-1)^{n}\mathrm{li}'\big(\langle dN_{\phi}(e_{\phi}),\eta\rangle_{\Lambda^{1}}\big)\Big).
      \end{aligned}
    \end{equation*}
    Let $\tilde{f}:=\mathrm{li}_{\phi}(f_{\phi})-\mathrm{li}'_{\phi}(f_{\phi})$, then with (\ref{tracelift2}), we directly have that
    \begin{equation}\label{eq96}
      \mathrm{tr}(\tilde{f})=\mathrm{tr}\big(\mathrm{li}_{\phi}(f_{\phi})-\mathrm{li}'_{\phi}(f_{\phi})\big)=0.
    \end{equation}
     Since $\mathrm{li}_{\phi}(\mu)$, $\mathrm{li}_{\phi}'(\mu)$ both satisfy the Laplace equation, see Definition 7.6, we have
    \begin{equation}\label{eq97}
      \delta d\big(\mathrm{li}_{\phi}(f_{\phi})-\mathrm{li}'_{\phi}(f_{\phi})\big)=\delta d\tilde{f}=0,
    \end{equation}
    hence, from (\ref{eq96}) and (\ref{eq97}) we have
    \begin{equation*}
      \tilde{f}=\mathrm{li}_{\phi}(f_{\phi})-\mathrm{li}'_{\phi}(f_{\phi})=0.
    \end{equation*}
    Next, let $\hat{f}=\mathrm{li}(\tilde{e}_{\Sigma})-\mathrm{li}'(\tilde{e}_{\Sigma})+(-1)^{n}\mathrm{li}(\langle dN_{\phi}(e_{\phi}),\eta\rangle_{\Lambda^{1}})-(-1)^{n}\mathrm{li}'(\langle dN_{\phi}(e_{\phi}),\eta\rangle_{\Lambda^{1}})$. Using (\ref{tracelift}), we have 
    \begin{equation}\label{eq133}
      \mathrm{tr}(\hat{f})=0\quad\text{at }\partial\Omega.
    \end{equation}
    Since $\tilde{f}$ is harmonic we obtain 
    \begin{equation}
      0=\delta d\tilde{f}=\delta d\hat{f},
    \end{equation}
    which together with (\ref{eq133}) implies $\hat{f}=0$ in $\Omega$.
  \end{proof}

  In the following, we give the proof of Theorem \ref{thm6}.

  \begin{proof}
    \begin{enumerate}[(1)]

      \item From the definition of the trace-lifting operator $\mathrm{li}_{\phi}$ (\ref{tracelift2}), and taking $\phi=N_{\phi}(e_{\phi}^{i}),\ \psi=\mathrm{li}_{\phi}(f_{\phi}^{j}),\ i,j\in\{1,2\}$ in (\ref{eq82}) with $f_{\phi}^{j}\in H^{\frac{1}{2}}\Lambda^{0}(\partial\Omega)$, $e_{\phi}^{i}\in H^{-\frac{1}{2}}\Lambda^{n-1}(\partial\Omega)$, we obtain
      \begin{equation}\label{eq138}
        \begin{aligned}
          \langle dN_{\phi}(e_{\phi}^{i}),d\big(\mathrm{li}_{\phi}(f_{\phi}^{j})\big)\rangle_{L^{2}\Lambda^{1}(\Omega)}=\int_{\partial\Omega}e_{\phi}^{i}\wedge f_{\phi}^{j}.
        \end{aligned}
      \end{equation}
      Thus, (\ref{bilinear3}) can be rewritten as
      \begin{equation}\label{bilinear4}
        \begin{aligned}
          \llangle(f^{1},e^{1}),(f^{2},e^{2})\rrangle=&\int_{\Omega}\Big(e_{\eta}^{1}\wedge f_{\eta}^{2}+e_{\eta}^{2}\wedge f_{\eta}^{1}+dN_{\phi}(e_{\phi}^{1})\wedge\ast d\big(\mathrm{li}_{\phi}(f_{\phi}^{2})\big)\\
          &+dN_{\phi}(e_{\phi}^{2})\wedge\ast d\big(\mathrm{li}_{\phi}(f_{\phi}^{1})\big)\Big)
          +\int_{\Sigma}\Big(e_{\Sigma}^{1}\wedge f_{\Sigma}^{2}+e_{\Sigma}^{2}\wedge f_{\Sigma}^{1}\Big)\\
          &+\int_{\Gamma}\Big(e_{b}^{1}\wedge f_{b}^{2}+e_{b}^{2}\wedge f_{b}^{1}\Big).
        \end{aligned}
      \end{equation}

      \item Since $e_{\eta}\in\mathring{\mathfrak{B}}^{(n-1)}=d\mathring{H}\Lambda^{n-2}(\Omega)$, it follows that $de_{\eta}=0$ and $\mathrm{tr}(e_{\eta})=0$. Then using (\ref{eq42}) and (\ref{part}), together with $\delta(\ast e_{\eta})=(-1)^{n}\ast(de_{\eta})=0$, we obtain for any $f\in L^{2}\Lambda^{0}(\Omega)$ that 
      \begin{equation*}
        \begin{aligned}
          \int_{\Omega}e_{\eta}\wedge df=&\langle f,\delta(\ast e_{\eta})\rangle_{L^{2}\Lambda^{0}(\Omega)}+(-1)^{n-1}\int_{\partial\Omega}\mathrm{tr}(f)\wedge\mathrm{tr}(e_{\eta})=0.
        \end{aligned}
      \end{equation*}
      Choosing $f=\mathrm{li}_{\phi}(f_{\phi}^{j})-\mathrm{li}\big(\tilde{e}_{\Sigma}^{j})+(-1)^{n-1}\mathrm{li}\big(\langle dN_{\phi}(e_{\phi}^{j}),\eta\rangle_{\Lambda^{1}(\Omega)}\big)$ for $i,j\in\{1,2\}$, then gives 
      \begin{equation*}
        \begin{aligned}
          \int_{\Omega}e_{\eta}^{i}\wedge d\Big(\mathrm{li}_{\phi}(f_{\phi}^{j})-\mathrm{li}(\tilde{e}_{\Sigma}^{j})+(-1)^{n-1}\mathrm{li}\big(\langle dN_{\phi}(e_{\phi}^{j}),\eta\rangle_{\Lambda^{1}}\big)\Big)=0,
        \end{aligned}
      \end{equation*}
      which implies using the expression for $f_{\phi}$ in (\ref{dirac2}) that  
      \begin{equation}\label{eq75}
        \begin{aligned}
          \int_{\Omega}e_{\eta}^{i}\wedge[\eta,dN_{\phi}(e_{\phi}^{j})]_{1}=-\int_{\Omega}e_{\eta}^{i}\wedge\ast\Big((\ast e_{\eta}^{j}+dN_{\phi}(e_{\phi}^{j}))\wedge(\ast d\eta)\Big).
        \end{aligned}
      \end{equation}

      \item Next, we show that $D_{2}\subset D_{2}^{\bot}$. Let $(f^{1},e^{1})\in D_{2}$ be fixed, and consider any $(f^{2},e^{2})\in D_{2}$. From (\ref{bilinear4}), we have that
      \begin{equation}\label{eq78}
        \begin{aligned}
          \llangle(f^{1},e^{1}),(f^{2},e^{2})\rrangle
          =&\int_{\Omega}\Big(-e_{\eta}^{1}\wedge[\eta,dN_{\phi}(e_{\phi}^{2})]_{1}-e_{\eta}^{2}\wedge[\eta,dN_{\phi}(e_{\phi}^{1})]_{1}\\
          &+(-1)^{n-1}dN_{\phi}(e_{\phi}^{1})\wedge\big(\ast e_{\eta}^{2}+dN_{\phi}(e_{\phi}^{2})\big)\wedge(\ast d\eta)\\
          &+dN_{\phi}(e_{\phi}^{1})\wedge\ast d\Big(\mathrm{li}\big(\tilde{e}_{\Sigma}^{2}\big)+(-1)^{n}\mathrm{li}\big(\langle dN_{\phi}(e_{\phi}^{2}),\eta\rangle_{\Lambda^{1}}\big)\Big)\\
          &+dN_{\phi}(e_{\phi}^{1})\wedge\ast[\eta,dN_{\phi}(e_{\phi}^{2})]_{1}\\
          &+(-1)^{n-1}dN_{\phi}(e_{\phi}^{2})\wedge\big(\ast e_{\eta}^{1}+dN_{\phi}(e_{\phi}^{1})\big)\wedge(\ast d\eta)\\
          &+dN_{\phi}(e_{\phi}^{2})\wedge\ast d\Big(\mathrm{li}\big(\tilde{e}_{\Sigma}^{1}\big)+(-1)^{n}\mathrm{li}\big(\langle dN_{\phi}(e_{\phi}^{1}),\eta\rangle_{\Lambda^{1}}\big)\Big)\\          
          &+dN_{\phi}(e_{\phi}^{2})\wedge\ast[\eta,dN_{\phi}(e_{\phi}^{1})]_{1}\Big)\\
          &-\int_{\Sigma}\big(e_{\Sigma}^{1}\wedge e_{\phi}^{2}+e_{\Sigma}^{2}\wedge e_{\phi}^{1}\big)-\int_{\Gamma}\big(\tilde{e}_{\Sigma}^{1}\wedge e_{\phi}^{2}+\tilde{e}_{\Sigma}^{2}\wedge e_{\phi}^{1}\big).
        \end{aligned}
      \end{equation}

      Using (\ref{eq82}) with $g_{\partial}=e_{\phi}^{i}$ and the definition of the lifting operator (\ref{tracelift}), we obtain for $i,j\in\{1,2\}$
      \begin{equation}\label{eq76}
        \begin{aligned}
          &\int_{\Omega}dN_{\phi}(e_{\phi}^{i})\wedge\ast d\Big(\mathrm{li}\big(\tilde{e}_{\Sigma}^{j}\big)+(-1)^{n}\mathrm{li}\big(\langle dN_{\phi}(e_{\phi}^{j}),\eta\rangle_{\Lambda^{1}}\big)\Big)\\
          =&\Big\langle dN_{\phi}(e_{\phi}^{i}),d\Big(\mathrm{li}\big(\tilde{e}_{\Sigma}^{j}\big)+(-1)^{n}\mathrm{li}\big(\langle dN_{\phi}(e_{\phi}^{j}),\eta\rangle_{\Lambda^{1}}\big)\Big)\Big\rangle_{L^{2}\Lambda^{1}(\Omega)}\\
          =&\int_{\partial\Omega}\tilde{e}_{\Sigma}^{j}\wedge e_{\phi}^{i}+(-1)^{n}\int_{\partial\Omega}\langle dN_{\phi}(e_{\phi}^{j}),\eta\rangle_{\Lambda^{1}}\wedge e_{\phi}^{i}.
        \end{aligned}
      \end{equation}
      Next, from (\ref{lie}) and Lemma \ref{lemma2}, we have for $i,j\in\{1,2\}$
      \begin{equation}\label{eq77}
        \begin{aligned}
          \int_{\Omega}dN_{\phi}(e_{\phi}^{i})\wedge\ast[\eta,dN_{\phi}(e_{\phi}^{j})]_{1}=&(-1)^{n-1}\int_{\Omega}dN_{\phi}(e_{\phi}^{i})\wedge\ast\delta\big(\eta\wedge dN_{\phi}(e_{\phi}^{j})\big)\\
          =&(-1)^{n-1}\int_{\partial\Omega}\big(\langle dN_{\phi}(e_{\phi}^{i}),\eta\rangle_{\Lambda^{1}}i_{\mathcal{N}}dN_{\phi}(e_{\phi}^{j})\big)v_{\Sigma}\\
          =&(-1)^{n-1}\int_{\partial\Omega}\langle dN_{\phi}(e_{\phi}^{i}),\eta\rangle_{\Lambda^{1}}\wedge e_{\phi}^{j},
        \end{aligned}
      \end{equation}
      where in the first step we use that fact that $\delta dN_{\phi}(e_{\phi}^{i})=0,i=1,2$ and $\delta\eta=0$ for $\eta\in\mathring{\mathfrak{B}}^{\ast 1}$. To show the last step, let $u=dN_{\phi}(e_{\phi}^{j})$. Then $u$ satisfies
      \begin{equation}\label{eq95}
        \delta u=0 \text{ in }\Omega,\quad\ast\boldsymbol{n}(u)=e_{\phi}^{j}\text{ at }\partial\Omega.
      \end{equation}
      By applying \citep[Proposition 3.20]{csato2011pullback}, we have that
      \begin{equation*}
        \boldsymbol{n}(u)=\mathcal{N}^{\flat}\wedge(i_{\mathcal{N}}u).
      \end{equation*}
      With (\ref{eq50}) and $i_{\mathcal{N}}u\in L^{2}\Lambda^{0}(\partial\Omega)$, we have
      \begin{equation*}
        \ast\boldsymbol{n}(u)=\ast\big(\mathcal{N}^{\flat}\wedge(i_{\mathcal{N}}u)\big)=(i_{\mathcal{N}}u)\wedge(\ast\mathcal{N}^{\flat})=(i_{\mathcal{N}}u)v_{\Sigma}.
      \end{equation*}
      Finally, using (\ref{eq95}), we obtain that $i_{\mathcal{N}}dN_{\phi}(e_{\phi}^{j})v_{\Sigma}=e_{\phi}^{j}$.
      Substituting (\ref{eq75}), (\ref{eq76}) and (\ref{eq77}) into (\ref{eq78}), we have 
      \begin{align*}
          &\llangle(f^{1},e^{1}),(f^{2},e^{2})\rrangle\\
          =&\int_{\Omega}\Big((-1)^{n-1}\big(\ast e_{\eta}^{1}+dN_{\phi}(e_{\phi}^{1})\big)\wedge\big(\ast e_{\eta}^{2}+dN_{\phi}(e_{\phi}^{2})\big)\wedge(\ast d\eta)\\
          &+(-1)^{n-1}\big(\ast e_{\eta}^{2}+dN_{\phi}(e_{\phi}^{2})\big)\wedge\big(\ast e_{\eta}^{1}+dN_{\phi}(e_{\phi}^{1})\big)\wedge(\ast d\eta)\Big)\\
          &+\int_{\partial\Omega}\Big(\tilde{e}_{\Sigma}^{2}\wedge e_{\phi}^{1}+(-1)^{n}\langle dN_{\phi}(e_{\phi}^{2}),\eta\rangle_{\Lambda^{1}}\wedge e_{\phi}^{1} \Big)\\
          &+\int_{\partial\Omega}\Big(\tilde{e}_{\Sigma}^{1}\wedge e_{\phi}^{2}+(-1)^{n}\langle dN_{\phi}(e_{\phi}^{1}),\eta\rangle_{\Lambda^{1}}\wedge e_{\phi}^{2} \Big)\\ 
          &+(-1)^{n-1}\int_{\partial\Omega}\Big(\langle dN_{\phi}(e_{\phi}^{1}),\eta\rangle_{\Lambda^{1}}\wedge e_{\phi}^{2}+\langle dN_{\phi}(e_{\phi}^{2}),\eta\rangle_{\Lambda^{1}}\wedge e_{\phi}^{1}\Big)\\
          &-\int_{\partial\Omega}(\tilde{e}_{\Sigma}^{1}\wedge e_{\phi}^{2}+\tilde{e}_{\Sigma}^{2}\wedge e_{\phi}^{1})\\
          =&\ 0.
        \end{align*}
        Therefore, $(f^{2},e^{2})\in D^{\bot}$ which implies that $D_{2}\subset D_{2}^{\bot}$.

        \item Finally, we show that $D_{2}^{\bot}\subset D_{2}$. Let $(f^{1},e^{1})\in D_{2}^{\bot}\subset\mathcal{F}_{2}\times\mathcal{E}_{2}$. Then,
        \begin{equation*}
          \llangle(f^{1},e^{1}),(f^{2},e^{2})\rrangle=0,\quad\forall(f^{2},e^{2})\in D_{2},
        \end{equation*}
        which equals
        \begin{equation}\label{eq79}
          \begin{aligned}
            &\int_{\Omega}\Big(-e_{\eta}^{1}\wedge[\eta,dN_{\phi}(e_{\phi}^{2})]_{1}+e_{\eta}^{2}\wedge f_{\eta}^{1}+(-1)^{n-1}dN_{\phi}(e_{\phi}^{1})\wedge\big(\ast e_{\eta}^{2}+dN_{\phi}(e_{\phi}^{2})\big)\\
            &\wedge(\ast d\eta)+dN_{\phi}(e_{\phi}^{1})\wedge\ast d\big(\mathrm{li}(\tilde{e}_{\Sigma}^{2})+(-1)^{n}\mathrm{li}(\langle dN_{\phi}(e_{\phi}^{2}),\eta\rangle_{\Lambda^{1}})\big)\\
            &+dN_{\phi}(e_{\phi}^{1})\wedge\ast[\eta,dN_{\phi}(e_{\phi}^{2})]_{1}
            +dN(e_{\phi}^{2})\wedge\ast d(\mathrm{li}_{\phi}(f_{\phi}^{1}))\Big)\\
            &+\int_{\Sigma}(-e_{\Sigma}^{1}\wedge e_{\phi}^{2}+e_{\Sigma}^{2}\wedge f_{\Sigma}^{1})+\int_{\Gamma}(-e_{b}^{1}\wedge e_{\phi}^{2}+\tilde{e}_{\Sigma}^{2}\wedge f_{b}^{1})=0.
          \end{aligned}
        \end{equation}
        Next, we take $\tilde{e}_{\Sigma}^{2}\in H_{00}^{\frac{1}{2}}\Lambda^{0}(\Sigma)$ which implies $\tilde{e}_{\Sigma}^{2}|_{\Gamma}=0$, $e_{\phi}^{2}\in H^{-\frac{1}{2}}\Lambda^{n-1}(\partial\Omega)$ such that $e_{\phi}^{2}|_{\Gamma}=0$.
        Using (\ref{eq75}), (\ref{eq76}) and (\ref{eq77}), then (\ref{eq79}) can be rewritten as
        \begin{equation}\label{eq92}
          \begin{aligned}
            &\int_{\Omega}\Big(e_{\eta}^{1}\wedge\ast\big((\ast e_{\eta}^{2}+dN_{\phi}(e_{\phi}^{2}))\wedge(\ast d\eta)\big)+e_{\eta}^{2}\wedge f_{\eta}^{1}+dN_{\phi}(e_{\phi}^{2})\\
            &\wedge\ast d(\mathrm{li}_{\phi}(f_{\phi}^{1}))
            +(-1)^{n-1}dN_{\phi}(e_{\phi}^{1})\wedge\big(\ast e_{\eta}^{2}+dN_{\phi}(e_{\phi}^{2})\big)\wedge(\ast d\eta)\Big)\\
            &+\int_{\partial\Omega}\big(\tilde{e}_{\Sigma}^{2}\wedge e_{\phi}^{1}+(-1)^{n}\langle dN_{\phi}(e_{\phi}^{2}),\eta\rangle_{\Lambda^{1}}\wedge e_{\phi}^{1}\big)\\
            &+(-1)^{n-1}\int_{\partial\Omega}\langle dN_{\phi}(e_{\phi}^{1}),\eta\rangle_{\Lambda^{1}}\wedge e_{\phi}^{2}+\int_{\Sigma}(-e_{\Sigma}^{1}\wedge e_{\phi}^{2}+e_{\Sigma}^{2}\wedge f_{\Sigma}^{1})=0.
          \end{aligned}
        \end{equation}
        Using the boundary conditions $\tilde{e}_{\Sigma}^{2}|_{\Gamma}=e_{\phi}^{2}|_{\Gamma}=0$, we have
        \begin{equation*}
          \int_{\partial\Omega}\tilde{e}_{\Sigma}^{2}\wedge e_{\phi}^{1}-\int_{\Sigma}e_{\Sigma}^{1}\wedge e_{\phi}^{2}=\int_{\Sigma}e_{\Sigma}^{2}\wedge e_{\phi}^{1}-\int_{\partial\Omega}\tilde{e}_{\Sigma}^{1}\wedge e_{\phi}^{2}.
        \end{equation*}
        and (\ref{eq92}) becomes
        \begin{equation*}
          \begin{aligned}
            &\int_{\Omega}\Big((-1)^{n-1}\big(\ast e_{\eta}^{1}+dN_{\phi}(e_{\phi}^{1})\big)\wedge\big(\ast e_{\eta}^{2}+dN_{\phi}(e_{\phi}^{2})\big)\wedge(\ast d\eta)+e_{\eta}^{2}\wedge f_{\eta}^{1}\\
            &+dN(e_{\phi}^{2})\wedge\ast d(\mathrm{li}_{\phi}(f_{\phi}^{1})) \Big)
            +\int_{\partial\Omega}\big(-\tilde{e}_{\Sigma}^{1}\wedge e_{\phi}^{2}+(-1)^{n-1}\langle dN_{\phi}(e_{\phi}^{1}),\eta\rangle_{\Lambda^{1}}\wedge e_{\phi}^{2} \big)\\
            &+(-1)^{n}\int_{\partial\Omega}\langle dN_{\phi}(e_{\phi}^{2}),\eta\rangle_{\Lambda^{1}}\wedge e_{\phi}^{1}+\int_{\Sigma}\big(e_{\Sigma}^{2}\wedge f_{\Sigma}^{1}+e_{\Sigma}^{2}\wedge e_{\phi}^{1}\big)=0.
          \end{aligned}
        \end{equation*}
        Again, applying (\ref{eq75}), (\ref{eq76}) and (\ref{eq77}), we have that
        \begin{equation}\label{eq83}
          \begin{aligned}
            &\int_{\Omega}\Big((-1)^{n-1}\big(\ast e_{\eta}^{1}+dN_{\phi}(e_{\phi}^{1})\big)\wedge\big(\ast e_{\eta}^{2}+dN_{\phi}(e_{\phi}^{2})\big)\wedge(\ast d\eta)+e_{\eta}^{2}\wedge f_{\eta}^{1}\\
            &+dN(e_{\phi}^{2})\wedge\ast d(\mathrm{li}_{\phi}(f_{\phi}^{1}))-dN_{\phi}(e_{\phi}^{2})\wedge\ast d\big(\mathrm{li}(\tilde{e}_{\Sigma}^{1})+(-1)^{n}\mathrm{li}(\langle dN_{\phi}(e_{\phi}^{1}),\eta\rangle_{\Lambda^{1}})\big)\\
            &-dN_{\phi}(e_{\phi}^{2})\wedge\ast[\eta,dN_{\phi}(e_{\phi}^{1})]_{1}\Big)+\int_{\Sigma}\big(e_{\Sigma}^{2}\wedge f_{\Sigma}^{1}+e_{\Sigma}^{2}\wedge e_{\phi}^{1}\big)=0.
          \end{aligned}
        \end{equation}
        Since the first term in (\ref{eq83}) can be split into two parts, it can be further evaluated as
        \begin{equation}\label{eq84}
          \begin{aligned}
            &\big(\ast e_{\eta}^{1}+dN_{\phi}(e_{\phi}^{1})\big)\wedge\big(\ast e_{\eta}^{2}+dN_{\phi}(e_{\phi}^{2})\big)\wedge(\ast d\eta)\\
            =&(-1)^{n}e_{\eta}^{2}\wedge\ast\Big(\big(\ast e_{\eta}^{1}+dN_{\phi}(e_{\phi}^{1})\big)\wedge(\ast d\eta)\Big)-dN_{\phi}(e_{\phi}^{2})\wedge\big(\ast e_{\eta}^{1}+dN_{\phi}(e_{\phi}^{1})\big)\wedge(\ast d\eta).
          \end{aligned}
        \end{equation}
        Finally, using (\ref{eq75}) and (\ref{eq84}), then (\ref{eq83}) can be transformed into 
        \begin{equation*}
          \begin{aligned}
            &\int_{\Omega}\Big(e_{\eta}^{2}\wedge\big(f_{\eta}^{1}+[\eta,dN_{\phi}(e_{\phi}^{1})]_{1}\big)+dN_{\phi}(e_{\phi}^{2})\wedge\ast\Big(d\big(\mathrm{li}_{\phi}(f_{\phi}^{1}\big))\\
            &-\ast\big(\big(\ast e_{\eta}^{1}+dN_{\phi}(e_{\phi}^{1})\big)\wedge(\ast d\eta)\big)-d\big(\mathrm{li}(\tilde{e}_{\Sigma}^{1})+(-1)^{n}\mathrm{li}(\langle dN_{\phi}(e_{\phi}^{1}),\eta\rangle_{\Lambda^{1}})\big)\\
            &-[\eta,dN_{\phi}(e_{\phi}^{1})]_{1}\Big)+\int_{\Sigma}\big(e_{\Sigma}^{2}\wedge f_{\Sigma}^{1}+e_{\Sigma}^{2}\wedge e_{\phi}^{1}\big)=0.
          \end{aligned}
        \end{equation*}
        Since $e_{\eta}^{2},\ e_{\phi}^{2}, e_{\Sigma}^{2}$ are arbitrary, the integral can only be zero if 
        \begin{equation}\label{eq80}
          \begin{cases}
            f_{\eta}^{1}&=-[\eta,dN_{\phi}(e_{\phi}^{1})]_{1},\\
            d\big(\mathrm{li}_{\phi}(f_{\phi}^{1})\big)&=\ast\Big((\ast e_{\eta}^{1}+dN_{\phi}(e_{\phi}^{1}))\wedge(\ast d\eta)\Big)+d\Big(\mathrm{li}(\tilde{e}_{\Sigma}^{1})\\
            &\quad +(-1)^{n}\mathrm{li}(\langle dN_{\phi}(e_{\phi}^{1}),\eta\rangle_{\Lambda^{1}})\Big)+[\eta,dN_{\phi}(e_{\phi}^{1})]_{1},\\
            f_{\Sigma}^{1}&=-e_{\phi}^{1}.
          \end{cases}
        \end{equation}
        Substituting (\ref{eq80}) into (\ref{eq79}), following the same steps as in Step (3), we obtain
        \begin{equation*}
          \int_{\Gamma}\big((-e_{b}^{1}+\tilde{e}_{\Sigma}^{1})\wedge e_{\phi}^{2}+\tilde{e}_{\Sigma}^{2}\wedge(f_{b}^{1}+e_{\phi}^{1}) \big)=0.
        \end{equation*}
        Hence,
        \begin{equation*}
          f_{b}^{1}=-e_{\phi}^{1}|_{\Gamma},\quad e_{b}^{1}=\tilde{e}_{\Sigma}^{1}|_{\Gamma}.  
        \end{equation*}

    \end{enumerate}
  \end{proof}

  \begin{corollary}
    Given an oriented and simply connected $n$-dimensional manifold $\Omega$ with Lipschitz continuous boundary $\partial\Omega$.
    Given the essential boundary condition (\ref{vor4}f) and the port variable $f_{b}=-g$.
    The distributed-parameter port-Hamiltonian system for the incompressible Euler equations with a free surface (\ref{vor4}) with respect to the variables $(\eta,\phi_{\partial},\Sigma)\in\mathring{\mathfrak{B}}^{\ast 1}\times H^{\frac{1}{2}}\Lambda^{0}(\partial\Omega)\times H^{-\frac{1}{2}}\Lambda^{n-1}(\partial\Omega)$,
    Dirac structure $D_{2}$ (\ref{dirac2}) and Hamiltonian (\ref{hami4}), is given by
    \begin{equation}\label{coro4}
      \begin{aligned}    
        &
      \begin{pmatrix}
        -\eta_{t} \\
        -d\phi_{t}\\
        -\Sigma_{t}
       \end{pmatrix} 
       =
       \begin{pmatrix}
        0 &  -[\eta,dN_{\phi}(\cdot)]_{1} & 0 \\
        \ast\big((\ast\cdot)\wedge\ast d\eta\big)& A_{22} & d\big(\mathrm{li}_{\Sigma}(E(\cdot))\big)\\
        0 & -1 & 0
        \end{pmatrix}
        \begin{pmatrix}
          \frac{\delta\tilde{H}}{\delta\eta} \\
          \frac{\delta\tilde{H}}{\delta\phi_{\partial}}\\
          \frac{\delta\tilde{H}}{\delta\Sigma}
         \end{pmatrix},
        \end{aligned}
      \end{equation}
      where $A_{22}=\ast\big(dN_{\phi}(\cdot)\wedge\ast d\eta \big)+(-1)^{n}d\big(\mathrm{li}(\langle dN_{\phi}(\cdot),\eta\rangle_{\Lambda^{1}})\big)+[\eta,dN_{\phi}(\cdot)]_{1}$, 
      with port variables $f_{b},e_{b}$ defined as
      \begin{equation}\label{coro5}
        \begin{aligned} 
        &
       \begin{pmatrix}
        f_{b}  \\
        e_{b}
       \end{pmatrix} 
       =
       \begin{pmatrix}
        -1\quad & 0 \\
        0\quad & E(\cdot)
       \end{pmatrix}
       \begin{pmatrix}
        \frac{\delta\tilde{H}}{\delta\phi_{\partial}} \\
        \frac{\delta\tilde{H}}{\delta\Sigma}
       \end{pmatrix}.
     \end{aligned}
    \end{equation}
  \end{corollary}

  \begin{proof}
    Substituting the functional derivatives
    \begin{equation*}
      e_{\eta}=\frac{\delta\tilde{H}}{\delta\eta},\quad e_{\phi}=\frac{\delta\tilde{H}}{\delta\phi_{\partial}},\quad e_{\Sigma}=\frac{\delta\tilde{H}}{\delta\Sigma},
    \end{equation*}
    given in (\ref{func7}) into (\ref{coro4}), with the essential boundary condition (\ref{vor4}f) and the port variable $f_{b}=-g$, we immediately obtain (\ref{vor4}).
  \end{proof}

  From the bilinear form (\ref{bilinear4}), we can define the following bilinear form 
  \begin{equation*}
    [e_{1},e_{2}]_{D}=\int_{\Omega}\Big(e_{\eta}^{1}\wedge f_{\eta}^{2}+dN_{\phi}(e_{\phi}^{1})\wedge\ast d\big(\mathrm{li}_{\phi}(f_{\phi}^{2})\big)\Big)+\int_{\Sigma}e_{\Sigma}^{1}\wedge f_{\Sigma}^{2}+\int_{\Gamma}e_{b}^{1}\wedge f_{b}^{2}.
  \end{equation*}

  For $\frac{\delta\tilde{\mathcal{F}}}{\delta\eta},\frac{\delta\tilde{\mathcal{G}}}{\delta\eta}\in\mathring{\mathfrak{B}}^{(n-1)}$, $\frac{\delta\tilde{\mathcal{F}}}{\delta\phi_{\partial}},\frac{\delta\tilde{\mathcal{G}}}{\delta\phi_{\partial}}\in H^{-\frac{1}{2}}\Lambda^{n-1}(\partial\Omega)$, $\frac{\delta\tilde{\mathcal{F}}}{\delta\Sigma},\frac{\delta\tilde{\mathcal{G}}}{\delta\Sigma}\in H^{\frac{1}{2}}\Lambda^{0}(\Sigma)$, the Poisson bracket $\{\cdot,\cdot\}_{D}(\eta,\phi_{\partial},\Sigma)$ associated to the Dirac structure (\ref{dirac2}) is defined as
  \begin{equation}\label{bracket2}
    \begin{aligned}
      &\{\tilde{\mathcal{F}},\tilde{\mathcal{G}}\}_{D}(\eta,\phi_{\partial},\Sigma)\\
      =&-\int_{\Omega}\frac{\delta\tilde{\mathcal{G}}}{\delta\eta}\wedge[\eta,dN_{\phi}(\frac{\delta\tilde{\mathcal{F}}}{\delta\phi_{\partial}})]_{1}+\int_{\Omega}dN_{\phi}(\frac{\delta\tilde{\mathcal{G}}}{\delta\phi_{\partial}})\wedge\ast\Big(\ast\big((\ast\frac{\delta\tilde{\mathcal{F}}}{\delta\eta}+dN_{\phi}(\frac{\delta\tilde{\mathcal{F}}}{\delta\phi_{\partial}}))\wedge\ast d\eta\big)\\
      &+(-1)^{n}d\big(\mathrm{li}(\langle dN_{\phi}(\frac{\delta\tilde{\mathcal{F}}}{\delta\phi_{\partial}}),\eta\rangle_{\Lambda^{1}})\big)+[\eta,dN_{\phi}(\frac{\delta\tilde{\mathcal{F}}}{\delta\phi_{\partial}})]_{1}+d\big(\mathrm{li}(E(\frac{\delta\tilde{\mathcal{F}}}{\delta\Sigma})\big)\Big)\\
      &-\int_{\Sigma}\frac{\delta\tilde{\mathcal{G}}}{\delta\Sigma}\wedge\frac{\delta\tilde{\mathcal{F}}}{\delta\phi_{\partial}}-\int_{\Gamma}E(\frac{\delta\tilde{\mathcal{G}}}{\delta\Sigma})\wedge\frac{\delta\tilde{\mathcal{F}}}{\delta\phi_{\partial}},
    \end{aligned}
  \end{equation}
  with $\eta\in\mathring{\mathfrak{B}}^{\ast 1}$.

  \begin{remark}
    Using (\ref{eq75}), (\ref{eq76}) and (\ref{eq77}), we have following equalities
    \begin{enumerate}  
      \item[(1)]
      \begin{equation}\label{eq124}
      \begin{aligned}
        -\int_{\Omega}\frac{\delta\tilde{\mathcal{G}}}{\delta\eta}\wedge[\eta,dN_{\phi}(\frac{\delta\tilde{\mathcal{F}}}{\delta\phi_{\partial}})]_{1}=&\int_{\Omega}\frac{\delta\tilde{\mathcal{G}}}{\delta\eta}\wedge\ast\Big(\big(\ast\frac{\delta\tilde{\mathcal{F}}}{\delta\eta}+dN_{\phi}(\frac{\delta\tilde{\mathcal{F}}}{\delta\phi_{\partial}})\big)\wedge(\ast d\eta)\Big),
      \end{aligned}
    \end{equation}
    \item[(2)]
    \begin{equation}\label{eq125}
      \begin{aligned}
        &\int_{\Omega}dN_{\phi}(\frac{\delta\tilde{\mathcal{G}}}{\delta\phi_{\partial}})\wedge\ast\Big((-1)^{n}d\big(\mathrm{li}(\langle dN_{\phi}(\frac{\delta\tilde{\mathcal{F}}}{\delta\phi_{\partial}}),\eta\rangle_{\Lambda^{1}})\big)+d\big(\mathrm{li}(E(\frac{\delta\tilde{\mathcal{F}}}{\delta\Sigma}))\big)\Big)\\
        =&\int_{\partial\Omega}\Big(E(\frac{\delta\tilde{\mathcal{F}}}{\delta\Sigma})+(-1)^{n}\langle dN_{\phi}(\frac{\delta\tilde{\mathcal{F}}}{\delta\phi_{\partial}}),\eta\rangle_{\Lambda^{1}}\Big)\wedge\frac{\delta\tilde{\mathcal{G}}}{\delta\phi_{\partial}},
      \end{aligned}
    \end{equation}
    \item[(3)]
    \begin{equation}\label{eq126}
      \begin{aligned}
        \int_{\Omega}dN_{\phi}(\frac{\delta\tilde{\mathcal{G}}}{\delta\phi_{\partial}})\wedge\ast[\eta,dN_{\phi}(\frac{\delta\tilde{\mathcal{F}}}{\delta\phi_{\partial}})]_{1}=(-1)^{n-1}\int_{\partial\Omega}\langle dN_{\phi}(\frac{\delta\tilde{\mathcal{G}}}{\delta\phi_{\partial}}),\eta\rangle_{\Lambda^{1}}\wedge\frac{\delta\tilde{\mathcal{F}}}{\delta\phi_{\partial}}.
      \end{aligned}
    \end{equation}
    Substituting (\ref{eq124}), (\ref{eq125}) and (\ref{eq126}) into (\ref{bracket2}), we can rewritte the Poisson bracket as
    \begin{equation}\label{bracket4}
      \begin{aligned}
        \{\tilde{\mathcal{F}},\tilde{\mathcal{G}}\}_{D}(\eta,\phi_{\partial},\Sigma)=&(-1)^{n-1}\int_{\Omega}(\ast d\eta)\wedge\big(\ast\frac{\delta\tilde{\mathcal{G}}}{\delta\eta}+dN_{\phi}(\frac{\delta\tilde{\mathcal{G}}}{\delta\phi_{\partial}})\big)\wedge\big(\ast\frac{\delta\tilde{\mathcal{F}}}{\delta\eta}+dN_{\phi}(\frac{\delta\tilde{\mathcal{F}}}{\delta\phi_{\partial}})\big)\\
        &+\int_{\partial\Omega}\Big(E\big(\frac{\delta\tilde{\mathcal{F}}}{\delta\Sigma}\big)+(-1)^{n}\langle dN_{\phi}\big(\frac{\delta\tilde{\mathcal{F}}}{\delta\phi_{\partial}}\big),\eta\rangle_{\Lambda^{1}}\Big)\wedge\frac{\delta\tilde{\mathcal{G}}}{\delta\phi_{\partial}}\\
        &-\int_{\partial\Omega}\Big(E\big(\frac{\delta\tilde{\mathcal{G}}}{\delta\Sigma}\big)+(-1)^{n}\langle dN_{\phi}\big(\frac{\delta\tilde{\mathcal{G}}}{\delta\phi_{\partial}}\big),\eta\rangle_{\Lambda^{1}}\Big)\wedge\frac{\delta\tilde{\mathcal{F}}}{\delta\phi_{\partial}}.                 
      \end{aligned}
      \end{equation}
    \end{enumerate}
  \end{remark}

  \begin{remark}
    The bracket $\{\cdot,\cdot\}_{D}(\eta,\phi_{\partial},\Sigma)$ (\ref{bracket4}) is linear, anti-symmetric and satisfies the Jacobi identity, which follows directly from the fact that $\{\cdot,\cdot\}_{D}(\eta,\phi_{\partial},\Sigma)$ has the same structure as the Poisson bracket $\{\cdot,\cdot\}(\eta,\phi_{\partial},\Sigma)$ stated in Lemma \ref{lemma3}. The bracket $\{\cdot,\cdot\}_{D}(\eta,\phi_{\partial},\Sigma)$ is thus a Poisson bracket and the system generated by the Dirac structure $D_{2}$ is a Poisson structure.
  \end{remark}
 
\begin{theorem}\label{thm7}
  Let $\Omega\subset\mathbb{R}^{n},n\in\{1,2,3\}$ be an oriented $n$-dimensional connected manifold with Lipschitz continuous boundary $\partial\Omega=\Sigma\cup\Gamma$ with free surface $\Sigma$ and fixed boundary $\Gamma$, $\Sigma\cap\Gamma=\emptyset$. 
  For any functionals $\tilde{\mathcal{F}}(\eta,\phi_{\partial},\Sigma):\mathring{\mathfrak{B}}^{\ast 1}\times H^{\frac{1}{2}}\Lambda^{0}(\partial\Omega)\times H^{-\frac{1}{2}}\Lambda^{n-1}(\Sigma)\to\mathbb{R}$, the port-Hamiltonian formulation of the incompressible Euler equations with a free surface (\ref{coro4}),(\ref{coro5}) is equivalent to
  \begin{equation}\label{eq127}
    \dot{\tilde{\mathcal{F}}}(\eta,\phi_{\partial},\Sigma)=\{\tilde{\mathcal{F}},\tilde{H}\}_{D}(\eta,\phi_{\partial},\Sigma)-\int_{\Gamma}E(\frac{\delta\tilde{\mathcal{F}}}{\delta\Sigma})\wedge\ast\boldsymbol{n}(d\phi),
  \end{equation}
  with the Hamiltonian $\tilde{H}$ given by (\ref{hami4}).
\end{theorem}

The proof to Theorem \ref{thm7} is identical to Theorem \ref{thm5}.
  
  \begin{corollary}
    Given the Hamiltonian functional $\tilde{H}:\mathring{\mathfrak{B}}^{\ast 1}\times H^{\frac{1}{2}}\Lambda^{0}(\partial\Omega)\times H^{-\frac{1}{2}}\Lambda^{n-1}(\Sigma)\to\mathbb{R}$ stated in (\ref{hami4}), 
    and the essential boundary conditions (\ref{vor4}e), (\ref{vor4}f), the rate of change of the Hamiltonian for the incompressible Euler equations with a free surface in the $(\eta,\phi_{\partial},\Sigma)$ variables (\ref{vor4}) is equal to
    \begin{equation*}
      \dot{\tilde{H}}=-\int_{\Gamma}\mathrm{tr}(h)\wedge g.
    \end{equation*}
  \end{corollary}

\subsection{Dirac structure in terms of the $(\omega,\phi_{\partial},\Sigma)$ variables}

Using the results presented in Section 4.3, we will present in this section the Dirac structure and port-Hamiltonian formulation for the $n$-dimensional ($n\in\{2,3\}$) inviscid incompressible Euler equations with a free surface in terms of the $(\omega,\phi_{\partial},\Sigma)$ variables.
We will first define the following linear spaces for the Dirac structure
\begin{equation}\label{space9}
\begin{aligned} 
\mathcal{F}_{3}&:=\mathring{V}\Lambda^{2}(\Omega)\times H^{\frac{1}{2}}\Lambda^{0}(\partial\Omega)\times H^{-\frac{1}{2}}\Lambda^{n-1}(\Sigma)\times H^{-\frac{1}{2}}\Lambda^{n-1}(\Gamma),\\
\mathcal{E}_{3}&:=\mathring{V}^{\ast}\Lambda^{n-2}(\Omega)\times H^{-\frac{1}{2}}\Lambda^{n-1}(\partial\Omega)\times H^{\frac{1}{2}}\Lambda^{0}(\Sigma)\times H^{\frac{1}{2}}\Lambda^{0}(\Gamma),
\end{aligned}
\end{equation}
where the spaces $\mathcal{F}_{3}$ and $\mathcal{E}_{3}$ are dual to each other with a non-degenerate pairing using $L^{2}$ as pivot space.

The following theorem states the Dirac structure with respect to the $(\omega,\phi_{\partial},\Sigma)$ variables.

\begin{theorem}\label{thm2}
Let $\Omega\subset\mathbb{R}^{n}$, $n\in\{2,3\}$ be a $n$-dimensional oriented and simply connected manifold with boundary $\partial\Omega=\Sigma\cup\Gamma$ and $\Sigma\cap\Gamma=\emptyset$. 
Assume $\omega=dv\in\mathring{V}\Lambda^{2}(\Omega)$ is the solution of the incompressible Euler equations with a free surface in the $(\omega,\phi_{\partial},\Sigma)$ variables (\ref{vor1})--(\ref{vor3}).
Given the function spaces $\mathcal{F}_{3}$ and $\mathcal{E}_{3}$, defined in (\ref{space9}), together with the bilinear form
\begin{equation}\label{bilinear}
\begin{aligned}
\llangle(f^{1},e^{1}),(f^{2},e^{2})\rrangle=&\int_{\Omega}\big(e_{\omega}^{1}\wedge f_{\omega}^{2}+e_{\omega}^{2}\wedge f_{\omega}^{1}\big)+\int_{\partial\Omega}\big(e_{\phi}^{1}\wedge f_{\phi}^{2}+e_{\phi}^{2}\wedge f_{\phi}^{1}\big)\\
&+\int_{\Sigma}\big(e_{\Sigma}^{1}\wedge f_{\Sigma}^{2}+e_{\Sigma}^{2}\wedge f_{\Sigma}^{1}\big)+\int_{\Gamma}\big(e_{b}^{1}\wedge f_{b}^{2}+e_{b}^{2}\wedge f_{b}^{1}\big),
\end{aligned}
\end{equation}
where 
\begin{equation*}
f^{i}=(f_{\omega}^{i},f_{\phi}^{i},f_{\Sigma}^{i},f_{b}^{i})\in\mathcal{F}_{3},\ e^{i}=(e_{\omega}^{i},e_{\phi}^{i},e_{\Sigma}^{i},e_{b}^{i})\in\mathcal{E}_{3},\ i=1,2.
\end{equation*}
Then $D_{3}\subset\mathcal{F}_{3}\times\mathcal{E}_{3}$, defined as
\begin{equation}\label{dirac7}
\begin{aligned}
D_{3}:=\Bigg\{ & (f_{\omega},f_{\phi},f_{\Sigma},f_{b},e_{\omega},e_{\phi},e_{\Sigma},e_{b})\in \mathcal{F}_{3}\times\mathcal{E}_{3}\mid\\
&
  \begin{pmatrix}
   f_{\omega}  \\
   d\big(\mathrm{li}_{\phi}(f_{\phi})\big)\\
   \par \\
   f_{\Sigma}|_{\Sigma}
  \end{pmatrix} 
  =
  \begin{pmatrix}
   d\ast\Big(\big((-1)^{n-1}\ast de_{\omega}+dN_{\phi}(e_{\phi})\big)\wedge\ast\omega\Big) \\
   \ast\Big(\big((-1)^{n-1}\ast de_{\omega}+dN_{\phi}(e_{\phi})\big)\wedge\ast\omega\Big)+[\delta N_{\beta}(\omega),dN_{\phi}(e_{\phi})]_{1}\\
   +d\big(\mathrm{li}(\tilde{e}_{\Sigma})+(-1)^{n}\mathrm{li}(\langle dN_{\phi}(e_{\phi}),\delta N_{\beta}(\omega)\rangle_{\Lambda^{1}})\big)\\
   -e_{\phi}
  \end{pmatrix},
   \\
 &
  \begin{pmatrix}
   f_{b}  \\
   e_{b}
  \end{pmatrix} 
  =
  \begin{pmatrix}
   -e_{\phi}|_{\Gamma} \\
   \tilde{e}_{\Sigma}|_{\Gamma} \end{pmatrix}
 \Bigg\},
\end{aligned}
\end{equation}
with $\tilde{e}_{\Sigma}=E(e_{\Sigma})$, is a Dirac structure. 
\end{theorem}

Before proving this theorem, we first need to check if $d\big(\mathrm{li}(f_{\phi})\big)$ is well-defined.
\begin{lemma}\label{lemma16}
  Given $e_{\omega}\in\mathring{V}^{\ast}\Lambda^{n-2}(\Omega)$, $e_{\phi}\in H^{-\frac{1}{2}}\Lambda^{n-1}(\partial\Omega)$ with $\omega\in\mathring{V}\Lambda^{2}(\Omega)$, then
  $d\big(\mathrm{li}_{\phi}(f_{\phi})\big)$ in (\ref{dirac7}) is well-defined.
\end{lemma}

\begin{proof}
  Following the same approach as we did in Lemma \ref{lemma15}, we can prove Lemma \ref{lemma16}.
\end{proof}

Next, we will give a proof of Theorem \ref{thm2}.
\begin{proof}
  
\begin{enumerate}[(1)]
  \item Using the definition of the trace-lifting operator $\mathrm{li}_{\phi}$ (\ref{tracelift2}), the Laplace solution solution operator $N_{\phi}$ (\ref{laplace}), and (\ref{eq138})
  we can rewrite (\ref{bilinear}) as
  \begin{equation}\label{bilinear2}
    \begin{aligned}
      \llangle(f^{1},e^{1}),(f^{2},e^{2})\rrangle=&\int_{\Omega}\Big(e_{\omega}^{1}\wedge f_{\omega}^{2}+e_{\omega}^{2}\wedge f_{\omega}^{1}+dN_{\phi}(e_{\phi}^{1})\wedge\ast d\big(\mathrm{li}_{\phi}(f_{\phi}^{2})\big)\\
      +&dN_{\phi}(e_{\phi}^{2})\wedge\ast d\big(\mathrm{li}_{\phi}(f_{\phi}^{1})\big)\Big)
      +\int_{\Sigma}\Big(e_{\Sigma}^{1}\wedge f_{\Sigma}^{2}+e_{\Sigma}^{2}\wedge f_{\Sigma}^{1}\Big)\\
      +&\int_{\Gamma}\Big(e_{b}^{1}\wedge f_{b}^{2}+e_{b}^{2}\wedge f_{b}^{1}\Big).
    \end{aligned}
  \end{equation}
  
  \item We check that $\int_{\Omega}e_{\omega}^{1}\wedge d\ast\big((\ast de_{\omega}^{2})\wedge\ast\omega\big)$ is skew-symmetric in $e_{\omega}^{1}$ and $e_{\omega}^{2}$. For arbitrary $e_{\omega}^{i}\in\mathring{V}^{\ast}\Lambda^{n-2}(\Omega),\ i=1,2$, which implies the boundary condition $\mathrm{tr}(e_{\omega}^{i})=0$, by using the integration by parts formula (\ref{part}), we obtain
  \begin{equation*}
    \begin{aligned}
      \int_{\Omega}e_{\omega}^{1}\wedge d\ast\big((\ast de_{\omega}^{2})\wedge(\ast\omega)\big)
      =&(-1)^{n-1}\int_{\Omega}\delta(\ast e_{\omega}^{1})\wedge(\ast de_{\omega}^{2})\wedge(\ast\omega)\\
      =&(-1)^{n}\langle de_{\omega}^{2},(\ast de_{\omega}^{1})\wedge(\ast\omega)\rangle_{L^{2}\Lambda^{n-1}(\Omega)}\\
      =&-\int_{\Omega}e_{\omega}^{2}\wedge d\ast\big((\ast de_{\omega}^{1})\wedge(\ast\omega)\big).
    \end{aligned}
  \end{equation*}

  \item We show that $D_{3}\subset D_{3}^{\bot}$. Let $(f^{1},e^{1})\in D_{3}$ be fixed, and consider any $(f^{2},e^{2})\in D_{3}$. From (\ref{dirac7}) and (\ref{bilinear2}), we have that
  \begin{equation}\label{eq99}
    \begin{aligned}
      &\llangle (f^{1},e^{1}),(f^{2},e^{2})\rrangle\\
      =&\int_{\Omega}\Big((-1)^{n-1}e_{\omega}^{1}\wedge d\ast\big((\ast de_{\omega}^{2})\wedge(\ast\omega)\big)+e_{\omega}^{1}\wedge d\ast\big(dN_{\phi}(e_{\phi}^{2})\wedge\ast\omega\big)\\
      &+(-1)^{n-1}e_{\omega}^{2}\wedge d\ast\big((\ast de_{\omega}^{1})\wedge(\ast\omega)\big)+e_{\omega}^{2}\wedge d\ast\big(dN_{\phi}(e_{\phi}^{1})\wedge\ast\omega\big)\\
      &+dN_{\phi}(e_{\phi}^{1})\wedge(\ast de_{\omega}^{2})\wedge(\ast\omega)+(-1)^{n-1}dN_{\phi}(e_{\phi}^{1})\wedge dN_{\phi}(e_{\phi}^{2})\wedge(\ast\omega)\\
      &+dN_{\phi}(e_{\phi}^{2})\wedge(\ast de_{\omega}^{1})\wedge(\ast\omega)+(-1)^{n-1}dN_{\phi}(e_{\phi}^{2})\wedge dN_{\phi}(e_{\phi}^{1})\wedge(\ast\omega)\\
      &+dN_{\phi}(e_{\phi}^{1})\wedge\ast[\delta N_{\beta}(\omega),dN_{\phi}(e_{\phi}^{2})]_{1}+dN_{\phi}(e_{\phi}^{2})\wedge\ast[\delta N_{\beta}(\omega),dN_{\phi}(e_{\phi}^{1})]_{1}\\
      &+dN_{\phi}(e_{\phi}^{1})\wedge\ast d\big(\mathrm{li}(\tilde{e}_{\Sigma}^{2})+(-1)^{n}\mathrm{li}(\langle dN_{\phi}(e_{\phi}^{2}),\delta N_{\beta}(\omega)\rangle_{\Lambda^{1}})\big)\\
      &+dN_{\phi}(e_{\phi}^{2})\wedge\ast d\big(\mathrm{li}(\tilde{e}_{\Sigma}^{1})+(-1)^{n}\mathrm{li}(\langle dN_{\phi}(e_{\phi}^{1}),\delta N_{\beta}(\omega)\rangle_{\Lambda^{2}})\big)\Big)\\
      &-\int_{\Sigma}\Big(e_{\Sigma}^{1}\wedge e_{\phi}^{2}+e_{\Sigma}^{2}\wedge e_{\phi}^{1}\Big)-\int_{\Gamma}\Big(\tilde{e}_{\Sigma}^{1}\wedge e_{\phi}^{2}+\tilde{e}_{\Sigma}^{2}\wedge e_{\phi}^{1}\Big).
    \end{aligned}
  \end{equation}
  Using the integration by parts formula (\ref{part}) and the condition $\mathrm{tr}(e_{\omega}^{i})=0,\ i=1,2$, since $e_{\omega}^{i}\in\mathring{V}^{\ast}\Lambda^{n-2}(\Omega)$, we have for $i,j\in\{1,2\}$ that
  \begin{equation}\label{eq66}
    \begin{aligned}
      \int_{\Omega}e_{\omega}^{i}\wedge d\ast\big(dN_{\phi}(e_{\phi}^{j})\wedge\ast\omega\big)=&(-1)^{n-1}\int_{\Omega}e_{\omega}^{i}\wedge\ast\delta\big(dN_{\phi}(e_{\phi}^{j})\wedge\ast\omega\big)\\
      =&(-1)^{n-1}\langle de_{\omega}^{i},dN_{\phi}(e_{\phi}^{j})\wedge\ast\omega)\rangle_{L^{2}\Lambda^{n-1}(\Omega)}\\
      =&-\int_{\Omega}dN_{\phi}(e_{\phi}^{j})\wedge(\ast de_{\omega}^{i})\wedge(\ast\omega).
    \end{aligned}
  \end{equation}
  Following (\ref{eq76}) and (\ref{eq77}), with $\eta=\delta N_{\beta}(\omega)$, we have
  Using (\ref{lie}) and Lemma \ref{lemma2}, we have for $i,j\in\{1,2\}$
  \begin{equation}\label{eq130}
    \begin{aligned}
      \int_{\Omega}dN_{\phi}(e_{\phi}^{i})\wedge\ast[\delta N_{\beta}(\omega),dN_{\phi}(e_{\phi}^{j})]_{1}=(-1)^{n-1}\int_{\partial\Omega}\langle dN_{\phi}(e_{\phi}^{i}),\delta N_{\beta}(\omega)\rangle_{\Lambda^{1}}\wedge e_{\phi}^{j},
    \end{aligned}
  \end{equation}
  and
  \begin{equation}\label{eq131}
    \begin{aligned}
      &\int_{\Omega}dN(e_{\phi}^{i})\wedge\ast d\Big(\mathrm{li}\big(\tilde{e}_{\Sigma}^{j}\big)+(-1)^{n}\mathrm{li}\big(\langle dN_{\phi}(e_{\phi}^{j}),\delta N_{\beta}(\omega)\rangle_{\Lambda^{1}}\big)\Big)\\
      =&\int_{\partial\Omega}\tilde{e}_{\Sigma}^{j}\wedge e_{\phi}^{i}+(-1)^{n}\int_{\partial\Omega}\langle dN_{\phi}(e_{\phi}^{j}),\delta N_{\beta}(\omega)\rangle_{\Lambda^{1}}\wedge e_{\phi}^{i}.
    \end{aligned}
  \end{equation}
  By the skew-symmetry of term $\int_{\Omega}e_{\omega}^{1}\wedge d\ast\big((\ast e_{\omega}^{2})\wedge\ast\omega\big)$, together with (\ref{eq66}), (\ref{eq130}), (\ref{eq131}),
  we obtain that (\ref{eq99}) is equal to
  \begin{equation*}
    \begin{aligned}
      \llangle (f^{1},e^{1}),(f^{2},e^{2})\rrangle
      =&\int_{\partial\Omega}\Big(e_{\phi}^{1}\wedge\tilde{e}_{\Sigma}^{2}+e_{\phi}^{2}\wedge\tilde{e}_{\Sigma}^{1}\Big)-\int_{\partial\Omega}\Big(e_{\phi}^{1}\wedge\tilde{e}_{\Sigma}^{2}+e_{\phi}^{2}\wedge\tilde{e}_{\Sigma}^{1}\Big)\\
      =&0.
    \end{aligned}
  \end{equation*}
  Therefore $(f^{1},e^{1})\in D_{3}^{\bot}$, which implies that $D_{3}\subset D_{3}^{\bot}$.

  \item Finally, we show that $D_{3}^{\bot}\subset D_{3}$. Let $(f^{1},e^{1})\in D_{3}^{\bot}\subset\mathcal{F}_{3}\times\mathcal{E}_{3}$. Then 
  \begin{equation*}
    \llangle(f^{1},e^{1}),(f^{2},e^{2})\rrangle=0,\quad\forall(f^{2},e^{2})\in D_{3},
  \end{equation*}
  which equals 
  \begin{equation}\label{eq48}
    \begin{aligned}
      &\int_{\Omega}\Big((-1)^{n-1}e_{\omega}^{1}\wedge d\ast\big((\ast de_{\omega}^{2})\wedge\ast\omega\big)+e_{\omega}^{1}\wedge d\ast\big(dN_{\phi}(e_{\phi}^{2})\wedge\ast\omega\big)+e_{\omega}^{2}\wedge f_{\omega}^{1}\\
      &+dN_{\phi}(e_{\phi}^{1})\wedge(\ast de_{\omega}^{2})\wedge(\ast\omega)+(-1)^{n-1}dN_{\phi}(e_{\phi}^{1})\wedge dN_{\phi}(e_{\phi}^{2})\wedge(\ast\omega)\\
      &+dN_{\phi}(e_{\phi}^{2})\wedge\ast d(\mathrm{li}_{\phi}(f_{\phi}^{1}))+dN_{\phi}(e_{\phi}^{1})\wedge\ast[\delta N_{\beta}(\omega),dN_{\phi}(e_{\phi}^{2})]_{1}\\
      &+dN_{\phi}(e_{\phi}^{1})\wedge\ast d\big(\mathrm{li}(\tilde{e}_{\Sigma}^{2})+(-1)^{n}\mathrm{li}(\langle dN_{\phi}(e_{\phi}^{2}),\delta N_{\beta}(\omega)\rangle_{\Lambda^{1}})\big)\Big)\\
      &+\int_{\Sigma}\Big(-e_{\Sigma}^{1}\wedge e_{\phi}^{2}+e_{\Sigma}^{2}\wedge f_{\Sigma}^{1}\Big)
      +\int_{\Gamma}\Big(-e_{b}^{1}\wedge e_{\phi}^{2}+\tilde{e}_{\Sigma}^{2}\wedge f_{b}^{1}\Big)=0.
    \end{aligned}
  \end{equation}
  Taking $\tilde{e}_{\Sigma}^{2}=E(e_{\Sigma}^{2})\in H_{00}^{\frac{1}{2}}\Lambda^{0}(\partial\Omega)$, which implies that $\tilde{e}_{\Sigma}^{2}=0$ at $\Gamma$, and $e_{\phi}^{2}\in H^{-\frac{1}{2}}\Lambda^{n-1}(\partial\Omega)$ with $e_{\phi}^{2}=0$ at $\Gamma$,
  then from the skew-symmetry of $\int_{\Omega}e_{\omega}^{1}\wedge d\ast\big((\ast e_{\omega}^{2})\wedge\ast\omega\big)$, together with (\ref{eq66}), (\ref{eq130}), and (\ref{eq131}), then (\ref{eq48}) can be rewritten as
  \begin{equation*}
    \begin{aligned}
      &\int_{\Omega}e_{\omega}^{2}\wedge\Big( f_{\omega}^{1}+(-1)^{n}d\ast\big((\ast de_{\omega}^{1})\wedge(\ast\omega)\big)-d\ast\big(dN_{\phi}(e_{\phi}^{1})\wedge\ast\omega\big)\Big)\\
      +&\int_{\Omega}dN_{\phi}(e_{\phi}^{2})\wedge\Big(\ast d\big(\mathrm{li}_{\phi}(f_{\phi}^{1})\big)-(\ast de_{\omega}^{1})\wedge(\ast\omega)+(-1)^{n}dN_{\phi}(e_{\phi}^{1})\wedge(\ast\omega)\Big)\\
      +&\int_{\partial\Omega}(-1)^{n-1}\langle dN_{\phi}(e_{\phi}^{1}),\delta N_{\beta}(\omega)\rangle_{\Lambda^{1}}\wedge e_{\phi}^{2}+\int_{\partial\Omega}\tilde{e}_{\Sigma}^{2}\wedge e_{\phi}^{1}\\
      +&\int_{\partial\Omega}(-1)^{n}\langle dN_{\phi}(e_{\phi}^{2}),\delta N_{\beta}(\omega)\rangle_{\Lambda^{1}}\wedge e_{\phi}^{1}+\int_{\Sigma}\Big(-e_{\Sigma}^{1}\wedge e_{\phi}^{2}+e_{\Sigma}^{2}\wedge f_{\Sigma}^{1}\Big)=0.
    \end{aligned}
  \end{equation*}
  Again applying (\ref{eq130}) and (\ref{eq131}), with $e_{\phi}^{2}=\tilde{e}_{\Sigma}^{2}=0$ at $\Gamma$, we further have
  \begin{equation*}
    \begin{aligned}
      &\int_{\Omega}e_{\omega}^{2}\wedge\Big( f_{\omega}^{1}+(-1)^{n}d\ast\big((\ast de_{\omega}^{1})\wedge(\ast\omega)\big)-d\ast\big(dN_{\phi}(e_{\phi}^{1})\wedge\ast\omega\big)\Big)\\
      +&\int_{\Omega}dN_{\phi}(e_{\phi}^{2})\wedge\Big(\ast d\big(\mathrm{li}_{\phi}(f_{\phi}^{1})\big)-(\ast de_{\omega}^{1})\wedge(\ast\omega)+(-1)^{n}dN_{\phi}(e_{\phi}^{1})\wedge(\ast\omega)\\
      -&\ast[\delta N_{\beta}(\omega),dN_{\phi}(e_{\phi}^{1})]_{1}-\ast d\big(\mathrm{li}(\tilde{e}_{\Sigma}^{1})+(-1)^{n}\mathrm{li}(\langle dN_{\phi}(e_{\phi}^{1}),\delta N_{\beta}(\omega)\rangle_{\Lambda^{1}})\big)\Big)\\
      +&\int_{\Sigma}\Big(e_{\Sigma}^{2}\wedge e_{\phi}^{1}+e_{\Sigma}^{2}\wedge f_{\Sigma}^{1}\Big)=0.
    \end{aligned}
  \end{equation*}

  This integral can be zero only if
  \begin{equation}\label{eq49}
    \begin{cases}
      &f_{\omega}^{1}=d\ast\Big(\big((-1)^{n-1}\ast de_{\omega}^{1}+dN_{\phi}(e_{\phi}^{1}) \big)\wedge\ast\omega\Big),\\
      &d\big(\mathrm{li}_{\phi}(f_{\phi}^{1})\big)= \ast\Big(\big((-1)^{n-1}\ast de_{\omega}^{1}+dN_{\phi}(e_{\phi}^{1})\big)\wedge\ast\omega\Big)+[\delta N_{\beta}(\omega),dN_{\phi}(e_{\phi}^{1})]_{1}\\
      &\qquad\qquad\qquad d\big(\mathrm{li}(\tilde{e}_{\Sigma}^{1})+(-1)^{n}\mathrm{li}(\langle dN_{\phi}(e_{\phi}^{1}),\delta N_{\beta}(\omega)\rangle_{\Lambda^{1}})\big)\\
      &f_{\Sigma}^{1}|_{\Sigma}=-e_{\phi}^{1}.
    \end{cases}
  \end{equation}

  Substituting (\ref{eq49}) into (\ref{eq48}) yields
  \begin{equation*}
    \begin{aligned}
      \int_{\partial\Omega}\Big(\tilde{e}_{\Sigma}^{1}\wedge e_{\phi}^{2}+\tilde{e}_{\Sigma}^{2}\wedge e_{\phi}^{1}\Big)-\int_{\Sigma}\Big(e_{\Sigma}^{1}\wedge e_{\phi}^{2}+e_{\Sigma}^{2}\wedge e_{\phi}^{1}\Big)+\int_{\Gamma}\Big(f_{b}^{1}\wedge \tilde{e}_{\Sigma}^{2}-e_{\phi}^{2}\wedge e_{b}^{1}\Big)=0.
    \end{aligned}
  \end{equation*}
  Hence,
  \begin{equation*}
    \begin{aligned}
    f_{b}^{1}=-e_{\phi}^{1}|_{\Gamma},\quad e_{b}^{1}=\tilde{e}_{\Sigma}^{1}|_{\Gamma},
    \end{aligned}
  \end{equation*}
  and $D_{3}^{\bot}\subset D_{3}$, which together with $D_{3}\subset D_{3}^{\bot}$ implies that $D_{3}$ is a Dirac structure.
\end{enumerate}

\end{proof}

\begin{corollary}
  Given an oriented and simply connected $n$-dimensional manifold $\Omega$ with Lipschitz continuous boundary $\partial\Omega$.
  Given the essential boundary condition (\ref{vor2}e) and the port variable $f_{b}=-g$. The distributed-parameter port-Hamiltonian system for the incompressible Euler equations with a free surface (\ref{vor1})--(\ref{vor3}) with respect to the variables
  $(\omega,\phi_{\partial},\Sigma)\in\mathring{V}\Lambda^{2}(\Omega)\times H^{\frac{1}{2}}\Lambda^{0}(\partial\Omega)\times H^{-\frac{1}{2}}\Lambda^{n-1}(\Sigma)$, Dirac structure (\ref{dirac7}), and Hamiltonian (\ref{hami5}), is given by
  \begin{equation}\label{coro2}
    \begin{aligned}    
      &
    \begin{pmatrix}
      -\omega_{t} \\
      -d\phi_{t}\\
      -\Sigma_{t}
     \end{pmatrix} 
     =
     \begin{pmatrix}
      (-1)^{n-1}d\ast\big(\ast d(\cdot)\wedge\ast\omega\big)& d\ast\big(dN_{\phi}(\cdot)\wedge\ast\omega\big)  & 0 \\
      (-1)^{n-1}\ast\big(\ast d(\cdot)\wedge\ast\omega\big)& B_{22} & d\big(\mathrm{li}(E(\cdot))\big)\\
      0 & -1 & 0
      \end{pmatrix}
      \begin{pmatrix}
        \frac{\delta\bar{H}}{\delta\omega} \\
        \frac{\delta\bar{H}}{\delta\phi_{\partial}}\\
        \frac{\delta\bar{H}}{\delta\Sigma}
       \end{pmatrix},
      \end{aligned}
    \end{equation}
    where $B_{22}=\ast\big(dN_{\phi}(\cdot)\wedge\ast\omega\big)+(-1)^{n}d\big(\mathrm{li}(\langle dN_{\phi}(\cdot),\delta N_{\beta}(\omega)\rangle_{\Lambda^{1}})\big)+[\delta N_{\beta}(\omega),dN_{\phi}(\cdot)]_{1}$,
    with port variables $f_{b},e_{b}$ defined as
    \begin{equation}\label{coro6}
      \begin{aligned} 
     \begin{pmatrix}
      f_{b}  \\
      e_{b}
     \end{pmatrix} 
     =
     \begin{pmatrix}
      -1\quad & 0 \\
      0\quad & E(\cdot)
     \end{pmatrix}
     \begin{pmatrix}
      \frac{\delta\bar{H}}{\delta\phi_{\partial}} \\
      \frac{\delta\bar{H}}{\delta\Sigma}
     \end{pmatrix}.
   \end{aligned}
  \end{equation}
\end{corollary}

\begin{proof}
  After defining the flow and energy variables in (\ref{dirac7}), respectively, as
  \begin{equation*}
    f_{\omega}=-\omega_{t},\quad f_{\phi}=-(\phi_{\partial})_{t},\quad f_{\Sigma}=-\Sigma_{t},
  \end{equation*}
  and
  \begin{equation*}
    e_{\omega}=\frac{\delta\bar{H}}{\delta\omega},\quad e_{\phi}=\frac{\delta\bar{H}}{\delta\phi_{\partial}},\quad e_{\Sigma}=\frac{\delta\tilde{H}}{\delta\Sigma},
  \end{equation*}
  we immediately obtain (\ref{coro2}). Using the functional derivatives of the Hamiltonian (\ref{eq120}), and the essential boundary condition (\ref{vor2}e) we obtain (\ref{vor1}a), (\ref{vor2}a), (\ref{vor2}c).
  The port variable $f_{b}=\frac{\delta\bar{H}}{\delta\phi_{\partial}}=-g$, gives the essential boundary condition (\ref{vor2}d) at $\Gamma$. Finally, (\ref{vor1}b), (\ref{vor1}c) follow from $\omega\in\mathring{V}\Lambda^{2}(\Omega)$ and (\ref{vor2}b) follows from $v\in P^{\ast}\Lambda^{1}(\Omega)$ together with (\ref{vor3}).
\end{proof}

From the bilinear form (\ref{bilinear2}), we can define the following bilinear form 
\begin{equation*}
  [e_{1},e_{2}]_{D}=\int_{\Omega}\Big(e_{\omega}^{1}\wedge f_{\omega}^{2}+dN(e_{\phi}^{1})\wedge\ast d\big(\mathrm{li}_{\phi}(f_{\phi}^{2})\big)\Big)+\int_{\Sigma}e_{\Sigma}^{1}\wedge f_{\Sigma}^{2}+\int_{\Gamma}e_{b}^{1}\wedge f_{b}^{2}.
\end{equation*}

For $\frac{\delta\bar{\mathcal{F}}}{\delta\omega},\frac{\delta\bar{\mathcal{G}}}{\delta\omega}\in\mathring{V}^{\ast}\Lambda^{n-2}(\Omega)$, $\frac{\delta\bar{\mathcal{F}}}{\delta\phi_{\partial}},\frac{\delta\bar{\mathcal{G}}}{\delta\phi_{\partial}}\in H^{-\frac{1}{2}}\Lambda^{n-1}(\partial\Omega)$ and $\frac{\delta\bar{\mathcal{F}}}{\delta\Sigma},\frac{\delta\bar{\mathcal{G}}}{\delta\Sigma}\in H^{\frac{1}{2}}\Lambda^{0}(\Sigma)$, the Poisson bracket $\{\cdot,\cdot\}_{D}(\omega,\phi_{\partial},\Sigma)$ associated to the Dirac structure (\ref{dirac7}) is defined as
\begin{equation}\label{bracket7}
  \begin{aligned}
    \{\bar{\mathcal{F}},\bar{\mathcal{G}}\}_{D}(\omega,\phi_{\partial},\Sigma)=&\int_{\Omega}\frac{\delta\bar{\mathcal{G}}}{\delta\omega}\wedge d\ast\Big(\big((-1)^{n-1}\ast d(\frac{\delta\bar{\mathcal{F}}}{\delta\omega})+dN_{\phi}(\frac{\delta\bar{\mathcal{F}}}{\delta\phi_{\partial}})\big)\wedge\ast\omega\Big)\\
    &+\int_{\Omega}dN_{\phi}(\frac{\delta\bar{\mathcal{G}}}{\delta\phi_{\partial}})\wedge\Big(\big(\ast d(\frac{\delta\bar{\mathcal{F}}}{\delta\omega})+(-1)^{n-1}dN_{\phi}(\frac{\delta\bar{\mathcal{F}}}{\delta\phi_{\partial}})\big)\wedge\ast\omega\\
    &+\ast[\delta N_{\beta}(\omega),dN_{\phi}(\frac{\delta\bar{\mathcal{F}}}{\delta\phi_{\partial}})]_{1}\\
    &+\ast d\big(\mathrm{li}(E(\frac{\delta\bar{\mathcal{F}}}{\delta\Sigma}))+(-1)^{n}\mathrm{li}(\langle dN_{\phi}(\frac{\delta\mathcal{F}}{\delta\phi_{\partial}}),\delta N_{\beta}(\omega)\rangle_{\Lambda^{1}})\big)\Big)\\
    &-\int_{\Sigma}\frac{\delta\bar{\mathcal{G}}}{\delta\Sigma}\wedge\frac{\delta\bar{\mathcal{F}}}{\delta\phi_{\partial}}-\int_{\Gamma}E(\frac{\delta\bar{\mathcal{G}}}{\delta\Sigma})\wedge\frac{\delta\bar{\mathcal{F}}}{\delta\phi_{\partial}}.
  \end{aligned}
\end{equation}

\begin{remark}
  Using (\ref{eq130}), (\ref{eq131}) and Stokes theorem, (\ref{bracket7}) can be represented as
  \begin{equation}
    \begin{aligned}
      &\{\bar{\mathcal{F}},\bar{\mathcal{G}}\}_{D}(\omega,\phi_{\partial},\Sigma)\\
      =&(-1)^{n-1}\int_{\Omega}(\ast\omega)\wedge\Big((-1)^{n-1}\ast d\frac{\delta\bar{\mathcal{G}}}{\delta\omega}+dN_{\phi}(\frac{\delta\bar{\mathcal{G}}}{\delta\phi_{\partial}})\big)
      \wedge\big((-1)^{n-1}\ast d\frac{\delta\bar{\mathcal{F}}}{\delta\omega}+dN_{\phi}(\frac{\delta\bar{\mathcal{F}}}{\delta\phi_{\partial}})\Big)\\
      &+\int_{\partial\Omega}\Big(\frac{\delta\bar{\mathcal{F}}}{\delta\Sigma}+(-1)^{n}\langle dN_{\phi}(\frac{\delta\bar{\mathcal{F}}}{\delta\phi_{\partial}}),\delta N_{\beta}(\omega)\rangle_{\Lambda^{1}} \Big)\wedge\frac{\delta\bar{\mathcal{G}}}{\delta\phi_{\partial}}\\
      &-\int_{\partial\Omega}\Big(\frac{\delta\bar{\mathcal{G}}}{\delta\Sigma}+(-1)^{n}\langle dN_{\phi}(\frac{\delta\bar{\mathcal{G}}}{\delta\phi_{\partial}}),\delta N_{\beta}(\omega)\rangle_{\Lambda^{1}} \Big)\wedge\frac{\delta\bar{\mathcal{F}}}{\delta\phi_{\partial}}.
    \end{aligned}
  \end{equation}
\end{remark}

\begin{remark}
  The bracket $\{\cdot,\cdot\}_{D}(\omega,\phi_{\partial},\Sigma)$ (\ref{bracket7}) is linear, anti-symmetric and satisfies the Jacobi identity, which follows directly from the fact that $\{\cdot,\cdot\}_{D}(\omega,\phi_{\partial},\Sigma)$ has the same structure as the Poisson bracket $\{\cdot,\cdot\}(\omega,\phi_{\partial},\Sigma)$ stated in Lemma \ref{lemma9}. The bracket $\{\cdot,\cdot\}_{D}(\omega,\phi_{\partial},\Sigma)$ is thus a Poisson bracket and the system generated by the Dirac structure $D_{3}$ is a Poisson structure.
\end{remark}

\begin{theorem}
Let $\Omega\subset\mathbb{R}^{n},n\in\{2,3\}$ be an $n$-dimensional oriented connected manifold with Lipschitz continuous boundary $\partial\Omega=\Sigma\cup\Gamma$ with free surface $\Sigma$ and fixed boundary $\Gamma$, $\Sigma\cap\Gamma=\emptyset$. 
For any functionals $\bar{\mathcal{F}}(\omega,\phi_{\partial},\Sigma):\mathring{V}\Lambda^{2}(\Omega)\times H^{\frac{1}{2}}\Lambda^{0}(\partial\Omega)\times H^{-\frac{1}{2}}\Lambda^{n-1}(\Sigma)\to\mathbb{R}$, the port-Hamiltonian formulation of the incompressible Euler equations with a free surface (\ref{coro4}) can be expressed as
\begin{equation}\label{eq132}
  \dot{\bar{\mathcal{F}}}(\omega,\phi_{\partial},\Sigma)=\{\bar{\mathcal{F}},\bar{H}\}_{D}(\omega,\phi_{\partial},\Sigma)-\int_{\Gamma}E(\frac{\delta\bar{\mathcal{F}}}{\delta\Sigma})\wedge\ast\boldsymbol{n}(d\phi),
\end{equation}
with the Hamiltonian $\bar{H}$ given by (\ref{hami7}).
\end{theorem}

\begin{proof}
  Substituting $\bar{H}$ (\ref{hami7}) into (\ref{bracket7}), we can obtain the representation of $\{\bar{H},\bar{\mathcal{F}}\}_{D}$.
  With the relation $\{\bar{\mathcal{F}},\bar{H}\}_{D}=-\{\bar{H},\bar{\mathcal{F}}\}_{D}$, together with the functional chain rule and (\ref{coro2}), (\ref{eq132}) holds.
\end{proof}

\begin{corollary}
  Given the Hamiltonian functional $\bar{H}:\mathring{V}\Lambda^{2}(\Omega)\times H^{\frac{1}{2}}\Lambda^{0}(\partial\Omega)\times H^{-\frac{1}{2}}\Lambda^{n-1}(\Sigma)$ $\to\mathbb{R}$ stated in (\ref{hami7}), 
  and the essential boundary conditions (\ref{vor2}d), (\ref{vor2}e), the rate of change the Hamiltonian for the incompressible Euler equations with a free surface in the $(\omega,\phi_{\partial},\Sigma)$ variables (\ref{vor1})--(\ref{vor3}) is equal to
  \begin{equation*}
    \dot{\bar{H}}=-\int_{\Gamma}\mathrm{tr}(h)\wedge g.
  \end{equation*}
\end{corollary}

\section{Acknowledgements}
The research of Xiaoyu Cheng was funded by a fellowship from China
Scholarship Council (No.201906340056). The research of J.J.W. van der Vegt was 
partially supported by the University of Science and Technology of China, Hefei, Anhui, China. The research of Yan Xu was supported by NSFC grant No. 12071455.

\bibliographystyle{plain}
\bibliography{Reference.bib } 

\end{document}